\documentclass[11pt,reqno]{amsart}
\usepackage[utf8]{inputenc}

\usepackage{hyperref}
\usepackage{latexsym}
\usepackage{amssymb,amsmath}
\usepackage{amsthm}
\usepackage{url}
\usepackage{bbm}
\usepackage[vcentermath]{youngtab}
\usepackage{booktabs}
\usepackage{enumerate}
\usepackage[shortlabels]{enumitem}

\usepackage{graphicx}
\usepackage[font=small]{subcaption} % 'caption' package is loaded automatically    
\usepackage[usenames,dvipsnames]{color}
\usepackage{lastpage}
\usepackage{extarrows}
\usepackage{outlines}
\usepackage[arrow,matrix,curve,cmtip,ps]{xy}
\usepackage{varwidth}
\usepackage[font=itshape]{quoting} %font=itshape italicises the indented text
\usepackage{float}
\usepackage{array}
\usepackage{mathrsfs} % to get \mathscr (capital letters only)
\usepackage{bm} % to get bold -and- italic script in math mode
\usepackage{arydshln} % to get dashed lines in table
\usepackage{multicol,multirow}

\usepackage{tikz}
\usetikzlibrary{matrix,chains,shapes,arrows,scopes,positioning}
\usetikzlibrary{calc,decorations.pathmorphing,shapes}
\usetikzlibrary{patterns}
\tikzstyle{empty}=[circle,draw=black!80,thick]
\tikzstyle{emptyn}=[circle,draw=black!80,fill=white,scale=0.5] 
\tikzstyle{nero}=[circle,draw=black!80,fill=black!80,thick] 
 
\newcommand{\triv}{\mathbbm{1}}
\newcommand{\N}{\mathbb{N}}

\newcommand{\cB}{\mathcal{B}}
\newcommand{\cC}{\mathcal{C}}
\newcommand{\cD}{\mathcal{D}}
\newcommand{\cLR}{\mathcal{LR}}
\newcommand{\cP}{\mathcal{P}}
\newcommand{\cQ}{\mathcal{Q}}
\newcommand{\cR}{\mathcal{R}}
\newcommand{\cV}{\mathcal{V}}
\newcommand{\cW}{\mathcal{W}}
\newcommand{\cX}{\mathcal{X}}
\newcommand{\sfs}{\mathsf{s}}
\newcommand{\sfx}{\mathsf{x}}
\newcommand{\sfy}{\mathsf{y}}

\newcommand{\Lin}{\operatorname{Lin}}
\newcommand{\Irr}{{\operatorname{Irr}}}

\newcommand{\Char}{{\operatorname{Char}}}
\newcommand{\down}{\big\downarrow}
\newcommand{\up}{\big\uparrow}
\newcommand{\tworow}[2]{\mathsf{t}_{#1}[{#2}]}
\newcommand{\hook}[2]{\mathsf{h}_{#1}[{#2}]}
\newtheorem{theorem}{Theorem}[section]
\newtheorem{lemma}[theorem]{Lemma}

\newtheorem{corollary}[theorem]{Corollary}
\newtheorem{proposition}[theorem]{Proposition}
\newtheorem{definition}[theorem]{Definition}
\newtheorem{notation}[theorem]{Notation}

\theoremstyle{definition}
\newtheorem{example}[theorem]{Example}

\theoremstyle{remark}
\newtheorem{remark}[theorem]{Remark}

\numberwithin{equation}{section}

\allowdisplaybreaks % allows e.g. align environments to break over multiple pages
\begin{document}

\title{Positivity of Sylow branching coefficients for symmetric groups}

%\date{\today}

\author{Stacey Law}
\address[S. Law]{Department of Pure Mathematics and Mathematical Statistics, University of Cambridge, Cambridge CB3 0WB, UK}
\email{swcl2@cam.ac.uk}

\thanks{}
 
\begin{abstract}
	In this article we investigate the positivity of Sylow branching coefficients for symmetric groups when $p=3$. In particular, we complete the discussion begun by Giannelli and the author in \cite[\textit{J.~Algebra} \textbf{506} (2018), 409--428]{GL1} and developed in \cite[\textit{J.~London~Math.~Soc.} (2) \textbf{103} (2021), 697--728]{GL2} concerning the case of odd primes.
\end{abstract}



\maketitle

%%%%%%%%%%%%%%%%%%%%%%%%%%%%%%%%%%%%%%%%%%%%%%%%%%%%%%%%%%%%%%%%%%%%%%%%%
\section{Introduction}\label{sec:intro}

Let $p$ be a prime. Studying the relationship between the character theory of a finite group $G$ and that of a Sylow $p$-subgroup $P$ of $G$, and how character theoretic information determines the structural properties of the groups, are key areas of research in the representation theory of finite groups. 
For instance, a classical theorem of It\^{o} and Michler \cite{Ito,Michler} shows that $P$ is abelian and normal in $G$ if and only if $p\nmid\chi(1)$ for all $\chi\in\Irr(G)$, the set of irreducible complex characters of $G$. 
Malle and Navarro later characterised the two properties of abelian and normality separately in \cite{MN21,MN12} respectively: $P$ is abelian if and only if $p\nmid\chi(1)$ for all $\chi\in\Irr(G)$ belonging to the principal $p$-block of $G$, as predicted by Brauer's famous Height Zero Conjecture \cite{Brauer}; and $P\trianglelefteq G$ if and only if $p\nmid\chi(1)$ for all irreducible constituents $\chi$ of $\triv_P\up^G$, where $\triv$ denotes the trivial character of $P$. 

The latter may be restated in terms of so-called Sylow branching coefficients, which concern the restriction of irreducible characters of a finite group to their Sylow subgroups and their decomposition into irreducible constituents.
%Let $n\in\N$, $p$ be prime and $P_n$ be a Sylow $p$-subgroup of the symmetric group $S_n$. For irreducible characters $\chi$ of $S_n$ and $\phi$ of $P_n$, the \textit{Sylow branching coefficient} corresponding to $\chi$ and $\phi$ is the multiplicity
For $\chi\in\Irr(G)$ and $\phi\in\Irr(P)$, the \textit{Sylow branching coefficient} corresponding to $\chi$ and $\phi$ is the multiplicity
\[ Z^\chi_\phi := \langle \chi\down_{P_n}, \phi\rangle. \] 
Thus the above may be restated in terms of the \emph{positivity} of Sylow branching coefficients as: $P\trianglelefteq G$ if and only if $p\nmid\chi(1)$ for all $\chi\in\Irr(G)$ such that $Z^\chi_{\triv_P}>0$. We remark that the final condition was recently sharpened to $p\nmid Z^\chi_{\triv_P}$ in \cite{GLLV} and that both results utilised the classification of finite simple groups, with the alternating groups being one of the most involved cases to study.

Along these lines of investigation, there have been a number of recent results on Sylow branching coefficients for the closely related symmetric groups. Let $P_n$ be a Sylow $p$-subgroup of the symmetric group $S_n$. %For example, it was shown in \cite{GN} that if $p\mid\chi(1)$ for $\chi\in\Irr(S_n)$, then $\chi\down_{P_n}$ has at least $p$ different linear constituents. 
%Looking specifically at the trivial linear character, 
The positivity of $Z^\chi_{\triv_{P_n}}$ in the case of $p$ odd was completely described in \cite{GL1}, and the generalisation to arbitrary linear characters $\phi$ of $P_n$ when $p\ge 5$ was investigated in \cite{GL2}, where it was noted that the situation for $p=3$ differed significantly. In this introduction we will allude to some of these differences between the cases of $p=3$ and $p\ge 5$; these differences will be discussed in detail later in the article once we have introduced the necessary notation.

The main aim of this article is to continue the study, begun in \cite{GL2}, of $Z^\chi_\phi$ when $p$ is an odd prime for all $\phi\in\Lin(P_n)$, the set of all linear characters of $P_n$. In particular, we extend the main tools of \cite{GL2} and introduce new combinatorial machinery to analyse for all such $\phi$ the set $\Omega(\phi):=\{\chi\in\Irr(S_n)\mid Z^\chi_{\phi}>0\}$ %consisting of all those irreducible characters $\chi$ of $S_n$ such that $Z^\chi_\phi>0$, 
when $p=3$.
The set $\Irr(S_n)$ is naturally in bijection with the set $\cP(n)$ of partitions of $n$. For $\lambda\in\cP(n)$, let $\chi^\lambda\in\Irr(S_n)$ be its corresponding irreducible character. Abbreviating $Z^{\chi^\lambda}_{\phi}$ to $Z^{\lambda}_{\phi}$, we may therefore view $\Omega(\phi)$ as a subset of $\cP(n)$ via
\[ \Omega(\phi) = \{\lambda\in\cP(n)\ |\ Z^{\lambda}_{\phi}>0 \}. \]
%where we have abbreviated $Z^{\chi^\lambda}_{\phi}$ to $Z^{\lambda}_{\phi}$ for simplicity.

Our first main result is \textbf{Theorem~\ref{thm:a}}, which explicitly determines $\Omega(\phi)$ when $p=3$ for $\phi$ belonging to a large subset of $\Lin(P_n)$ called the \textit{quasi-trivial} characters (see Definition~\ref{def:quasi-trivial}). This generalises \cite[Theorem A]{GL1} and completes the description of $\Omega(\phi)$ for all quasi-trivial characters when $p$ is odd, since the case of $p\ge 5$ is described in \cite[Theorem 2.9]{GL2}. While those forms of $\Omega(\phi)$ arising when $p\ge 5$ also arise in the case of $p=3$, in the latter case there turn out to be many additional forms of $\Omega(\phi)$ which are attained by infinite subfamilies of quasi-trivial characters of $P_n$, and which therefore could not be captured by the analysis used to treat $p\ge 5$.

The precise statement of Theorem~\ref{thm:a} will be presented in Section~\ref{sec: sylow prelims}, since it depends on a parametrisation of the linear characters of $P_n$, which is recorded in detail in the same section.
While we defer the full technical statement, the key idea is that much of $\Omega(\phi)$ can be described using sets of the following kind: for any $n,t\in\N$, we define
\[ \cB_n(t) :=\{\lambda\in\cP(n) \mid \lambda_1\le t,\ l(\lambda)\le t\}. \]
In other words, $\cB_n(t)$ is the set of partitions of $n$ whose Young diagram fits inside a $t\times t$ square grid. We also define integers $m(\phi)$ and $M(\phi)$ in order to give tight bounds on the set $\Omega(\phi)$ via $\cB_n\big(m(\phi)\big) \subseteq \Omega(\phi) \subseteq \cB_n\big(M(\phi)\big)$. In other words,
\[ m(\phi):=\max\{t\in\N\mid\cB_n(t)\subseteq\Omega(\phi)\}	\quad\text{and}\quad M(\phi):=\min\{t\in\N\mid\Omega(\phi)\subseteq\cB_n(t)\}. \]
Our second main result is the explicit determination of the values $m(\phi)$ and $M(\phi)$ for \textit{all} $\phi\in\Lin(P_n)$ when $p=3$. This is a much more intricate description than its $p\ge 5$ counterpart \cite[Theorem 2.11]{GL2}: as above, the number of general forms that occur when $p=3$ is notably greater than when $p\ge 5$. We state our results on the values of $m(\phi)$ and $M(\phi)$ for all $\phi\in\Lin(P_n)$ in \textbf{Theorem~\ref{thm:b-primepower}} (when $n$ is a power of 3) and \textbf{Theorem~\ref{thm:b-general}} (for general $n\in\N$), from where we will also see that these bounds almost fully determine $\Omega(\phi)$ as $M(\phi)-m(\phi)$ turns out to be very small.

Finally, when $p=3$ we prove in \textbf{Theorem~\ref{thm:c}} that if $\Omega_n:=\bigcap_{\phi\in\Lin(P_n)}\Omega(\phi)$, then
\[ \cB_n(\tfrac{2n}{9})\subseteq\Omega_n. \]
%(In fact, as we will observe in Remark~\ref{rem:2/9}, for all sufficiently small $\epsilon>0$ and sufficiently large $n$, $\cB_n(\tfrac{n}{2+\epsilon})\subseteq\Omega_n$.)
From this we can immediately deduce that $\underset{n\to\infty}{\lim} \frac{|\Omega_n|}{|\cP(n)|} = 1$, completing the analogous set of results for all primes $p$ (for $p\ge 5$, see \cite[Theorem 5.8]{GL2}; for $p=2$, see \cite[Theorem C]{LO}).

\begin{remark}
The case of $p\ge 5$ was studied in \cite{GL2}, but the proofs of our main results Theorems~\ref{thm:a}, \ref{thm:b-primepower} and \ref{thm:b-general} for $p=3$ have some substantial technical differences compared to that for $p\ge 5$, as we will see throughout the article. Nevertheless, $\Omega(\phi)$ is always closed under conjugation of partitions when $p$ is an odd prime, since in this case $P_n$ is contained in the alternating subgroup of $S_n$.

When $p=2$, the situation is completely different; for one, $\Omega(\phi)$ is no longer closed under conjugation in general. While (even) the characterisation of $\Omega(\triv_{P_n})$ is still open when $p=2$, different techniques involving plethysm coefficients were used recently to prove new results on $Z^\lambda_{\triv_{P_n}}$ for $p=2$, including an analogue of Theorem~\ref{thm:c} for $p=2$ in \cite[Theorem C]{LO}, which resolves questions of \cite{GL1,GL2}.\hfill$\lozenge$
\end{remark}

This article is structured as follows.
In Section~\ref{sec:prelims} we recall background on the representation theory of symmetric groups and their Sylow $p$-subgroups, and set up the relevant notation in order to state our main results, Theorems~\ref{thm:a}, \ref{thm:b-primepower} and~\ref{thm:b-general}. In particular, in Section~\ref{sec:sn-lr} we collate a number of combinatorial results on Littlewood--Richardson coefficients from \cite{GL2} which we will build upon in order to tackle the case of $p=3$. In Section~\ref{sec:LR}, we further develop the combinatorics of these coefficients.

Since the structure of Sylow subgroups of symmetric groups depends crucially on the case of prime powers, in Section~\ref{sec:prime-power} we consider when $n$ is a power of 3, proving Theorem~\ref{thm:a} in this case (i.e.~Theorem~\ref{thm:a-prime power}), as well as Theorem~\ref{thm:b-primepower}. For the case of general $n$ of our main results, the proof of Theorem~\ref{thm:a} is completed in Section~\ref{sec:a-full}, and Theorem~\ref{thm:b-general} is proved in Section~\ref{sec:b-full}. Finally, we conclude this article by deducing Theorem~\ref{thm:c}.

\bigskip

%%%%%%%%%%%%%%%%%%%%%%%%%%%%%%%%%%%%%%%%%%%%%%%%%%%%%%%%%%%%%%%%%%%%%%%%%
\section{Preliminaries}\label{sec:prelims}
Let $p$ be a fixed prime and let $P_n$ denote a Sylow $p$-subgroup of $S_n$. Following some initial background which holds for all primes, we will specialise to $p=3$ following Theorem~\ref{thm:M} for the remainder of Section~\ref{sec:prelims}, and after Proposition~\ref{prop:case2} we will fix $p=3$ for the rest of the article.
%We follow the notation of \cite{GL2}, and summarise below the main points for the convenience of the reader.

For a finite group $G$, let $\Char(G)$ denote the set of complex characters of $G$, and let $\Irr(G)$ (resp.~$\Lin(G)$) denote the subset of those which are irreducible (resp.~linear). If $\chi\in\Irr(G)$ and $\varphi\in\Char(G)$, we write $\chi\mid\varphi$ to mean that $\chi$ is an irreducible constituent of $\varphi$.

Let $\N:=\{1,2,\dotsc\}$ and $\N_0:=\N\cup\{0\}$. For $m\in\N$, define $[m]:=\{1,2,\dotsc,m\}$ and $[\overline{m}]:=\{0,1,\dotsc,m-1\}$. We use $\delta_{\alpha,\beta}$ (or $\delta_{\alpha\beta})$ to denote the Kronecker delta of two variables $\alpha$ and $\beta$.%, giving value 1 if $\alpha=\beta$ and 0 otherwise.

For a partition $\lambda=(\lambda_1,\lambda_2,\dotsc,\lambda_{l(\lambda)})$, its size, length and conjugate are denoted by $|\lambda|$, $l(\lambda)$ and $\lambda'$ respectively.
For $A\subseteq\bigcup_{n\in\N} \cP(n)$, define $A':=\{\lambda'\ |\ \lambda\in A\}$ and $A^\circ:=A\cup A'$. In other words, ${}^\circ$ denotes taking the closure under the conjugation of partitions.

If $n\in\N$ satisfies $n\ge\lambda_1$, then we let $(n,\lambda)$ denote the partition $(n,\lambda_1,\lambda_2,\dotsc,\lambda_{l(\lambda)})$. We also denote by $\lambda+\mu$ the component-wise addition of two partitions $\lambda$ and $\mu$.% (treating $\lambda_i=0$ whenever $i>l(\lambda)$ and $\mu$ similarly).

Throughout, we will also use Mackey's Theorem on the induction and restriction of characters.
\begin{theorem}[{\cite[\textsection 5]{IBook}}]\label{thm:mackey}
	Let $G$ be a finite group. Suppose $H,K\le G$ and let $\phi\in\Char(H)$. Then
	\[ \phi\up_H^G\down_K = \sum_{g\in H\setminus G/K}\phi^g\down_{H^g\cap K}^{H^g}\up^K. \]
	In particular, if $KH=G$ then $\phi\up_H^G\down_K = \phi\down^H_{H\cap K}\up^K$.
\end{theorem}

%========================================================================
\subsection{Wreath products}\label{sec:wreath}
Let $G$ be a finite group, $n\in\N$ and $H\le S_n$. Consider the imprimitive wreath product $G\wr H$. %; we sometimes refer to $G^{\times n}$ as the base group of this wreath product. 
As in \cite[Chapter 4]{JK}, we denote the elements of $G\wr H$ by $(g_1,\dotsc,g_n;h)$ for $g_i\in G$ and $h\in H$. Let $V$ be a $\mathbb{C}G$--module and let its character be $\phi$.
Let $V^{\otimes n}$ be the corresponding $\mathbb{C}G^{\times n}$--module. The left action of $G\wr H$ on $V^{\otimes n}$ defined by linearly extending
\[ (g_1,\dotsc,g_n;h)\ :\quad v_1\otimes \cdots\otimes v_n \longmapsto g_1v_{h^{-1}(1)}\otimes\cdots\otimes g_nv_{h^{-1}(n)} \]
turns $V^{\otimes n}$ into a $\mathbb{C}(G\wr H)$--module, which we denote by $\widetilde{V^{\otimes n}}$ (see \cite[(4.3.7)]{JK}), with corresponding character $\tilde{\phi}$. For $\psi\in\Char(H)$, we let $\psi$ also denote its inflation to $G\wr H$ and define
\[ \cX(\phi;\psi):=\tilde{\phi}\cdot\psi \in \Char(G\wr H). \]
Note that if $K\le G$ and $L\le H$, then we have
\[ \cX(\phi;\psi)\down^{G\wr H}_{K\wr L} = \cX(\phi\down^G_K; \psi\down^H_L). \]
Moreover, if $\varphi\in\Irr(K)$ then we denote by $\Irr(G\mid\varphi)$ the set of characters $\chi\in\Irr(G)$ such that $\varphi$ is an irreducible constituent of the restriction $\chi\down_K$. Then for $\phi\in\Irr(G)$ and $\psi\in\Irr(H)$, we have that $\cX(\phi;\psi)\in\Irr(G\wr H\mid\phi^{\times n})$, as $\tilde{\phi}\in\Irr(G\wr H)$ is an extension of $\phi^{\times n}\in\Irr(G^{\times n})$. Further, Gallagher's Theorem \cite[Corollary 6.17]{IBook} gives
\[ \Irr(G\wr H\mid \phi^{\times n}) = \{\cX(\phi;\psi)\mid \psi\in\Irr(H)\}. \]
Specifically, if $C_p$ is cyclic of order $p$, then every $\psi\in\Irr(G\wr C_p)$ is either of the form
\begin{itemize}\setlength\itemsep{0.5em}
\item[(a)] $\psi=(\phi_{i_1}\times\cdots\times\phi_{i_p})\up^{G\wr C_p}_{G^{\times p}}$, where $\phi_{i_1},\dotsc,\phi_{i_p}\in\Irr(G)$ are not all equal (notice such $\psi$ is indeed irreducible by Clifford theory); or
\item[(b)] $\psi=\cX(\phi;\theta)$ for some $\phi\in\Irr(G)$ and $\theta\in\Irr(C_p)$.
\end{itemize}
When (a) holds, $\psi\down_{G^{\times p}}$ is the sum of the $p$ irreducible characters of $G^{\times p}$ whose $p$ factors are a cyclic permutation of $\phi_{i_1},\dotsc,\phi_{i_p}$. When (b) holds, $\psi\down_{G^{\times p}} = \phi^{\times p}\cdot\theta(1) = \phi^{\times p}$.

We record some useful results describing the irreducible constituents of restrictions and inductions of characters of wreath products. The first is elementary, while the proofs of the next two may be found in \cite[Lemmas 2.18, 2.19]{SLThesis}.

\begin{lemma}\label{lem: easy observation}
Let $G$ be a finite group and $H\le S_n$ for some $n\in\N$. Let $\chi\in\Irr(G)$. Then 
\[ \chi^{\times n}\up^{G\wr H}_{G^{\times n}} = \sum_{\theta\in\Irr(H)} \theta(1)\cdot\cX(\chi;\theta). \]
\end{lemma}

\begin{lemma}\label{lem: 9.5}
	Let $p$ be an odd prime and $G$ be a finite group. 
	Let $\Delta\in\Char(G)$ and $\alpha\in\Irr(G)$. If $\langle \Delta,\alpha\rangle\ge 2$, then
	$ \langle \cX(\Delta;\theta), \cX(\alpha;\beta)\rangle \ge 2$
	for all $\theta,\beta\in\Irr(C_p)$.
\end{lemma}

\begin{lemma}\label{lem: 9.6}
Let $G$, $H$ be finite groups with $H\le S_m$ for some $m\in\N$, and let $\theta\in\Irr(H)$. Let $\alpha\in\Irr(G)$ and $\Delta\in\Char(G)$ be such that $\langle\Delta,\alpha\rangle=1$. Then for any $\beta\in\Irr(H)$,
\[ \langle\cX(\Delta;\theta), \cX(\alpha;\beta)\rangle = \langle\theta,\beta\rangle = \delta_{\theta,\beta}. \]
\end{lemma}

\medskip

%========================================================================
\subsection{The Sylow $p$-subgroups of $S_n$}\label{sec: sylow prelims}
We refer the reader to \cite[Chapter 4]{JK} for further detail on Sylow subgroups of symmetric groups. %In this subsection, we also introduce the notation necessary to state Theorems~\ref{thm:a}, \ref{thm:b-primepower} and \ref{thm:b-general}.
For any $k\in\N_0$, the Sylow $p$-subgroup $P_{p^k}$ is isomorphic to the $k$-fold wreath product $P_p\wr \cdots\wr P_p$, noting that $P_p$ is cyclic of order $p$.
If $n\in\N$ has $p$-adic expansion $n=\sum_{i=1}^t a_ip^{n_i}$, that is, $t\in\N$, $0\le n_1<\cdots<n_t$ and $a_i\in[p-1]$ for all $i$, then $P_n\cong (P_{p^{n_1}})^{\times a_1}\times\cdots\times (P_{p^{n_t}})^{\times a_t}$.

To fix a parametrisation of the linear characters of $P_n$, we begin with $n=p^k$.
Let $\Irr(P_p)=\{\phi_0, \phi_1,\dotsc, \phi_{p-1}\}=\Lin(P_p)$, where $\phi_0=\triv_{P_p}$ is the trivial character. 
When $k\geq 2$, 
\[ \Lin(P_{p^k}) = \bigsqcup_{\phi\in\Lin(P_{p^{k-1}})} \Irr(P_{p^k}\ |\ \phi^{\times p}) = \bigsqcup_{\phi\in\Lin(P_{p^{k-1}})} \{\cX(\phi;\psi)\ |\ \psi\in\Lin(P_p)\} \]
by \cite[Corollary 6.17]{IBook}. Thus we may parametrise $\Lin(P_{p^k}) = \{\phi(s) \mid s\in [\overline{p}]^k \}$ as follows: if $k=0$ and $s$ is the empty sequence, then $\phi(s)=\triv_{P_1}$, while if $k=1$ and $s=(\sfx)$ for $\sfx\in[\overline{p}]$, then $\phi(s)=\phi_{\sfx}$. If $k\geq 2$, then for any $s=(\sfs_1,\ldots, \sfs_k)\in [\overline{p}]^k$, we recursively define 
\[ \phi(s):=\cX\big(\phi(s^-); \phi_{\sfs_k}\big), \]
where $s^-=(\sfs_1,\ldots, \sfs_{k-1})\in [\overline{p}]^{k-1}$. For any $i\in[k-1]$, the associativity of wreath products gives $\phi(s)=\cX\big(\phi((\sfs_1,\dotsc,\sfs_i));\phi((\sfs_{i+1},\dotsc,\sfs_k))\big)$. 
Moreover, under this labelling the trivial character $\triv_{P_{p^k}}$ corresponds to the sequence $(0,\dotsc,0)\in[\overline{p}]^k$. (This indexing of $\Lin(P_{p^k})$ may equivalently be obtained from its canonical bijection to $\Lin(P_{p^k}/P_{p^k}')$, noting that $P_{p^k}/P_{p^k}'\cong (P_p)^k$.)

\begin{remark}\label{rem:binary}
	Let $s,t\in[\overline{p}]^k$. From \cite[Lemma 4.3]{SLThesis}, $\phi(s)$ and $\phi(t)$ are $N_{S_{p^k}}(P_{p^k})$--conjugate if and only if the sequences $s$ and $t$ contain 0s in exactly the same positions. If this condition on $s$ and $t$ holds, then $\phi(s)\up^{S_{p^k}}=\phi(t)\up^{S_{p^k}}$ and so $\Omega(\phi(s))=\Omega(\phi(t))$. Since our main object of study is $\Omega(\phi)$ for $\phi\in\Lin(P_{p^k})$, it therefore suffices to consider only $\phi=\phi(s)$ for $s\in\{0,1\}^k$, rather than $s\in[\overline{p}]^k=\{0,1,\dotsc,p-1\}^k$. 
	
	%Henceforth, when investigating $\Omega(\phi)$ for linear characters $\phi$ of $P_{p^k}$, we will only consider those of the form $\phi(s)$ for $s\in\{0,1\}^k$.
	When the meaning is clear, we make a number of abbreviations, such as $s=(\sfs_1,\sfs_2,\dotsc,\sfs_k)$ to $\sfs_1 \sfs_2 \dotsc \sfs_k$, strings of $l$ consecutive zeros (resp.~ones) to $0^l$ (resp.~$1^l$), $\Omega((\sfs_1,\dotsc,\sfs_k))$ to $\Omega(\sfs_1,\dotsc,\sfs_k)$, and for $\sfx\in\{0,1\}$ we let $(s,\sfx)$ denote the concatenation of $s$ and $(\sfx)$. \hfill$\lozenge$
\end{remark}

Now let $n\in\N$ be arbitrary. Suppose $n$ has $p$-adic expansion $n=\sum_{i=1}^t a_ip^{n_i}$ where $0\le n_1<\cdots<n_t$ and $a_i\in[p-1]$.
Since $P_n\cong (P_{p^{n_1}})^{\times a_1}\times\cdots\times (P_{p^{n_t}})^{\times a_t}$, then
\begin{equation}\label{eqn: index}
	\Lin(P_n) = \{\phi(\underline{\mathbf{s}})\ |\ \underline{{\mathbf{s}}}=\big(\mathbf{s}(1,1),\ldots, \mathbf{s}(1,a_1), \mathbf{s}(2,1),\ldots, \mathbf{s}(2,a_2), \ldots, \mathbf{s}(t,a_t)\big) \},
\end{equation}
where for all $i\in[t]$ and $j\in[a_i]$ we have that ${\bf{s}}(i,j)\in[\overline{p}]^{n_i}$, and 
\[ \phi(\underline{\mathbf{s}}) := \phi(\mathbf{s}(1,1))\times\cdots\times\phi(\mathbf{s}(1,a_1))\times \phi(\mathbf{s}(2,1))\times\cdots\times\phi(\mathbf{s}(2,a_2))\times\cdots\times\phi(\mathbf{s}(t,a_t)). \]
%Note that if $n_1=0$ then $\underline{\mathbf{s}}(1,j)$ equals the empty sequence $\emptyset$ for all $j\in[a_1]$. 

\begin{notation}\label{not:correspond}
	When we write that $\phi(\underline{\mathbf{s}})$ is a linear character of $P_n$, we mean that $\underline{\mathbf{s}}$ is %a sequence of sequences 
	of the form described in (\ref{eqn: index}) above. 
	To simplify notation, we define $R:=\sum_{i=1}^t a_i$, let $\{s_1,\dotsc,s_R\}$ be the multiset given by $\{\mathbf{s}(i,j) \mid i\in[t],\ j\in[a_i]\}$, and say that \emph{$\phi$ corresponds to $\{s_1,\dotsc,s_R\}$}.
\end{notation}

\begin{remark}
	Let $\phi(\underline{\mathbf{s}})$ and $\phi(\underline{\mathbf{t}})\in\Lin(P_n)$. By \cite[Lemma 4.5]{SLThesis}, $\phi(\underline{\mathbf{s}})$ and $\phi(\underline{\mathbf{t}})$ are $N_{S_n}(P_n)$--conjugate if and only if there exists some $\sigma\in S_{a_1}\times S_{a_2}\times\cdots\times S_{a_t}$ such that for all $i\in[t]$ and $j\in[a_i]$, the sequences $\mathbf{s}(i,j)$ and $\mathbf{t}(i,\sigma(j))\in [\overline{p}]^{n_i}$ contain 0s in exactly the same positions. When this holds, then $\phi(\underline{\mathbf{s}})\up^{S_n} = \phi(\underline{\mathbf{t}})\up^{S_n}$, and so $\Omega(\phi(\underline{\mathbf{s}})) = \Omega(\phi(\underline{\mathbf{t}}))$. In particular, the order of %the elements in the multiset 
	$s_1,\dotsc,s_R$ as described above does not matter for our purposes of studying $\Omega(\phi)$.\hfill$\lozenge$
\end{remark}

If $\phi\in\Lin(P_{p^k})$ corresponds to sequence $s\in\{0,1\}^k$, we sometimes denote $\Omega(\phi)$, $m(\phi)$ and $M(\phi)$ by $\Omega(s)$, $m(s)$ and $M(s)$ respectively. Similarly if $\phi\in\Lin(P_n)$ corresponds to $\underline{\mathbf{s}}$ as in \eqref{eqn: index}, we sometimes use the notation $\Omega(\underline{\mathbf{s}})$, $m(\underline{\mathbf{s}})$ and $M(\underline{\mathbf{s}})$.

With a parametrisation of $\Lin(P_n)$ fixed, we can now define what it means for a linear character of $P_n$ to be \textit{quasi-trivial}.

\begin{definition}\label{def:quasi-trivial}
	Let $p$ be a prime. Let $n\in\N$ and $\phi\in\Lin(P_n)$. Suppose $\phi$ corresponds to the multiset of sequences $\{s_1,\dotsc, s_R\}$. We say that $\phi$ is \emph{quasi-trivial} if there is at most one non-zero entry in each component sequence $s_i$.
	(Note that $\phi=\triv_{P_n}$ corresponds to $\{s_1,\dotsc,s_R\}$ where each sequence $s_i$ contains only 0s, or is empty.)
\end{definition}

\medskip

%========================================================================
\subsection{The main results}\label{sec:main-thms}
In this subsection we present the statements of our main results: a description of $\Omega(\phi)$ for quasi-trivial $\phi\in\Lin(P_n)$ (Theorem~\ref{thm:a}), and the values of $m(\phi)$ and $M(\phi)$ for all $\phi\in\Lin(P_n)$ (Theorems~\ref{thm:b-primepower} and~\ref{thm:b-general}). 

First, we record in our current notation \cite[Theorem A]{GL1}, which describes $\Omega(\triv_{P_n})$.
\begin{theorem}\label{thm:GL1A}
	Let $p$ be an odd prime and $n\in\N$. If $p=3$ then further suppose $n>10$. Then
	\[ \Omega(\triv_{P_n}) = \begin{cases}
		\cP(p^k)\setminus\{(p^k-1,1),(2,1^{p^k-2})\} & \text{if}\ n=p^k\text{ for some }k\in\N,\\
		\cP(n) & \text{otherwise}.
	\end{cases} \]
	(The values of $\Omega(\triv_{P_n})$ when $p=3$ and $n\le 10$ are listed in Table~\ref{table:m<27} below.)
\end{theorem}

For odd primes, both $p=3$ and otherwise, it turns out that the shape of $\Omega(\phi)$ and the values of $m(\phi)$ and $M(\phi)$ are largely governed by the positions of certain 1s in the sequences $s\in\{0,1\}^k$ when $n=p^k$ and $\phi=\phi(s)$, and similarly by the positions of 1s in various $s_i$ when $n\in\N$ is general and $\phi$ corresponds to $\{s_1,\dotsc,s_R\}$. Therefore we make the following collection of definitions in order to refer to such positions.

\begin{definition}\label{def:sequences}
	Let $k\in\N$ and $s=(\sfs_1,\sfs_2,\dotsc,\sfs_k)\in\{0,1\}^k$.
	\begin{itemize}%\setlength\itemsep{0.5em}
		\item[$\circ$] For $z\in\{0,1,\dotsc,k\}$, let $U_k(z):=\{s\in\{0,1\}^k \mid \#\{i\in[k] \mid \sfs_i=1\} = z \}$. 
		
		That is, $U_k(z)$ is the set of binary sequences of length $k$ containing exactly $z$ many 1s.	
		%Note $U_0(m)$ is empty for $m\in\N$, and $U_0(0)=\{\emptyset\}$.
		
		\item[$\circ$] If $s\in U_k(z)$ where $z\ge 1$, then define $f(s):=\min\{i\in[k] \mid \sfs_i=1\}$. %If there exists $i\in[k]$ such that $s_i\ne 0$, then define $f(s)=\min\{i\in[k] \mid s_i\ne 0\}$.
		\item[$\circ$] If $s\in U_k(z)$ where $z\ge 2$, then define $g(s)=\min\{i>f(s) \mid \sfs_i=1 \}$. 
		%and set $$\eta(s)=p^k-p^{k-f(s)}-p^{k-g(s)}.$$ %If $|\{i\in[k] : s_i\ne 0\}|\ge 2$, then define $g(s)=\min\{i>f(s) \mid s_i\ne 0 \}$ and set $\eta(s)=p^k-p^{k-f(s)}-p^{k-g(s)}$.
	\end{itemize}
	In other words, $f(s)$ and $g(s)$ are just the positions of the leftmost and second leftmost 1s in the sequence $s$, when they exist.
	
	Now suppose $\sfs_1=1$.
	\begin{itemize}
		\item[$\circ$] Define $F(s) := \max\{i\in[k] \mid \sfs_1=\sfs_2=\cdots=\sfs_i=1 \} $. Clearly $z\ge F(s)$.
		\item[$\circ$] If $s\in U_k(z)$ where $z\ge F(s)+1$, then define $G(s) := \min\{ i>F(s) \mid \sfs_i=1\}$.
		\item[$\circ$] If $s\in U_k(z)$ where $z\ge F(s)+2$, then define $H(s) := \min\{ i>G(s) \mid \sfs_i=1 \}$.
	\end{itemize}
	In other words, $F(s)$ is the rightmost position in the consecutive subsequence of 1s containing $\sfs_1$. Then $G(s)$ is the position of the leftmost 1 to the right of $F(s)$, and $H(s)$ is the next leftmost 1 to the right of $G(s)$, if these exist. To give a visual example\footnote{
	The following may be useful mnemonics for Definition~\ref{def:sequences}:
	\begin{itemize}
		\item $U_k(z)$ counts the number of \textbf{U}nits or 1s in $s$;
		\item reading from the left, $f(s)$ is the position of the \textbf{f}irst 1 in $s$ (and $g(s)$ is the position of the next 1);
		\item when $s$ starts with consecutive 1s, $F(s)$ is the position of the last 1 in that \textbf{F}irst block of 1s (and $G(s)$ and $H(s)$ are the positions of the next 1s).
	\end{itemize}
	%Clearly if $s\in U_k(z)$ and $s_1=1$, then $z\ge F(s)$. However, when $G(s)$ exists then $G(s)\ne F(s)+1$ by the maximality of $F(s)$, although $G(s)$ and $H(s)$ may be adjacent.
	},
	\[ s = (\underset{\sfs_1=1}{1}, 1, \dotsc, 1, \underset{F(s)}{1}, 0, \dotsc, 0, \underset{G(s)}{1}, 0, \dotsc, 0, \underset{H(s)}{1}, \dotsc) \]
\end{definition}

The above notation would already allow us to state our main result giving the values of $M(\phi)$ and $m(\phi)$ for all $\phi\in\Lin(P_n)$ in the case where $n$ is a power of 3 (Theorem~\ref{thm:b-primepower}). However, we save the statement for later, and present all of our main results together at the end of Section~\ref{sec:main-thms}. 

Formulas for $M(\phi)$ for all odd primes were already found in Theorem~\ref{thm:GL1A} and \cite[Theorems 4.5, 5.2]{GL2}, which we record in the following theorem.
%In the meantime, with Definition~\ref{def:sequences} we can now observe that $M(\phi)$ is easily described for $\phi\in\Lin(P_n)$ when $p=3$, and in fact follows the uniform pattern which holds for $p\ge 5$ as well.

\begin{theorem}
	\label{thm:M}
	Let $p$ be any odd prime. 
	\begin{itemize}
		\item[(a)] Let $k\in\N$ and $s\in\{0,1\}^k$, and suppose $s\in U_k(z)$ for some $z\in\{0,1,\dotsc,k\}$. Then
		\begin{small}
			\[ M(s) = \begin{cases}
				p^k & \text{if }z=0\ \ (\text{i.e.}\ \phi(s)=\triv_{P_{p^k}}),\\
				p^k-p^{k-f(s)} & \text{if }z>0.
			\end{cases} \]
		\end{small}
		
		\item[(b)] Let $n\in\N$ and suppose $\phi(\underline{\mathbf{s}})\in\Lin(P_n)$ as in \eqref{eqn: index}. Then $M(\underline{\mathbf{s}}) = \sum_{(i,j)} M\big(\mathbf{s}(i,j)\big)$.
	\end{itemize}
\end{theorem}

However, the number of forms of $\Omega(\phi)$ and $m(\phi)$ that arise when $p=3$ far exceed those for $p\ge 5$: a visual comparison is given by Figures~\ref{fig:5} and~\ref{fig:3}, and indeed, a much more intricate analysis is required to fully describe the behaviour for $p=3$ compared with $p\ge 5$.

\smallskip

For the remainder of Section~\ref{sec:main-thms}, let $p=3$ unless otherwise stated. Since for $n=\sum_{i=1}^t a_ip^{n_i}$, the Sylow subgroup $P_n$ is a direct product of smaller Sylow subgroups $P_{p^{n_i}}$, we will tackle the $n=p^k$ case first, before combining these to analyse what happens for general $n\in\N$.

When considering $\phi\in\Lin(P_{3^k})$, the forms of $\Omega(\phi)$ and $m(\phi)$ present a more uniform pattern for $k\ge 3$. Therefore, we will often state definitions and results for $k\ge 3$, then observe that instances of small $k$ may be computed directly or are otherwise straightforward. Equally, to deal with general $\varphi\in\Lin(P_n)$ for $n$ not necessarily a power of 3, we will split out the `small case behaviour' by instead considering characters of the form $\phi\times\psi\in\Lin(P_{n+l})$ where $n\in\N$ is divisible by 27 and $l\in\{0,1,\dotsc,26\}$.

Next, to each sequence $s\in\{0,1\}^k$ we assign a \textit{type}, $\sigma(s)$, which reflects the general form of $\Omega(s)$ and $m(s)$. We also define an associated quantity $N(s)$ which will be used to help compute the value of $m(\phi)$ for arbitrary linear $\phi$ later. Indeed, we will see later that $N(s)$ is the maximal $N$ such that %$\chi^{(N,n-N)}$ or $\chi^{(N,1^{n-N})}$ 
$(N,n-N)$ or $(N,1^{n-N})$ is contained in $\Omega(s)$ (and actually both will be contained in $\Omega(s)$). 

Partitions of these shapes will play an important role: specifically, we will call a partition $\lambda$ \emph{thin} if $\lambda$ is a hook, or if $\lambda$ satisfies $l(\lambda)\le 2$ or $\lambda_1\le 2$ (see Definition~\ref{def:thin} below).

\begin{definition}\label{def: typephi3}
	Let $k\in\N_{\ge 3}$. In Table~\ref{table:types} below, we list all possible sequences $s\in\{0,1\}^k$ in the second column. For each $s$, its %\emph{(3-Sylow-) type}
	\emph{type}\footnote{
			The naming of the types $\sigma(s)$ has been chosen to reflect certain patterns in the values of $N(s)$ compared with $m(s)$. The quantities $N(s)$ are central for computing $m(\phi)$ for arbitrary linear $\phi$ later: see Definition~\ref{def: all p=3} (i) and Theorems~\ref{thm:a} and~\ref{thm:b-general} below.
			% For example, some types such as 3 compared with 30 or 31 have been named to illustrate a general pattern with slightly different behaviour in some special cases. 
			%While type 6 seems to resemble type 5, we will see from later results such as Theorem~\ref{thm:a} that the calculations and values of $\Omega(\phi)$ and $m(\phi)$ differ notably.
			
			A key observation is that these types $\sigma(s)$ which have been defined for $p=3$ are very different from the types $\tau(s)$ \cite[Definition 2.6]{GL2} for $p\ge 5$. While the purpose of both $\sigma(s)$ and $\tau(s)$ is to group linear characters $\phi(s)\in\Lin(P_{p^k})$ according to the general forms of $\Omega(s)$ and $m(s)$, the groupings that arise when $p\ge 5$ are insufficient to describe and hence not applicable to $p=3$.} 
		is the number $\sigma(s)$ defined in the first column. We also define the number $N(s)$ for each $s$ in the third column. In light of Definition~\ref{def:quasi-trivial}, we indicate in the fourth column whether or not the sequences corresponding to each type are quasi-trivial.
	
	\begin{table}[h]
		\centering
		\begin{small}
			\[ \def\arraystretch{1.2}
			\begin{array}{|l|cl|cl|c|}
				\hline
				\sigma(s) & s & & \multicolumn{2}{c|}{N(s)} & \text{quasi-trivial?}\\
				\hline
				\hline
				1 & (0\dotsc 0), & k\ge 3 & m(s)+2 & =3^k & \checkmark \\
				10 & (1\dotsc 10\dotsc 0), & 2\le F(s)\le k-3 & m(s)+2 & = \tfrac{3^k+3^{k-F(s)}}{2} & \times \\ 
				11 & (1\dotsc 10) & 2\le F(s) = k-1 & m(s)+2 & = \tfrac{3^k+3}{2} & \times \\
				2 & (0\dotsc 01) & k\ge 3 & m(s) & =3^k-1 & \checkmark \\
				21 & (1\dotsc 1) & k\ge 3 & m(s) & =\tfrac{3^k+1}{2} & \times \\
				3 & (0\dotsc 010\dotsc 0) & 2\le f(s)\le k-3 & m(s)+1 & =3^k-3^{k-f(s)} & \checkmark \\
				30 & (0\dotsc 0100) & 2\le f(s)=k-2 & m(s)+1 & %=3^k-3^{k-f(s)} 
				= 3^k-9 & \checkmark \\
				31 & (0\dotsc 010) & 2\le f(s)=k-1 & m(s)+1 & %=3^k-3^{k-f(s)} 
				= 3^k-3 & \checkmark \\
				5 & (10\dotsc 0) & k\ge 4 & m(s)+2 & %=3^k-3^{k-f(s)} =\tfrac{3^k+3^{k-F(s)}}{2}
				= 2\cdot 3^{k-1} & \checkmark \\
				6 & (100) & & m(s)+5 & =18 & \checkmark \\
				7 & (1\dotsc 100) & 2\le F(s)=k-2 & m(s)+5 & = \tfrac{3^k+9}{2} & \times \\
				\hdashline
				22 & \text{all other}\ s & & m(s) & & \times \\
				\hline
			\end{array}\]
		\end{small}
		\caption{Definition of type $\sigma(s)$ and $N(s)$ for $s\in\{0,1\}^k$ with $k\ge 3$.}\label{table:types}
	\end{table}
\end{definition}

\begin{definition}\label{def:k<3}
	Let $l\in\{1,2,\dotsc,26\}$ and $\psi\in\Lin(P_l)$. Suppose $\psi$ corresponds to $\{s_1,\dotsc,s_R\}$. We list all possible $\psi$ in the first column of Table~\ref{table:m<27} below. The set $\Omega(\psi)$ may be obtained by direct computation; this is listed in the second column for reference. We define $N(\psi)$ for each $\psi$ in the third column. We also define subsets $\Psi_1$ and $\Psi_2$ of $\bigsqcup_{l=1}^{26}\Lin(P_l)$ as indicated.
	
%	Furthermore, define $\Psi_1$ and $\Psi_2$ to be the sets
%	\begin{multline*}
%		\qquad\qquad\Psi_1 = \{\phi(1,0)\cdot\triv_{P_{l-9}} : l\ge 9\} \cup \{ \phi(1,0)\cdot\phi_1\cdot\triv_{P_{l-12}} : l\ge 12,\ l\ne 18\} \\ \cup \{\phi(1,0)\cdot\phi(0,1)\cdot\triv_{P_{l-18}} : l\ge 18\} \cup \{ \phi(1,0)\cdot\phi(1,0)\cdot\triv_{P_{l-18}} : l\ge 18 \}\qquad\qquad
%	\end{multline*}
%	and
%	\[ \Psi_2 = \{ \phi_1\cdot\triv_{P_{l-3}} : l\ge 3,\ l\notin\{9,10,11,18,19,20\} \} \cup \{ \phi(0,1)\cdot\triv_{P_{l-9}} : l\ge 9 \}.\]
	
	\begin{table}[h]
		\centering
		\begin{footnotesize}
			\[ \def\arraystretch{1.2}
			\begin{array}{cl|c|l}
				\hline
				\psi & & \Omega(\psi) & N(\psi)\\
				\hline
				\hline
				%\triv_{P_l} & l\in\{3,4,6,9,10\} & \text{(see \cite[Theorem A]{GL1})} & l\\
				\triv_{P_l} & l\in\{3,4,6,9,10\} & \cP(l)\setminus X^\circ;\ {\scriptstyle X=\{(2,1)\}, \{(2,2)\},} & l\\
				&& {\scriptstyle \{(3,2,1)\}, \{(8,1),(5,4),(4,3,2)\}\ \text{or}\ \{(5,5)\}} &\\
				%\triv_{P_3} & & \cP(3)\setminus\{(2,1)\} & 3\\
				%\triv_{P_4} & & \cP(4)\setminus\{(2,2)\} & 4\\
				%\triv_{P_6} & & \cP(6)\setminus\{(3,2,1)\} & 6\\
				%\triv_{P_9} & & \cP(9)\setminus\{ (8,1),(5,4),(4,3,2)\}^\circ & 9\\
				%\triv_{P_{10}} & & \cP(10)\setminus\{(5,5)\}^\circ & 10\\
				\bm{\triv_{P_l}} & \bm{l\notin\{3,4,6,9,10\}} & \bm{\cP(l)} & \bm{l}\\
				\hline
				
				\phi_1\cdot\triv_{P_3} & (l=6) & \cB_6(5)\setminus\{(3,3)\}^\circ & 5 \\
				\phi_1\cdot\triv_{P_9} & (l=12) & \cB_{12}(11)\setminus\{(6^2),(4^3)\}^\circ & 11\\
				\bm{\phi_1\cdot\triv_{P_{l-3}}} & \bm{\scriptstyle{l\ge 3,\ l\notin\{6,9,10,11,12,18,19,20\}}} & \bm{\cB_l(l-1)} & \bm{l-1} \hspace{8pt} %\\
				
				\smash{\left.\begin{array}{@{}c@{}}\\[\jot]\\[\jot]\\[\jot]\\[\jot]\end{array}\right\}} \Psi_2\\[\jot]
				%\hdashline
				
				\phi(0,1) & (l=9) & \cB_9(8)\setminus\{(3^3)\} & 8\\
				\bm{\phi(0,1)\cdot\triv_{P_{l-9}}} & \bm{l>9} & \bm{\cB_l(l-1)} & \bm{l-1}\\
				\hline
				
				\phi(1,0) & (l=9) & \cB_9(6)\setminus\{(5,4),(5,1^4),(6,2,1) \}^\circ & 6\\
				\phi(1,0)\cdot\triv_{P_1} & (l=10) & \cB_{10}(7)\setminus\{(5,5),(7,2,1)\}^\circ & 7\\
				\bm{\phi(1,0)\cdot\triv_{P_{l-9}}} & \bm{l\ge 11} & \bm{\cB_l(l-3)\setminus\{(l-3,2,1)\}^\circ} & \bm{l-3}\\
				%\hdashline
				
				\phi(1,0)\cdot\phi_1 & (l=12) & \cB_{12}(8)\setminus\{(6,6),(8,2,2)\}^\circ & 8 \hspace{25pt} %\\
				
				\smash{\left.\begin{array}{@{}c@{}}\\[\jot]\\[\jot]\\[\jot]\\[\jot]\\[\jot]\\[\jot]\end{array}\right\}} \Psi_1\\[\jot]
				
				\bm{\phi(1,0)\cdot\phi_1\cdot\triv_{P_{l-12}}} & \bm{l\ge 13,\ l\notin\{18, 19, 20\}} & \bm{\cB_l(l-4)\setminus\{(l-4,2,2)\}^\circ} & \bm{l-4}\\
				%\hdashline
				
				\bm{\phi(1,0)\cdot\phi(0,1)\cdot\triv_{P_{l-18}}} & \bm{l\ge 18} & \bm{\cB_l(l-4)\setminus\{(l-4,2,2)\}^\circ} & \bm{l-4}\\
				%\hdashline
				
				\bm{\phi(1,0)\cdot\phi(1,0)\cdot\triv_{P_{l-18}}} & \bm{l\ge 18} & \bm{\cB_l(l-6)\setminus\{(l-6,3,2,1)\}^\circ} & \bm{l-6}\\
				\hline
				
				\phi(1,1) & (l=9) & \Omega(\psi)\supseteq\cB_{9}(5)\setminus\{(3^3)\}\ \text{a.c.n.o.t.p} & 5\\
				\phi(1,1)\cdot\triv_{P_3} & (l=12) & \Omega(\psi)\supseteq\cB_{12}(8)\setminus\{(6,6)\}^\circ\ \text{a.c.n.o.t.p} & 8\\
				\phi(1,1)\cdot\phi(1,0) & (l=18) & \Omega(\psi)\supseteq\cB_{18}(11)\setminus\{(9,9)\}^\circ\ \text{a.c.n.o.t.p} & 11\\
				\phi(1,1)\cdot\triv_{P_9} & (l=18) & \Omega(\psi)\supseteq\cB_{18}(14)\setminus\{(9,9)\}^\circ\ \text{a.c.n.o.t.p} & 14\\
				\bm{s_i=(1,1)\ \text{\bf{for some}}\ i}  & \text{\bf{except above}} & \bm{\Omega(\psi)\supseteq\cB_l(m(\psi))}\ \text{\bf{a.c.n.o.t.p}} & \bm{m(\psi)}\\
				\hline
				
				\text{\bf{all other}}\ \bm{\psi} & \bm{\scriptstyle{(\text{in this case }l-m(\psi)\ge 2)}} & \bm{\Omega(\psi) = \cB_l(m(\psi))} & \bm{m(\psi)}\\
				\hline
			\end{array}\]
		\end{footnotesize}
		\caption{Definition of $N(\psi)$ for $\psi\in\Lin(P_l)$, $l<27$. Here `and contains no other thin partitions' is abbreviated to \emph{a.c.n.o.t.p}.
		We remark that the bold entries %in Table~\ref{table:m<27}
		represent general forms of $\Omega(\psi)$, while the unbolded entries mimic the general forms but with some small partitions excluded.
		While not indicated in the table, we note that $\psi$ is quasi-trivial if and only if $s_i\ne(1,1)$ for all $i\in[R]$.
		}\label{table:m<27}
	\end{table}
\end{definition}

\begin{definition}\label{def: all p=3}
	Let $n\in\N$ be divisible by 27, and $l\in\{0,1,\dotsc,26\}$.
	Let $\phi\times\psi\in\Lin(P_{n+l})$ where $\phi\in\Lin(P_n)$ and $\psi\in\Lin(P_l)$. (If $l=0$ then we may omit the $\psi$ term, or equally treat it as $\psi=\triv_{P_0}$ where $P_0$ is the trivial group.) Suppose $\phi$ corresponds to $\{s_1,\dotsc,s_R\}$.
	\begin{itemize}
		\item[(i)] Define $N(\phi):=\sum_{i=1}^R N(s_i)$ and $N(\phi\times\psi) := N(\phi)+N(\psi)$. 
		
		(See Tables~\ref{table:types} and~\ref{table:m<27} for the values of $N(s_i)$ and $N(\psi)$.)
		
		\item[(ii)] If $\phi$ is quasi-trivial, equivalently if $\sigma(s_i)\in \{1,2,3,30,31,5,6\}$ for all $i\in[R]$, then define
		\[ \sigma(\phi):=(y_1,y_2,y_3,y_{30},y_{31},y_5,y_6) \]
		where $y_j = \#\{i\in[R] \mid \sigma(s_i)=j\}$.
		(We remark that $\phi$ is quasi-trivial if and only if $\phi(s_i)$ is for each $i$, and $\phi\times\psi$ is quasi-trivial if and only if both $\phi$ and $\psi$ are.)
		
		\item[(iii)] Let $E\supset F$ be the sets defined as follows:
		\[ E\ :=\ {\tiny \begin{matrix}
			\{ (R-1,0,1,0,0,0,0),\\
			\ (R-1,0,0,1,0,0,0),\\
			\ (R-1,0,0,0,1,0,0),\\
			\ (R-2,0,0,0,2,0,0),\\
			\ (R-2,1,0,1,0,0,0),\\
			\ (R-2,1,0,0,1,0,0),\\
			\ (R-1,0,0,0,0,1,0),\\
			\ (R-2,1,0,0,0,0,1),\\
			\ \ (R-1,0,0,0,0,0,1)\},
		\end{matrix}}\qquad
		F\ :=\ \begin{matrix}
			\{ (R-1,0,0,1,0,0,0),\\
			(R-1,0,0,0,1,0,0),\\
			(R-1,0,0,0,0,0,1)\}.
		\end{matrix} \]
	\end{itemize}
\end{definition}

\begin{remark}
	Notice that if $\sigma(\phi)\in E$, then $n-N(\phi)\in\{3,4,6,9,10\}\cup\{3^k \mid k\ge 3\}$; this is exactly the set of $t\in\N$ such that $\Omega(\triv_{P_t})\ne\cP(t)$ by Theorem~\ref{thm:GL1A}.
	Moreover, if $\sigma(\phi)\in F$ then $n-N(\phi)\in\{3,9\}$, exactly the subset %of the above 
	such that $\Omega(\triv_{P_{t+1}})\ne\cP(t+1)$ also holds.\hfill$\lozenge$
\end{remark}

We are now ready to give the precise statement of our first main result, which determines $\Omega(\phi)$ for all quasi-trivial characters $\phi$ of $P_n$ when $p=3$. Note the operator $\star$ is defined in Definition~\ref{def:star operator} below.

\begin{theorem}\label{thm:a}
	Let $p=3$. Let $n\in\N$ be divisible by 27 and $l\in\{0,1,\dotsc,26\}$. Let $\phi\times\psi\in\Lin(P_{n+l})$ be quasi-trivial, where $\phi\in\Lin(P_n)$ and $\psi\in\Lin(P_l)$. Suppose $\phi$ corresponds to $\{s_1,\dotsc,s_R\}$, and let $N=N(\phi\times\psi)$. Then
%	\begin{small}
%		\[ \Omega(\phi\times\psi) = \begin{cases}
%			\cP(n)\setminus\{(n-1,1)\}^\circ & \text{if}\ \sigma(\phi)=(1,0,0,0,0,0,0)\ \text{and}\ l=0,\\
%			\cC & \text{if}\ \sigma(\phi)=(0,0,0,0,0,1,0)\ \text{and}\ l=0,\\
%			\Omega(1,0,0)\star\Omega(\psi) & \text{if}\ \sigma(\phi)=(0,0,0,0,0,0,1),\\
%			\cB_{n+l}(N-1)\sqcup\{(N,\mu) \mid \mu\in\Omega(\triv_{P_{n+l-N}}) \}^\circ & \text{if}\ \sigma(\phi)\in E\ \text{and}\ \psi=\triv_{P_l}\\
%			& \text{(excluding $\sigma(\phi)=(0,0,0,0,0,1,0)$ and $l=0$,}\\
%			& \text{and $\sigma(\phi)=(0,0,0,0,0,0,1)$)},\\
%			& \quad\sigma(\phi)\in F\ \text{and}\ \psi\in\Psi_2,\\
%			& \quad\sigma(\phi)=(R,0,0,0,0,0,0)\ \text{and}\ \psi\in\Psi_1,\\
%			& \quad\sigma(\phi)=(R-1,0,0,0,1,0,0)\ \text{and}\\
%			& \qquad\qquad \psi=\phi(1,0)\cdot\triv_{P_{l-9}},\ l\ge 9,\ \text{or}\\
%			& \quad\sigma(\phi)=(R-1,1,0,0,0,0,0)\ \text{and}\\
%			& \qquad\qquad \psi=\phi(1,0)\cdot\triv_{P_{l-9}},\ l\ge 9;\\
%			\cB_{n+l}(N) & \text{otherwise},
%		\end{cases} \]
%		where $a:=\tfrac{n}{3}$ and $\cC:=\cB_n(2a)\setminus\{(2a,a-1,1),(2a,2,1^{a-2}),(2a-1,a+1),(2a-1,1^{a+1}) \}^\circ$.%, and $\cD:=\Omega(1,0,0)\star\Omega(\psi)$ is obtained by direct computation.
%	\end{small}
	the value of $\Omega(\phi\times\psi)$ is as follows:
	\[ \begin{array}{ll}
		\hline
		\cP(n)\setminus\{(n-1,1)\}^\circ & \text{if}\ \sigma(\phi)=(1,0,0,0,0,0,0)\ \text{and}\ l=0,\\
		\cC & \text{if}\ \sigma(\phi)=(0,0,0,0,0,1,0)\ \text{and}\ l=0,\\
		\Omega(1,0,0)\star\Omega(\psi) & \text{if}\ \sigma(\phi)=(0,0,0,0,0,0,1),\\
		\cB_{n+l}(N-1)\sqcup\{(N,\mu) \mid \mu\in\Omega(\triv_{P_{n+l-N}}) \}^\circ & \text{if}\ \sigma(\phi)\in E\ \text{and}\ \psi=\triv_{P_l}\\
		& \qquad\text{\tiny (excluding $\sigma(\phi)=(0,0,0,0,0,1,0)$ and $l=0$,}\\
		& \qquad\text{\tiny and $\sigma(\phi)=(0,0,0,0,0,0,1)$)};\ \text{or}\\
		& \quad\sigma(\phi)\in F\ \text{and}\ \psi\in\Psi_2;\ \text{or}\\
		& \quad\sigma(\phi)=(R,0,0,0,0,0,0)\ \text{and}\ \psi\in\Psi_1;\ \text{or}\\
		& \quad\sigma(\phi)=(R-1,0,0,0,1,0,0)\ \text{and}\\
		& \qquad\qquad \psi=\phi(1,0)\cdot\triv_{P_{l-9}},\ l\ge 9;\ \text{or}\\
		& \quad\sigma(\phi)=(R-1,1,0,0,0,0,0)\ \text{and}\\
		& \qquad\qquad \psi=\phi(1,0)\cdot\triv_{P_{l-9}},\ l\ge 9;\\
		\cB_{n+l}(N) & \text{otherwise},\\
		\hline
	\end{array} \]
	where $a:=\tfrac{n}{3}$ and $\cC:=\cB_n(2a)\setminus\{(2a,a-1,1),(2a,2,1^{a-2}),(2a-1,a+1),(2a-1,1^{a+1}) \}^\circ$.
\end{theorem}

\begin{remark}
	\begin{itemize}
		\item[(i)] We take $n\ge 1$ in Theorem~\ref{thm:a} above 
		since the case of $n=0$ (i.e.~$n+l=l<27$) is given in Table~\ref{table:m<27}.
		
		\item[(ii)] As observed in Definition~\ref{def: all p=3}, $\Omega(\triv_{P_t})=\cP(t)$ for all $t\notin\{3,4,6,9,10\}\cup\{3^k\mid k\ge 3\}$. 
		The set $\Omega(\triv_{P_t})$ is listed in Table~\ref{table:m<27} if $t\in\{3,4,6,9,10\}$, and is $\cP(3^k)\setminus\{(3^k-1,1)\}^\circ$ if $t\in\{3^k\mid k\ge 3\}$ (the first case in Theorem~\ref{thm:a}). 
		Hence $\Omega(\triv_{P_{n+l-N}})$ in Theorem~\ref{thm:a} is known in all cases.
		
		\item[(iii)] The exact value of $\Omega(1,0,0)$ is listed following Proposition~\ref{prop:case2} below, and that of $\Omega(\psi)$ in Table~\ref{table:m<27}. Thus $\Omega(1,0,0)\star\Omega(\psi)$ may be obtained by direct computation using the Littlewood--Richardson rule (see Definition~\ref{def:star operator} and Theorem~\ref{theorem:LR} below).
		
		\item[(iv)] For a comparison with the case of $p\ge 5$, we refer the reader to \cite[Theorem 2.9]{GL2}.
		\hfill$\lozenge$
	\end{itemize}
\end{remark}

Next, we give precise statements of our second main result, which determine $m(\phi)$ and $M(\phi)$ for all $\phi\in\Lin(P_n)$ when $p=3$. For clarity, we split this into the cases of $n=3^k$ in Theorem~\ref{thm:b-primepower} and arbitrary $n$ in Theorem~\ref{thm:b-general}.

\begin{theorem}\label{thm:b-primepower}
	Let $k\in\N$ and $\phi(s)\in\Lin(P_{3^k})$. 
	Suppose $s=(\sfs_1,\dotsc,\sfs_k)\in\{0,1\}^k$ with $s\in U_k(z)$ for some $z\in\{0,1,\dotsc,k\}$. Then
	\begin{small}
		\[ M(s) = \begin{cases}
			3^k & \text{if}\ z=0,\\
			3^k-3^{k-f(s)} & \text{if}\ z>0.
		\end{cases} \]	
	\end{small}
	Moreover, if $\sfs_1=0$ then
	\begin{small}
		\[ m(s) = \begin{cases}
			3 & \text{if}\ s=(0,0),\\
			3^k-2 & \text{if}\ z=0\ \text{and}\ k\ne 2,\\
			2 & \text{if}\ s=(0,1),\\
			3^k-3^{k-f(s)}-1+\delta_{f(s),k} & \text{if}\ z=1\ \text{and}\ k\ne 2,\\
			3^k-3^{k-f(s)}-3^{k-g(s)} & \text{if}\ z\ge 2.
		\end{cases} \]
	\end{small}
	If $\sfs_1=1$ and $z=F(s)$, then
	\begin{small}
		\[ m(s) = \begin{cases}
			2 & \text{if}\ F(s)=k=2\ \ (\text{i.e.}\ s=(1,1)),\\
			\tfrac{3^k+1}{2} & \text{if}\ F(s)=k\ne 2,\\
			\tfrac{3^k-1}{2} & \text{if}\ F(s)=k-2,\\
			\tfrac{3^k+3^{k-F(s)}}{2}-2 & \text{otherwise},
		\end{cases} \]
	\end{small}
	while if $\sfs_1=1$ and $z>F(s)$ then
	\begin{small}
		\[ m(s) = \begin{cases}
			\tfrac{3^k+3^{k-F(s)}}{2} - 3^{k-G(s)} & \text{if}\ z=F(s)+1,\\
			\tfrac{3^k+3^{k-F(s)}}{2} - 3^{k-G(s)} - 3^{k-H(s)} & \text{if}\ z\ge F(s)+2.
		\end{cases} \]
	\end{small}
\end{theorem}

\begin{remark}
	Note that the general form of $m(s)$ is $3^k$ minus $3^{k-d}$ where $d$ runs over the positions of certain leftmost 1s in $s$. For example, $\tfrac{3^k+3^{k-F(s)}}{2}=3^k-3^{k-1}-3^{k-2}-\cdots-3^{k-F(s)}$. \hfill$\lozenge$
\end{remark}

\begin{theorem}\label{thm:b-general}
	Let $n\in\N$ with 3-adic expansion $n=\sum_{i=1}^t a_i\cdot 3^{n_i}$. Suppose $\phi\in\Lin(P_n)$ corresponds to $\{s_1,\dotsc,s_R\}$. Then $M(\phi) = \sum_{i=1}^R M(s_i)$. Moreover,
	\[ m(\phi) = N(\phi) \]
	unless (at least) one of the following holds:
	\begin{itemize}
		\item $R=1$, $\phi=\phi(s)$, $\sigma(s)\in\{1,10,11,5,6,7\}$, in which case $m(\phi)$ is given by Theorem~\ref{thm:b-primepower}.
		\item $R=2$, $\phi=\phi(s_1)\times\phi(s_2)$ where $\{\sigma(s_1),\sigma(s_2)\}=\{11,21\}$, in which case $m(\phi)=\tfrac{n}{2}-1$.
		\item $R=2$, $\phi=\phi(s)\times\triv_{P_3}$ where $\sigma(s)=21$, in which case $m(\phi)=\tfrac{n}{2}-1$.
		\item $n<27$, in which case $m(\phi)$ can be seen from Table~\ref{table:m<27} (in all such instances, $m(\phi)\ge\tfrac{2n}{9}$).
		\item $\phi$ is quasi-trivial, and, writing $\phi=\phi'\times\psi$ where $\psi\in\Lin(P_l)$ for $l<27$ and $27\mid n-l$, one of the following holds, in which case $m(\phi)=N(\phi)-1$:
		\begin{itemize}
			\item[$\circ$] $\sigma(\phi')\in E$ and $\psi=\triv_{P_l}$;
			\item[$\circ$] $\sigma(\phi')\in F$ and $\psi\in\Psi_2$;
			\item[$\circ$] $\sigma(\phi')=(R,0,0,0,0,0,0)$ (i.e.~$\phi=\triv_{P_n}$) and $\psi\in\Psi_1$;
			\item[$\circ$] $\sigma(\phi')=(R-1,0,0,0,1,0,0)$ and $\psi=\phi(1,0)\cdot\triv_{P_{l-9}}$, $l\ge 9$; or
			\item[$\circ$] $\sigma(\phi')=(R-1,1,0,0,0,0,0)$ and $\psi=\phi(1,0)\cdot\triv_{P_{l-9}}$, $l\ge 9$.
		\end{itemize}
	\end{itemize}
\end{theorem}

\begin{example}
	Fix $p=3$. \textit{(i)} Let $n=3^4$. We consider all $s\in\{0,1\}^4$ and tabulate the previously discussed data for $s$. Below, $qt$ means `quasi-trivial' and $a.c.n.o.t.p$ for `and contains no other thin partitions'.
	
	\begin{footnotesize}
		\[ \begin{array}{c|c|ccccc|cc|c|l}
			s & \sigma(s) & f(s) & g(s) & F(s) & G(s) & H(s) & m(s) & M(s) & qt? & \Omega(s)\\
			\hline
			0000 & 1 & - & - & - & - & - & 79 & 81 & \checkmark & \cP(81)\setminus\{(80,1)\}^\circ\\
			0001 & 2 & 4 & - & - & - & - & 80 & 80 & \checkmark & \cB_{81}(80)\\
			0010 & 31 & 3 & - & - & - & - & 77 & 78 & \checkmark & \cB_{81}(77)\sqcup\{(78,\mu)\mid\mu\in\Omega(0)\}^\circ\\
			0011 & 22 & 3 & 4 & - & - & - & 77 & 78 & \times & \cB_{81}(77)\sqcup\{(78,\mu)\mid\mu\in\Omega(1)\}^\circ\\
			0100 & 30 & 2 & - & - & - & - & 71 & 72 & \checkmark & \cB_{81}(71)\sqcup\{(72,\mu)\mid\mu\in\Omega(00)\}^\circ\\
			0101 & 22 & 2 & 4 & - & - & - & 71 & 72 & \times & \cB_{81}(71)\sqcup\{(72,\mu)\mid\mu\in\Omega(01)\}^\circ\\
			0110 & 22 & 2 & 3 & - & - & - & 69 & 72 & \times & \supseteq\cB_{81}(69)\ \text{a.c.n.o.t.p}\\
			0111 & 22 & 2 & 3 & - & - & - & 69 & 72 & \times & \supseteq\cB_{81}(69)\ \text{a.c.n.o.t.p} \\
			1000 & 5 & 1 & - & 1 & - & - & 52 & 54 & \checkmark &  {\scriptstyle \cB_{81}(54)\setminus\{(54,26,1),(54,2,1^{25}),(53,28),(53,1^{28})\}^\circ} \\
			1001 & 22 & 1 & 4 & 1 & 4 & - & 53 & 54 & \times & \cB_{81}(53)\sqcup\{(54,\mu)\mid\mu\in\Omega(001)\}^\circ\\
			1010 & 22 & 1 & 3 & 1 & 3 & - & 51 & 54 & \times & \supseteq\cB_{81}(51)\ \text{a.c.n.o.t.p} \\
			1011 & 22 & 1 & 3 & 1 & 3 & 4 & 50 & 54 & \times & \supseteq\cB_{81}(50)\ \text{a.c.n.o.t.p} \\
			1100 & 7 & 1 & 2 & 2 & - & - & 40 & 54 & \times & (\text{see Lemma~\ref{lem:110x} (i)}) \\
			1101 & 22 & 1 & 2 & 2 & 4 & - & 44 & 54 & \times & \supseteq\cB_{81}(44)\ \text{a.c.n.o.t.p} \\
			1110 & 11 & 1 & 2 & 3 & - & - & 40 & 54 & \times & (\text{see Lemma~\ref{lem:111} (i)}) \\
			1111 & 21 & 1 & 2 & 4 & - & - & 41 & 54 & \times & \supseteq\cB_{81}(41)\ \text{a.c.n.o.t.p} \\
			\hline
		\end{array} \]
	\end{footnotesize}

	From Theorem~\ref{thm:a}, we know $\Omega(s)$ exactly if $\phi(s)$ is quasi-trivial. We are able to determine $\Omega(s)$ completely for $s\in\{0011, 0101, 1001\}$ although these are non-quasi-trivial because $M(s)-m(s)=1$, so the result follows from Lemma~\ref{lem: 9.8}. That $\Omega(s)$ contains $\cB_{81}(m(s))$ and no other thin partitions for the $s$ indicated follows from Propositions~\ref{prop:4.8} and~\ref{prop:case2}. 
	
	\smallskip
	
	\textit{(ii)} Let $n=3^4+2\cdot 3^2=99$. Each $\phi(\underline{\textbf{s}})\in\Lin(P_{99})$ is labelled by $\phi(\underline{\textbf{s}})=\phi(r)\times\phi(s)\times\phi(t)$ for some $r\in\{0,1,2\}^4$ and $s,t\in\{0,1,2\}^2$. To give some examples, we consider $\phi(\underline{\textbf{s}})$ of form $\phi(0^4)\times\phi(s)\times\phi(t)$ with $s,t\in\{0,1\}^2$: 
	
	\begin{footnotesize}
		\[ \begin{array}{cc|cc|c|cc}
			0^4, s, t && N(0^4), N(s), N(t); & N(\phi(\underline{\textbf{s}})) & qt? && \Omega(s)\\
			\hline
			0^4, 00, 00 && 81, 9, 9; & 99 & \checkmark && \cP(99)\\
			0^4, 00, 01 && 81, 9, 8; & 98 & \checkmark && \cB_{99}(98)\\
			0^4, 01, 01 && 81, 8, 8; & 97 & \checkmark && \cB_{99}(97)\\
			0^4, 00, 10 && 81, 9, 6; & 96 & \checkmark && \cB_{99}(95)\sqcup\{(96,\mu)\mid\mu\in\Omega(\triv_{P_3})\}^\circ\\
			0^4, 10, 01 && 81, 6, 8; & 95 & \checkmark && \cB_{99}(94)\sqcup\{(95,\mu)\mid\mu\in\Omega(\triv_{P_4})\}^\circ\\
			0^4, 10, 10 && 81, 6, 6; & 93 & \checkmark && \cB_{99}(92)\sqcup\{(93,\mu)\mid\mu\in\Omega(\triv_{P_6})\}^\circ\\
			0^4, 00, 11 && 81, 9, 5; & 95 & \times && \cB_{99}(95)\subseteq\Omega(s)\subseteq\cB_{99}(96)\\
			0^4, 01, 11 && 81, 8, 5; & 94 & \times && \cB_{99}(94)\subseteq\Omega(s)\subseteq\cB_{99}(95)\\
			0^4, 10, 11 && 81, 6, 5; & 92 & \times && \cB_{99}(92)\subseteq\Omega(s)\subseteq\cB_{99}(93)\\
			0^4, 11, 11 && 81, 5, 5; & 91 & \times && \cB_{99}(91)\subseteq\Omega(s)\subseteq\cB_{99}(93)\\
			\hline
		\end{array} \]
	\end{footnotesize}

	From Theorem~\ref{thm:a} we know $\Omega(\underline{\textbf{s}})$ exactly if $\phi(\underline{\textbf{s}})$ is quasi-trivial. In the non-quasi-trivial cases, we observe that almost all of $\Omega(\underline{\textbf{s}})$ is determined since $M(\underline{\textbf{s}})-m(\underline{\textbf{s}})$ is small.
	\hfill$\lozenge$
\end{example}

\medskip

%========================================================================
\subsection{Representation theory of $S_n$ and the Littlewood\textendash Richardson rule}\label{sec:sn-lr}

We refer the reader to \cite{J,JK} for further detail.
For each $n\in\N$, $\Irr(S_n)$ is naturally parametrised by $\cP(n)$. For a partition $\lambda\in\cP(n)$, also written $\lambda\vdash n$, its corresponding character is denoted by $\chi^\lambda\in\Irr(S_n)$. We note that $\triv_{S_n}$ corresponds to $(n)$, and the sign character to $(1^n)$ \cite[2.1.7]{JK}. When clear from context, we sometimes abbreviate $\chi^\lambda$ to simply $\lambda$, for instance %wreath product characters 
$\cX(\chi^\lambda;\chi^\mu)$ %involving characters of symmetric groups 
to $\cX(\lambda;\mu)$.

The Young diagram associated to the partition $\lambda=(\lambda_1,\lambda_2,\dotsc,\lambda_k)$ is defined by $[\lambda]:=\{(i,j)\in\N\times\N\ |\ 1\leq i\leq k,\ 1\leq j\leq\lambda_i\}$.
Commonly, Young diagrams are drawn using boxes:
\[ \text{e.g.}\quad \lambda=(3,2),\qquad [\lambda]={\tiny\yng(3,2)}\ , \]
viewed in matrix orientation with $(1,1)$ upper-leftmost and $(i,j)$ lying in row $i$ and column $j$. 

%For $\lambda\in\cP(n)$, we define $\lambda^-=\{\mu\vdash n-1\mid \mu\subseteq\lambda \}$ and $\lambda^+=\{\nu\vdash n+1\mid \lambda\subseteq\nu\}$. The well-known Branching Theorem (see \cite[\textsection 9]{J}, for example) gives that
%\begin{equation}\label{eqn:branching}
%	\chi^\lambda\down^{S_n}_{S_{n-1}} = \sum_{\mu\in\lambda^-} \chi^\mu\qquad\text{and}\qquad\chi^\lambda\up^{S_{n+1}}_{S_n} = \sum_{\nu\in\lambda^+} \chi^\nu.
%\end{equation}

A partition $\lambda\in\cP(n)$ is a \textit{hook} if $\lambda=(n-x,1^x)$ for some $x\in[\overline{n}]$. Equivalently, $\lambda$ is a hook if $(2,2)\notin[\lambda]$.
The following definition of \textit{thin} partitions will play a crucial role in the descriptions of $\Omega(\phi)$ for $\phi\in\Lin(P_n)$ in our main theorems.

\begin{definition}\label{def:thin}
	\begin{itemize}
		\item[(i)] We call a partition $\lambda$ \emph{thin} if $\lambda$ is a hook, or a partition with $l(\lambda)\le 2$ or $\lambda_1\le 2$.
		
		\item[(ii)] For $m\le n\in\N$, we denote the hook partition $(m,1^{n-m})$ by $\hook{n}{m}$. 
		
		When $m\ge\tfrac{n}{2}$, we denote the partition $(m,n-m)$ by $\tworow{n}{m}$. 
		
		That is, $\hook{n}{m}$ and $\tworow{n}{m}$ are the \textbf{h}ook, respectively \textbf{t}wo-row, partitions of $n$ whose first row has length $m$. (Strictly speaking, by `two-row partition' we actually mean it has at most two rows, since if $m=n$ then $\tworow{n}{m}=(n)$.) %(When $n$ is understood, we may omit the subscript $_n$.)
	\end{itemize}
\end{definition}

%Let $m,n\in\N$. For $\mu\vdash m$ and $\nu\vdash n$, the Littlewood\textendash Richardson rule (see \cite[Chapter 16]{J}) describes the decomposition of the induced character
%\[ (\chi^\mu\times\chi^\nu)\up^{S_{m+n}}_{S_m\times S_n} \]
%into irreducible constituents.
%Before we recall the Littlewood\textendash Richardson rule, we introduce some notation and technical definitions.

We record some results on constituents of character restrictions related to wreath product groups that we will need later.

\begin{theorem}[{\cite[Corollary 9.1]{PW}}]\label{thm:PW9.1}
	Let $m,n\in\N$, $\mu\vdash m$ and $\nu\vdash n$. The lexicographically greatest $\lambda\vdash mn$ such that $\chi^\lambda\mid\cX(\mu;\nu)\up^{S_{mn}}_{S_m\wr S_n}$ is $\lambda=(n\mu_1,\dotsc,n\mu_{l(\mu)-1}, n(\mu_{l(\mu)}-1)+\nu_1,\nu_2,\dotsc,\nu_{l(\nu)})$. Moreover, $\chi^\lambda$ occurs with multiplicity 1.
\end{theorem}

%\begin{lemma}\label{thm:dBPW1.5}
%	Let $a,b,c\in\N$. Then $\chi^{(ab-1,ac+1)}\mid \cX\big((b,c);(a-1,1)\big)\up^{S_{a(b+c)}}_{S_{b+c}\wr S_a}$.
%\end{lemma}
%\begin{proof}
%	This is a special case of \cite[Theorem 1.5]{dBPW}.
%\end{proof}

\begin{theorem}[{\cite[Theorem 3.5]{GTT}}]\label{thm:GTT3.5}
	Let $p$ be an odd prime and $k\in\N$. Let $K:=S_{p^{k-1}}\wr S_p$ and $\chi\in\Irr_{p'}(S_{p^k})$. Then $\chi\down_K$ has a unique irreducible constituent $\chi^\ast$ lying in $\Irr_{p'}(K)$, and it appears with multiplicity 1. The map $\chi\mapsto\chi^\ast$ gives a bijection $\Irr_{p'}(S_{p^k})\to\Irr_{p'}(K)$.
	More precisely, $\chi=\chi^\lambda$ where $\lambda\vdash p^k$ is a hook. If $\lambda=(p^k-(mp+x),1^{mp+x})$ for some $x\in[\overline{p}]$, then $\chi^\ast\in\{\cX(\mu;\nu_1),\cX(\mu;\nu_2)\}$ where $\mu=(p^{k-1}-m,1^m)$, $\nu_1=(p-x,1^x)$ and $\nu_2=(x+1),1^{p-1-x})$.
\end{theorem}

Given partitions $\lambda$ and $\mu$, we say that $\mu\subseteq \lambda$ if $[\mu]\subseteq [\lambda]$. 
By a skew shape $\gamma$ we mean a difference of Young diagrams $[\lambda\setminus\mu]:=[\lambda]\setminus[\mu]$ %for some partitions $\lambda$ and $\mu$ with $[\mu]\subseteq[\lambda]$
where $\mu\subseteq\lambda$, and $|\gamma|:=|\lambda|-|\mu|$. %By convention, the highest row of $[\lambda]$ for a partition $\lambda$ is numbered 1, but the highest row of a skew shape $\gamma=[\lambda\setminus\mu]$ need not be the highest row of $[\lambda]$.

\begin{definition}\label{def:LR-rule}
	Let $\lambda=(\lambda_1,\dotsc,\lambda_k)\in\cP(n)$ and let $\cC=(c_1,\dotsc,c_n)$ be a sequence of positive integers. We say that $\cC$ has \emph{weight} $\lambda$ if $|\{i\in[n] \mid c_i=j\}| = \lambda_j$ for all $j\in[k]$. We say that an element $c_j$ of $\cC$ is \emph{good} if $c_j=1$ or if
	\[ |\{i\in [j-1] \mid c_i=c_j-1\}|>|\{i\in [j-1]\mid c_i=c_j\}|. \]
	Finally, we say that the sequence $\cC$ is \emph{good} if $c_j$ is good for every $j\in[n]$.
\end{definition}

\begin{theorem}[Littlewood\textendash Richardson Rule]\label{theorem:LR}
	Let $m,n\in\N$, let $\mu\vdash m$ and $\nu\vdash n$. Then
	\[ (\chi^\mu\times\chi^\nu)\up^{S_{m+n}}_{S_m\times S_n} = \sum_{\lambda\vdash m+n} c^\lambda_{\mu\nu}\cdot \chi^\lambda \]
	where $c^\lambda_{\mu\nu}$ equals the number of ways to replace the boxes of $[\lambda\setminus\mu]$ by natural numbers such that
	\begin{itemize}
		\item[(i)] the sequence obtained by reading the numbers from right to left, top to bottom is a good sequence of weight $\nu$,
		\item[(ii)] the numbers are weakly increasing along rows, and
		\item[(iii)] the numbers are strictly increasing down columns.
	\end{itemize}
\end{theorem}

\begin{definition}\label{def:LR-filling}
	Let $\nu$ be a partition and $\gamma$ be a skew shape such that $|\gamma|=|\nu|$. 
	\begin{itemize}
		\item[(i)] We call a way of replacing the boxes of $\gamma$ by numbers satisfying the conditions of Theorem~\ref{theorem:LR} a \emph{Littlewood--Richardson (LR) filling of $\gamma$ of weight $\nu$}.
		\item[(ii)] Let $\cLR(\gamma)$ denote the set of all possible weights of LR fillings of $\gamma$.
	\end{itemize}
\end{definition}
%Clearly every skew shape has at least one LR filling.

The \textit{Littlewood--Richardson coefficients} $c^\lambda_{\mu\nu}$ (sometimes notated as $c^\lambda_{\mu,\nu}$ for clarity) are symmetric, i.e.~$c^\lambda_{\mu\nu}=c^\lambda_{\nu\mu}$. Moreover, clearly $c^\lambda_{\mu\nu}>0$ implies that $\lambda_1\le \mu_1+\nu_1$ and $l(\lambda)\le l(\mu)+l(\nu)$.

We can also define \emph{iterated Littlewood--Richardson coefficients} $c^{\lambda}_{\mu^1,\dotsc,\mu^r}$ as follows. Let $r\in\N$ and $\mu^1,\dotsc,\mu^r$ be partitions, and let $\lambda\vdash n:= |\mu^1|+\cdots+|\mu^r|$. Then 
\[ c^{\lambda}_{\mu^1,\dotsc,\mu^r} := \langle \chi^\lambda, (\chi^{\mu^1}\times\cdots\times\chi^{\mu^r}) \up^{S_n}_{S_{|\mu^1|}\times\cdots\times S_{|\mu^r|}} \rangle. \]
When $r=2$, these are the usual Littlewood--Richardson coefficients as in Theorem~\ref{theorem:LR}, and letting $m=|\mu^1|+\cdots+|\mu^{r-1}|$ when $r\ge 2$, it is easy to see that
\[ c^{\lambda}_{\mu^1,\dotsc,\mu^r} = \sum_{\gamma\vdash m} c^{\gamma}_{\mu^1,\dotsc,\mu^{r-1}} \cdot c^\lambda_{\gamma,\mu^r}. \]
The iterated Littlewood--Richardson coefficients are also symmetric under any permutation of the partitions $\mu^1,\dotsc,\mu^r$.
An \emph{iterated Littlewood--Richardson (LR) filling of $[\lambda]$ by $\mu^1,\dotsc,\mu^r$} is a way of replacing the boxes of $[\lambda]$ by numbers defined recursively as follows: if $r=1$ then $[\lambda]=[\mu^1]$ has a unique LR filling, which is of weight $\mu^1$; if $r\ge 2$ then we mean an iterated LR filling of $[\gamma]$ by $\mu^1,\dotsc,\mu^{r-1}$ together with an LR filling of $[\lambda\setminus\gamma]$ of weight $\mu^r$ (for some $\gamma\subseteq\lambda$ such that this is possible).

%We record some definitions and results which will be useful later.

\begin{lemma}[{\cite[Lemma 2.19]{GL2}}]\label{lem: iteratedLR}
	Let $a,b_1,\dotsc,b_a\in\N$. Let $\nu^1,\dotsc,\nu^a$ be partitions such that $b_i\ge |\nu^i|$ for all $i$ and let $c=|\nu^1|+\cdots+|\nu^a|$. %b_i\ge|\nu_i| and \sum |\nu_i|=c implies \sum b_i\ge c 
	Let $\mu\vdash c$ and let $\lambda=(b_1+b_2+\cdots+b_a,\mu)$. Then the iterated Littlewood--Richardson coefficients $c^\lambda_{(b_1,\nu^1),\dotsc,(b_a,\nu^a)}$ and $c^\mu_{\nu^1,\dotsc,\nu^a}$ are equal.
\end{lemma}

\begin{lemma}\label{lem:BK}\cite[Lemma 4.4]{BK}
	Let $\mu$ and $\gamma$ be partitions such that $[\gamma]\subsetneq[\mu]$. The following are equivalent:
	\begin{itemize}\setlength\itemsep{0.5em}
		\item[(i)] $|\cLR([\mu\setminus\gamma])|=1$;
		\item[(ii)] there is a unique Littlewood\textendash Richardson filling of $[\mu\setminus\gamma]$;
		\item[(iii)] $[\mu\setminus\gamma]\cong[\nu]$ or its $180^\circ$-rotation, for some partition $\nu\vdash|\mu|-|\gamma|$.
	\end{itemize}
\end{lemma}

\begin{definition}\label{def:star operator}
	For $n,m\in\N$ and $A\subseteq\cP(n)$, $B\subseteq\cP(m)$, let 
	\[ A\star B:=\{\lambda\vdash n+m \mid \exists\ \mu\in A,\ \nu\in B\ \text{such that}\ c^{\lambda}_{\mu\nu}>0 \}. \]
\end{definition}

Clearly $\star$ is commutative and associative. Next, let $n\in\N$ and let $n = \sum_{i=1}^t a_ip^{n_i}$ be its $p$-adic expansion. %, where $0\le n_1<\cdots<n_t$. 
Given $\phi\in\Lin(P_n)$, recall that we may write $\phi=\phi(\underline{\mathbf{s}})$ %= \phi(\mathbf{s}(1,1))\times\cdots\times\phi(\mathbf{s}(t,a_t))$ 
as in (\ref{eqn: index}).

\begin{lemma}[{\cite[Lemma 5.1]{GL2}}]\label{lem: omega star omega}
	Let $p$ be any prime. For all $n\in\N$ and $\phi(\underline{\mathbf{s}})\in\Lin(P_n)$, 
	\[ \Omega(\underline{\mathbf{s}})=\Omega(\mathbf{s}(1,1))\star\cdots\star\Omega(\mathbf{s}(i,j))\star\cdots\star\Omega(\mathbf{s}(t,a_t)).\]
\end{lemma}

Recall that $\cB_n(t)$ is the set of those partitions of $n$ whose Young diagrams fit inside an $t\times t$ square grid. This is clearly closed under conjugation of partitions, i.e.~$\cB_n(t)=\cB_n(t)^\circ$. The following key combinatorial proposition describes how to combine sets of this form using the operator $\star$, under suitable hypotheses on the parameters.
\begin{proposition}[{\cite[Proposition 3.2]{GL2}}]\label{prop: B star B}
	Let $n,n',t,t'\in\N$ with $\tfrac{n}{2}<t\le n$ and $\tfrac{n'}{2}<t'\le n'$. Then $\cB_n(t) \star\cB_{n'}(t') = \cB_{n+n'}(t+t')$.
\end{proposition}
We remark that the above cannot be extended to all $t\le n$ or $t'\le n'$: for example, %considering $(4,4)\in\cP(8)$ shows that 
$\cB_7(3)\star\cB_1(1)\ne\cB_8(4)$. %Next, 

\begin{definition}\label{def: Dset newOmegas}
	Let $q,m\in\N$ with $q\geq 2$ and let $\cB\subseteq\cP(m)$.
	Let $H=(S_{m})^{\times q}\leq S_{qm}$.
	Let $\cD(q, m, \cB)$ be the subset of $\cP(qm)$ consisting of all those partitions $\lambda\in\cP(qm)$ for which there exists $\mu_1,\mu_2,\ldots, \mu_{q}\in\cB$, not all equal, such that 
	\[ \chi^{\mu_1}\times\chi^{\mu_2}\times\cdots\times\chi^{\mu_q}\ \Big{|}\ \chi^\lambda\down_{H}. \]
	% can be phrased as $\chi^\lambda\in\Irr(S_{qm}\mid \chi^{\mu_1}\times\cdots\times\chi^{\mu_q})$
	Note if $\cB^\circ=\cB$, then $\cD(q,m,\cB)^\circ = \cD(q,m,\cB)$.
\end{definition}

%Under appropriate restrictions on the parameters, we have that $\cB_{qm}(qt-1) \subseteq \cD\big(q,m,\cB_m(t)\big)$, as described in the following proposition.

\begin{proposition}[{\cite[Proposition 3.7]{GL2}}]\label{prop: BinD}
	Let $m,t\in\N$ with $\tfrac{m}{2}+1<t\le m$. Let $q\in\N_{\ge 3}$. Then $\cB_{qm}(qt-1) \subseteq \cD\big(q,m,\cB_m(t)\big)$.
\end{proposition}

\medskip

%%%%%%%%%%%%%%%%%%%%%%%%%%%%%%%%%%%%%%%%%%%%%%%%%%%%%%%%%%%%%%%%%%%%%%%%%
\subsection{Further combinatorics of Littlewood--Richardson coefficients}\label{sec:LR}

Recall that $[\lambda]$ denotes the Young diagram of $\lambda$. Let $\gamma$ and $\delta$ be skew shapes (this includes Young diagrams of partitions). We write $\gamma\cong\delta$ if $\gamma$ is a translation of $\delta$ in the plane, and denote by $\gamma^c$ the $180^\circ$-rotation of $\gamma$ (up to translation). We also write $\gamma\in A$ if $\gamma\cong\alpha$ for some $\alpha\in A$, where $A$ is a set of skew shapes.

Since $\Omega(\triv_{P_n})$ has the form $\cB_n(n-2)\cup \{(n)\}^\circ$ for almost all $n=p^k$ when $p$ is an odd prime by Theorem~\ref{thm:GL1A}, we will need the following.

\begin{lemma}\label{lem: 10.3}
	Let $n,m,t\in\N$ and suppose that $\tfrac{m}{2}<t\le m$. If $n\ge 5$, then
	\[ \cB_m(t) \star (\cB_n(n-2)\cup \{(n)\}^\circ) = \cB_{m+n}(t+n). \]
	In particular, $\cP(m+n) = \cP(m)\star(\cP(n)\setminus\{(n-1,1)\}^\circ)$.
	% n should be \ge 3 to get B_n(n-2)\cup{(n),(1^n)} = P(n)\setminus{(n-1,1)}^\circ
	If $n=3$, then 
	\[ \cB_m(t)\star \{(3),(1^3)\} = \begin{cases}
		\cB_{m+3}(t+3)\setminus \{(t+1,t+1)\}^\circ  & \mathrm{if}\ m=2t-1\\
		\cB_{m+3}(t+3) & \mathrm{otherwise}.
	\end{cases}\]
	In particular, $\cP(m+3)=\cP(m)\star\{(3),(1^3)\}$ unless $m=1$.
\end{lemma}
\begin{proof}
	The case of $n\ge 5$ is \cite[Lemma 3.3]{GL2}. Now suppose $n=3$ and let $X:=\cB_m(t)\star\{(3),(1^3)\}$. Clearly $X\subseteq\cB_{m+3}(t+3)$. Let $\lambda\in\cB_{m+3}(t+3)$; we wish to show that $\lambda\in X$ except if $m=2t-1$ and $\lambda\in\{(t+1,t+1)\}^\circ$.
	
	The assertion is clear if $t=1$ as then $m=1$, so now assume $t\ge 2$. Without loss of generality we may assume $\lambda_1\ge l(\lambda)$. %If $\lambda_1\le 2$ then $|\lambda|\le 4$ implies $m=t=1$, a contradiction since $t\ge 2$. Thus $\lambda_1\ge 3$. 
	Then $\lambda_1\ge 3$ (i.e.~$(3)\subseteq\lambda$), else $|\lambda|\le 4$ and $m=t=1$.
	If $\lambda\in\cB_{m+3}(t)$, then any $\mu\in\cLR(\lambda\setminus[(3)])\ne\emptyset$ satisfies $\mu\in\cB_m(t)$, whence $\lambda\in X$ since $\chi^\mu\times\chi^{(3)}\mid\chi^\lambda\down_{S_m\times S_3}$.
	
	Now suppose $\lambda\notin\cB_{m+3}(t)$, so $\lambda_1\in\{t+1,t+2,t+3\}$. If $\lambda_1=t+3$ then we must have $\lambda_2\le t$ and $l(\lambda)\le t$. %since t>m/2
	Hence $\mu=(\lambda_1-3,\lambda_2,\dotsc,\lambda_{l(\lambda)})\in\cB_m(t)$, and so $\lambda\in X$ since $\chi^\mu\times\chi^{(3)}\mid\chi^\lambda\down_{S_m\times S_3}$.
	
	Next suppose $\lambda_1=t+2$. If $\lambda_2\ge t$ then in fact $m=2t-1$ and $\lambda=(t+2,t)$, in which case $\chi^{(t,t-1)}\times\chi^{(3)}\mid\chi^\lambda\down_{S_m\times S_3}$ shows that $\lambda\in X$. Similarly if $l(\lambda)\ge t+1$ then $m=2t-1$ and $\lambda=(t+2,1^t)$, and $\lambda\in X$ by considering instead the partition $(t,1^{t-1})$. Otherwise, $l(\lambda)\le t$ and $\lambda_2\le t-1$. Hence $\mu=(\lambda_1-3,\lambda_2,\lambda_3\dotsc)$ is a partition and it belongs to $\cB_m(t)$, whence $\lambda\in X$.
	
	Finally suppose $\lambda_1=t+1$. If $\lambda_2\ge t-1$ then we can check the possibilities for $\lambda$ directly to see that $\lambda\in X$ except if $\lambda=(t+1,t+1)$. Indeed, $\lambda$ must be one of the following partitions: $(t+1,t+1)$, $(t+1,t,1)$, $(t+1,t)$, $(t+1,t-1)$, $(t+1,t-1,1)$, $(t+1,t-1,2)$ or $(t+1,t-1,1^2)$.
	If $l(\lambda)\ge t+1$ then we can again check directly that $\lambda\in X$, since $\lambda$ must be one of $(t+1,1^{t+1})$, $(t+1,1^t)$ or $(t+1,2,1^{t-1})$. Otherwise, $\lambda_2\le t-2$ and $l(\lambda)\le t$. Hence $\mu=(\lambda_1-3,\lambda_2,\dotsc)$ is a partition and it belongs to $\cB_m(t)$, whence $\lambda\in X$.
\end{proof}

\begin{lemma}\label{lem: P star 00}
	Let $t\in\N_{\ge 2}$. Then $\cP(t)\star\Omega(\triv_{P_9})=\cP(t+9)$.
\end{lemma}
\begin{proof}
	Let $\sum_{i\ge 0} a_i\cdot 3^i$ be the 3-adic expansion of $t$. First suppose $a_2\in\{0,1\}$. Then $\cP(t)\star\Omega(\triv_{P_9})\supseteq\Omega(\triv_{P_t})\star\Omega(\triv_{P_9})=\Omega(\triv_{P_{t+9}})$ by Lemma~\ref{lem: omega star omega}. Since $t\ge 2$, by Theorem~\ref{thm:GL1A}, $\Omega(\triv_{P_{t+9}})=\cP(t+9)$ unless $t+9=3^k$ for some $k\ge 3$, in which case $\Omega(\triv_{P_{t+9}})=\cP(t+9)\setminus\{(3^k-1,1)\}^\circ$. But then $(9)\subset(3^k-1,1)$ and $(9)\in\Omega(\triv_{P_9})$, so $(3^k-1,1)\in\cP(t)\star\Omega(\triv_{P_9})$. Since $\cP(t)\star\Omega(\triv_{P_9})$ is closed under conjugation, $\cP(t)\star\Omega(\triv_{P_9})=\cP(t+9)$.
	
	Otherwise $a_2=2$. Let $k=\min\{i\ge 3\mid a_i\ne 2\}$, so $t=\sum_{i\ge k} a_i\cdot 3^i + \sum_{i=2}^{k-1} 2\cdot 3^i + \sum_{i=0}^1 a_i\cdot 3^i$. We have that $\cP(18)\star\Omega(\triv_{P_9})=\cP(27)$, and clearly $\cP(2\cdot 3^i)\star\cP(3^i)=\cP(3^{i+1})$ for all $i$, so
	\[ \cP(t)\star\Omega(\triv_{P_9})\supseteq \Omega(\triv_{P_t})\star\Omega(\triv_{P_9}) = \left(\underset{\substack{i\ge k,\\i\in\{0,1\}}}{\star} \Omega(\triv_{P_{3^i}})^{\star a_i}\right) \star \cP(3^k) = \Omega(\triv_{P_u})\star\cP(3^k) \]
	by Lemma~\ref{lem: omega star omega}, where $u:=\sum_{i\ge k} a_i\cdot 3^i + \sum_{i=0}^1 a_i\cdot 3^i$. Notice $u\ne 9,10$ and $u+3^k=t+9$. If $u\notin\{4,6\}$ then $\Omega(\triv_{P_u})=\cP(u)$ or $\cP(u)\setminus\{(u-1,1)\}^\circ$ by Theorem~\ref{thm:GL1A}, so $\Omega(\triv_{P_u})\star\cP(3^k)=\cP(u+3^k)$ by Lemma~\ref{lem: 10.3}. If $u\in\{4,6\}$, then since $k\ge 3$ for any $\lambda\vdash u+3^k$ there exists $\gamma\vdash u$ such that $\gamma\subset\lambda$ and $\gamma\ne(2,2)$ if $u=4$, while $\gamma\ne(3,2,1)$ if $u=6$. Thus $\Omega(\triv_{P_u})\star\cP(3^k)=\cP(u+3^k)$.
\end{proof}

%\begin{proof}
%	Let $\sum_{i\ge 0} a_i\cdot 3^i$ be the 3-adic expansion of $t$. Let $r=\sum_{i\ne 2}a_i\cdot 3^i$ and $s=9a_2\in\{0,9,18\}$.
%	%Let $t=r+s$ with $r=\sum_{i\ge 0} a_i\cdot3^i$, $a_2=0$, and $s\in\{0,9,18\}$. 
%	Then by Lemma~\ref{lem: omega star omega},
%	\begin{align*}
%		\cP(t)\star\Omega(\triv_{P_9}) &\supseteq \Omega(\triv_{P_t})\star\Omega(\triv_{P_9}) = \Omega(\triv_{P_r})\star\Omega(\triv_{P_s})\star\Omega(\triv_{P_9})\\
%		&= \begin{cases}
%			\Omega(\triv_{P_r})\star\Omega(\triv_{P_9}) & \mathrm{if}\ s=0,\\
%			\Omega(\triv_{P_r})\star\cP(s+9) & \text{otherwise}.
%		\end{cases}\\
%		&= \begin{cases}
%			\Omega(\triv_{P_{r+9}})\ne\cP(r+9) & \text{if}\ r\in\{0,1\}\ \text{and}\ s=0,\\
%			\Omega(\triv_{P_{t+9}}) = \cP(t+9) & \mathrm{otherwise}.
%		\end{cases}
%	\end{align*}
%	However, $t\ge 2$ implies that we cannot have $r\in\{0,1\}$ and $s=0$, so $\cP(t)\star\Omega(\triv_{P_9})=\cP(t+9)$.% as desired.
%\end{proof}

\begin{lemma}\label{lem: LR tworow filling}
	Let $\gamma$ be a skew shape. Suppose $\nu\in\cLR(\gamma)$ for some $\nu=(s,t)$ with $s\ge t\ge 2$. 
	\begin{itemize}
		\item[(a)] Then either $(s,t-1,1)\in\cLR(\gamma)$, or all $t$ 2s appear in the same row of $\gamma$ for every Littlewood--Richardson filling of $\gamma$ of weight $\nu$.
		\item[(b)] If $s=t+1$, then either $\gamma\in \{[(t+1,t)], [(t+1,t)]^c \}$, or $\cLR(\gamma)\cap \{ (t+1,t-1,1),(t+2,t-1),(t,t,1) \}\ne\emptyset$.
		%\cB_{2t+1}(t+2)\setminus\{(t+1,t),(t+1,1^t)\}^\circ \ne \emptyset$.
	\end{itemize}
	
\end{lemma}
\begin{proof}
	\textbf{(a)} Let $\mathsf{F}$ be a Littlewood--Richardson (LR) filling of $\gamma$ of weight $\nu$. We show that if not all $t$ 2s appear in the same row, then $(s,t-1,1)\in\cLR(\gamma)$. Indeed, if there are at least two rows which contain 2s, then replace the rightmost 2 on the lowest such row by a 3 to produce a filling, call it $\mathsf{F}'$. Clearly $\mathsf{F'}$ satisfies conditions (ii) and (iii) of Theorem~\ref{theorem:LR}, and it satisfies condition (i) since there is a 2 in a strictly highest row (hence earlier in the reading order from right to left, top to bottom). Thus $\mathsf{F}'$ is an LR filling of $\gamma$ of weight $(s,t-1,1)$.
	
	\textbf{(b)} This follows from part (a) and a similar replacement argument.
\end{proof}

\begin{lemma}\label{lem: LR hook filling}
	Let $\gamma$ be a skew shape. Suppose $\nu\in\cLR(\gamma)$ for some $\nu=(s,1^t)$ with $s\ge 3$. 
	\begin{itemize}
		\item[(a)] Then either $(s-1,2,1^{t-1})\in\cLR(\gamma)$, or at least $s-1$ 1s appear in the same row of $\gamma$ for every Littlewood--Richardson filling of $\gamma$ of weight $\nu$.
		\item[(b)] If $s=t+1$, then either $\gamma\in \{[(t+1,1^t)],[(t+1,1^t)]^c \}$, or $\cLR(\gamma)\cap \{(t,2,1^{t-1}), (t+2,1^{t-1}), (t,1^{t+1}) \} \ne\emptyset$.
		%\cB_{2t+1}(t+2)\setminus\{(t+1,t),(t+1,1^t)\}^\circ \ne \emptyset$.
	\end{itemize}
\end{lemma}
\begin{proof}
	\textbf{(a)} Let $\mathsf{F}$ be an LR filling of $\gamma$ of weight $\nu$. We show that if no $s-1$ 1s appear in the same row, then $(s-1,2,1^{t-1})\in\cLR(\gamma)$. The entries of $\mathsf{F}$ are $s$ 1s, one 2, one 3 and so on, up to one $t+1$. By Theorem~\ref{theorem:LR}, the reading order (right to left, top to bottom) of $\mathsf{F}$ must contain the $2,3,\dotsc,t+1$ in that order, with the $s$ 1s dispersed throughout. Let $x$ be the number of 1s appearing before the unique 2 in the reading order of $\mathsf{F}$. 
	
	First suppose $x=1$. By assumption, there is a 1 in the reading order that is neither the first nor second of the $s$ 1s such that in $\gamma$ it is the rightmost 1 in its row. Moreover, it is not directly above the unique 2, since $x=1$. Thus we may replace this 1 by a 2 to produce a filling $\mathsf{F}'$ satisfying conditions (i)--(iii) of Theorem~\ref{theorem:LR}, whence $(s-1,2,1^{t-1})\in\cLR(\gamma)$.
	
	If $2\le x\le s-1$, then similarly replace a 1 by a 2, this time choosing the first 1 that appears later in the reading order than the unique 2 (it must be the rightmost 1 in its row in $\gamma$).
	
	If $x=s$, let $1\le l\le s$ be such that the $s$-th, $(s-1)$-th, $\dotsc$, $l$-th 1s in the reading order lie in the same row in $\gamma$, but the $(l-1)$-th 1 does not. Then we may replace the $l$-th 1 by a 2 (the assumptions imply that $l\ne 1$), unless $l=s$ \emph{and} the $s$-th 1 appears directly above the unique 2 in $\gamma$. In this case, the $s$-th 1 is the unique 1 in its connected component within $\gamma$; letting $j$ be such that the $(s-1)$-th, $\dotsc$, $j+1$-th, $j$-th 1 appear in the same row but the $(j-1)$-th 1 does not, we may replace the $j$-th 1 by a 2 since the assumptions imply that $j\ne 1$.
	
	\textbf{(b)} This follows from part (a) and a similar replacement argument.
\end{proof}

\bigskip

%%%%%%%%%%%%%%%%%%%%%%%%%%%%%%%%%%%%%%%%%%%%%%%%%%%%%%%%%%%%%%%%%%%%%%%%%
\section{The prime power case}\label{sec:prime-power}

The aim of this section is to prove our main results when $n$ is a power of 3, which will form the central stepping stone for arbitrary $n$. That is, we will prove Theorem~\ref{thm:a} in the case $n=3^k$ (this will be Theorem~\ref{thm:a-prime power}) which determines $\Omega(\phi)$ for all quasi-trivial $\phi$, and also Theorem~\ref{thm:b-primepower} which describes $m(\phi)$ for all linear $\phi$.

First, let $p$ be any odd prime and let $k\in\N$. Recall $\Lin(P_{p^k}) = \{\phi(s)\mid s\in[\overline{p}]^k \}$ with $\triv_{P_{p^k}}=\phi(0,\dotsc,0)$. Since $\Omega(\triv_{P_{p^k}})$ was determined in Theorem~\ref{thm:GL1A}, we consider $\phi(s)\in\Lin(P_{p^k})$ where $s=(\sfs_1,\dotsc,\sfs_k)\ne(0,\dotsc,0)$. In other words, $s\in U_k(z)$ with $z>0$, and $f(s)$ is well-defined. We encourage the reader to recall Definition~\ref{def:sequences} where we introduced notation to refer to the positions of certain leftmost non-zero positions in $s$, which will turn out to govern the shape of $\Omega(s)$ and $m(s)$, and to Remark~\ref{rem:binary} where we observed that it will suffice in fact to consider $s\in\{0,1\}^k$ rather than $s\in[\overline{p}]^k$. Furthermore, \emph{thin} partitions were introduced in Definition~\ref{def:thin}, and will play an important role in determining both the bound $m(s)$ (by investigating which thin partitions are or are not contained in $\Omega(s)$) and the shape of $\Omega(s)$ itself.

Let $s^-=(\sfs_1,\dotsc,\sfs_{k-1})$. The core of our strategy is to induct on $k$, computing $m(s)$ and $\Omega(s)$ in terms of $m(s^-)$ and $\Omega(s^-)$. This is the `inductive step'; we will also need to understand $m(s)$ and $\Omega(s)$ for appropriate `base cases' of $s$. Both the inductive steps and the collection of necessary base cases $s$ turn out to be far more involved when $p=3$ than when $p\ge 5$ (compare Figure~\ref{fig:5} for $p\ge 5$ with Figure~\ref{fig:3} for $p=3$ below), although there are several useful lemmas which hold for all odd primes which we will now recall, namely \cite[Lemmas 4.1, 4.3, 4.6]{GL2}.%, before specialising to $p=3$.

\begin{lemma}\label{lem: DinOm}%[{\cite[Lemma 4.1]{GL2}}]
	Let $p$ be an odd prime and $k\in\N_{\ge 2}$. Let $s=(\sfs_1,\dotsc,\sfs_k)\in[\overline{p}]^k$ and $s^-=(\sfs_1,\ldots, \sfs_{k-1})$. Then $\cD(p,p^{k-1},\Omega(s^-))\subseteq \Omega(s)$.
\end{lemma}

\begin{lemma}\label{lem: 001}%[{\cite[Lemma 4.3]{GL2}}]
	Let $p$ be an odd prime and $k\in\N_0$. Suppose $s=(0,\dotsc,0,\sfx)\in [\overline{p}]^{k+1}$ where $\sfx\ne 0$. Then %$\big\langle \chi^{(p^{k+1}-1,1)}\down_{P_{p^{k+1}}}, \phi(s)\big\rangle = 1$.
	$Z^{(p^{k+1}-1,1)}_{\phi(s)}=1$.
	Moreover, $\Omega(s)=\cB_{p^{k+1}}(p^{k+1}-1)$, except if $(p,k)=(3,1)$ in which case $\Omega(0,1)=\cB_9(8)\setminus\{(3^3)\}$.
\end{lemma}

\begin{lemma}\label{lem: 9.8}%[{\cite[Lemma 4.6]{GL2}}]
	Let $p$ be an odd prime. Let $z\le k\in\N$, and let $s\in U_k(z)$ with $f(s)<k$. Then
	\[ \Omega(s) \cap\{\lambda\vdash p^k \mid \lambda_1=M(s)\}^\circ = \{(M(s),\mu) \mid \mu\in\Omega(\sfs_{f(s)+1},\dotsc,\sfs_k) \}^\circ.\]
	In particular, if $z\ge 2$ then $\Omega(s) \cap\{\lambda\vdash p^k\ |\ \lambda_1=M(s)\}^\circ$ contains no thin partitions.
\end{lemma}

The two main technical ingredients for the inductive step in computing $m(s)$ when $p\ge 5$ were \cite[Lemma 4.7, Proposition 4.8]{GL2}. We will need extensions of these for our current purposes, and we prove these below in Lemma~\ref{lem:4.7} and Proposition~\ref{prop:4.8}.

\begin{lemma}\label{lem:4.7}
	Let $p$ be an odd prime and $k\in\N$. Suppose $s=(\sfs_1,\dotsc,\sfs_k)\in U_k(1)$ and $\sfx\in[\overline{p}]$. If $p=3$, then further assume $\sfs_1=0$ and $k\ge 2$. Define $f:=f(s)$ and $q:=p^{k+1}-p^{k+1-f}$.  Then 
	\begin{itemize}%\setlength\itemsep{0.5em}
		\item[(a)] %$\Omega(s,x) = \cB_{p^{k+1}}(p^{k+1}-p^{k+1-f}-1)\sqcup \{(p^{k+1}-p^{k+1-f},\mu) : \mu\in\Omega(s_{f+1},\dotsc,s_k,x) \}^\circ$.
		$\Omega(s,\sfx) = \cB_{p^{k+1}}(q-1)\sqcup \{(q,\mu) \mid \mu\in\Omega(\sfs_{f+1},\dotsc,\sfs_k,\sfx) \}^\circ$.
		\item[(b)] Moreover, if $\sfx\ne 0$ and 
		%$\lambda\in\{ \tworow{p^{k+1}}{p^{k+1}-p^{k+1-f}-1}, \hook{p^{k+1}}{p^{k+1}-p^{k+1-f}-1} \}^\circ$, then $\langle\chi^\lambda\down_{P_{p^{k+1}}},\phi(s,x)\rangle\ge 2$.
		$\lambda\in\{\tworow{p^{k+1}}{q-1}, \hook{p^{k+1}}{q-1}\}^\circ$, then $Z^\lambda_{\phi(s,\sfx)}\ge 2$.
	\end{itemize}
\end{lemma}

\begin{proof}
	See \cite[Lemma 4.7]{GL2} if $p\ge 5$. When $p=3$, the statements may be checked directly for $s=(0,1)$ so now we may assume $k\ge 3$. Then the same proof by induction on $k-f$ as in \cite[Lemma 4.7]{GL2} holds, with the following observations and additions:
	\begin{itemize}
		\item the various inequalities restricting small values of $p^k$ in order to apply previous results are still satisfied, since $k\ge 3$ when $p=3$ gives sufficiently large values of $p^k$;% the case $f=k$ follows since $p^k>4$ still holds;
		%\item when $f<k$, Proposition~\ref{prop: BinD} with $t=p^k-p^{k-f}-1$ and $m=p^k$ may still be applied since $t>\frac{m}{2}+1$ follows from the assumption that $k\ge 3$ when $p=3$;
		\item in one step of the proof of (a), given $\mu\vdash p^{k+1-f}+p$ we wish to find some partition $\nu_1\in\Omega(\sfs_{f+1},\dotsc,\sfs_k)$ such that $\nu_1\subseteq\mu$. If $p^{k-f}\ne 9$ then this follows from Theorem~\ref{thm:GL1A}, while if $p^{k-f}=9$ then such $\nu_1$ can be found directly since $|\mu|=30$;
		\item in the following step, given $r\in[p-1]$ and $\mu\vdash p^{k+1-f}+p-r$ we claim that there exists $\nu\vdash rp^{k-f}$ such that $\nu\subseteq\mu$, and furthermore if $r=1$ then it is possible to choose $\nu\in\Omega(\sfs_{f+1},\dotsc,\sfs_k)$ by Theorem~\ref{thm:GL1A}. Additionally if $p=3$, $r=2$ and $k-f=1$ then we want to choose $\nu\in\cP(6)\setminus\{3,2,1\}$, which is always possible.		
	\end{itemize}
\end{proof}

\begin{proposition}\label{prop:4.8}
	Let $p$ be an odd prime and $k\in\N$. Suppose that $s\in[\overline{p}]^k$ satisfies the following:
	\begin{itemize}%\setlength\itemsep{0.5em}
		\item[(i)] $\tfrac{p^k}{2}+1<m(s)\le p^k-3$ and $\Omega(s)\setminus\cB_{p^k}(m(s))$ contains no thin partitions, and 
		\item[(ii)] $Z^\nu_{\phi(s)}\ge 2$ for $\nu\in\{\tworow{p^k}{m(s)}, \hook{p^k}{m(s)}\}$. %$\langle \tworow{p^k}{m(s)}\down_{P_{p^k}}, \phi(s)\rangle \ge 2$ and $\langle \hook{p^k}{m(s)}\down_{P_{p^k}}, \phi(s)\rangle \ge 2$.
	\end{itemize}
	Then for all $\sfx\in[\overline{p}]$, $\Omega(s,\sfx)\setminus\cB_{p^{k+1}}(p\cdot m(s))$ contains no thin partitions,
	\[ m(s,\sfx)=p\cdot m(s), \]
	and $Z^\mu_{\phi(s,\sfx)}\ge 2$ for $\mu\in\{\tworow{p^{k+1}}{p\cdot m(s)}, \hook{p^{k+1}}{p\cdot m(s)}\}$.
	%\[ \langle \tworow{p^{k+1}}{p\cdot m(s)}\down_{P_{p^{k+1}}}, \phi(s,x)\rangle \ge 2\quad \text{and}\quad \langle \hook{p^{k+1}}{p\cdot m(s)}\down_{P_{p^{k+1}}}, \phi(s,x)\rangle \ge 2. \]
\end{proposition}

\begin{remark}\label{rem:4.8}
	\begin{itemize}
		\item[(i)] The primary purpose of Proposition~\ref{prop:4.8} is to show that if $\lambda$ is the \textit{longest thin} partition of $p^k$ for which $Z^\lambda_{\phi(s)}\ge 2$, then $m(s,\sfx)=p\cdot m(s)$ for all $\sfx\in [\overline{p}]$.
		
		\item[(ii)] When $p\ge 5$ we did not need to specify that $m(s)\le p^k-3$ (see \cite[Proposition 4.8]{GL2}): this is since from Theorem~\ref{thm:GL1A} $\Omega(\triv_{P_{p^k}})=\cP(p^k)\setminus\{(p^k-1,1)\}^\circ$, so $M(\triv_{P_{p^k}})=p^k$ and note $(p^k-1,1)$ is thin, and $m(s)\le M(s) = p^k-p^{f(s)}\le p^k-5$ whenever $s\ne(0,\dotsc,0)$. Now including $p=3$, if we specify $m(s)\le p^k-3$ then Proposition~\ref{prop:4.8} follows from an identical proof to \cite[Proposition 4.8]{GL2}. This is so that $t:=p^k-m(s)$ satisfies $t>\tfrac{t}{2}+1$, in order to apply Proposition~\ref{prop: BinD}.\hfill$\lozenge$
	\end{itemize}
\end{remark}

Now that we have introduced Lemma~\ref{lem:4.7} and Proposition~\ref{prop:4.8}, we can remind the reader that the inductive step for $p\ge 5$ can be visualised using Figure~\ref{fig:5} below. For comparison, the analogous diagram displaying the relevant results which make up the inductive step when $p=3$ is given by Figure~\ref{fig:3}. 

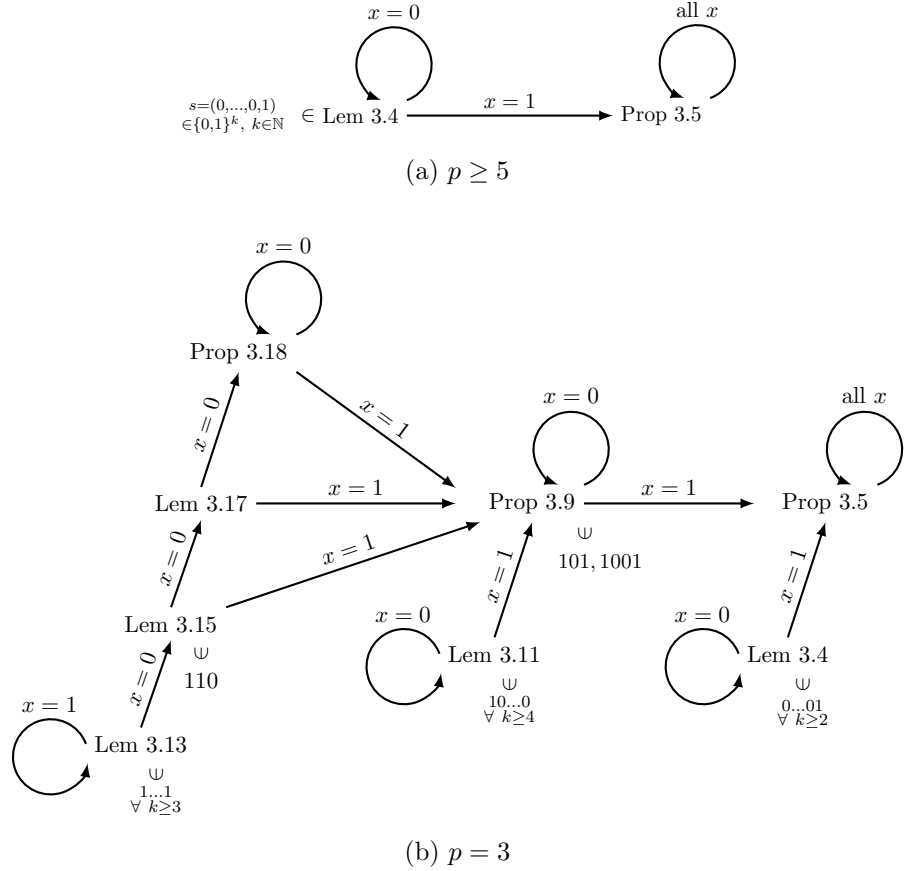
\begin{figure}[ht]
	\centering
	\begin{subfigure}{\textwidth}
		\centering
		\begin{tikzpicture}[scale=1, every node/.style={scale=0.75},thick]
			\draw (0,0) node[] (2) {Lem~\ref{lem:4.7}};
			\draw (4,0) node[] (1) {Prop~\ref{prop:4.8}};
			
			\draw [-latex] (2.north east) arc(-70:250:0.5) node[midway, above] {$x=0$};
			\draw [-latex] (1.north east) arc(-70:250:0.5) node[midway, above] {all $x$};
			
			\draw [-latex] (2.east) -- (1.west) node[midway, above] {$x=1$}; %x\ne 0
			
			\draw (-0.7,0) node[] {$\in$};
			%\draw (-1.7,0) node[] {$\substack{s=(0,\dotsc, 0,x)\\ \forall\ x\in[p-1],\ k\in\N}$};	
			\draw (-1.7,0) node[] {$\substack{s=(0,\dotsc, 0,1)\\ \in \{0,1\}^k,\ k\in\N}$};	
		\end{tikzpicture}
		\caption{$p\ge 5$}
		\label{fig:5}		
	\end{subfigure}

	\vspace{15pt}
	
	\begin{subfigure}{\textwidth}
		\centering
		\begin{tikzpicture}[scale=1, every node/.style={scale=0.8},thick]
			\draw (-3.5,2) node[] (3) {Prop~\ref{prop:case3}};
			\draw [-latex] (3.north east) arc(-70:250:0.5) node[midway, above] {$x=0$};
			\draw [-latex] (3.south east) -- (2.north west) node[midway,sloped, above] {$x=1$};
			
			\draw (-4,0) node[] (8) {Lem~\ref{lem:1100x}};
			\draw [-latex] (8.east) -- (2.west) node[midway,sloped,above] {$x=1$};
			\draw [-latex] (8.north) -- (3.south) node[midway,sloped,above] {$x=0$};
			
			\begin{scope}[xshift=4mm]
				\draw (0,0) node[] (2) {Prop~\ref{prop:case2}};
				\draw [-latex] (2.north east) arc(-70:250:0.5) node[midway, above] {$x=0$};
				\draw [-latex] (2.east) -- (1.west) node[midway, above] {$x=1$};
			\end{scope}
			
			\begin{scope}[xshift=-2mm]
				\draw (4.5,0) node[] (1) {Prop~\ref{prop:4.8}};
				\draw [-latex] (1.north east) arc(-70:250:0.5) node[midway, above] {all $x$};
			\end{scope}
			
			\draw (-4.4,-1.6) node[] (4) {Lem~\ref{lem:110x}};
			\draw [-latex] (4.north) -- (8.south) node[midway,sloped,above] {$x=0$};
			\draw [-latex] (4.north east) -- (2.south west) node[midway,sloped,above] {$x=1$};
			\draw (-4.2,-2) node[rotate=90,anchor=north] {$\in$};
			\draw (-4,-2.35) node[] {$110$};
			
			\draw (-4.8,-3.2) node[] (7) {Lem~\ref{lem:111}};
			\draw [-latex] (7.west) arc(20:350:0.5) node[midway, above right, yshift=0.65cm] {$x=1$};
			\draw [-latex] (7.north) -- (4.south) node[midway,sloped,above] {$x=0$};
			\draw (-4.8,-3.6) node[rotate=90,anchor=north] {$\in$};
			\draw (-4.6,-3.95) node[] {$\substack{1\dotsc1\\ \forall\ k\ge 3}$};
			
			\begin{scope}[xshift=4mm]
				\draw (-0.5,-2) node[] (5) {Lem~\ref{lem:1000}};
				\draw [-latex] (5.west) arc(20:350:0.5) node[midway, above right, yshift=0.65cm] {$x=0$};
				\draw [-latex] (5.north) -- (2.south) node[midway,sloped,above] {$x=1$};
				\draw (0.5,-0.4) node[rotate=90,anchor=north] {$\in$};
				\draw (0.9,-0.8) node[] {\small $101, 1001$};
				\draw (-0.5,-2.4) node[rotate=90,anchor=north] {$\in$};
				\draw (-0.3,-2.75) node[] {$\substack{10\dotsc 0\\ \forall\ k\ge 4}$};
			\end{scope}
			
			\begin{scope}[xshift=-2mm]
				\draw (4,-2) node[] (6) {Lem~\ref{lem:4.7}};
				\draw [-latex] (6.west) arc(20:350:0.5) node[midway, above right, yshift=0.65cm] {$x=0$};
				\draw [-latex] (6.north) -- (1.south) node[midway,sloped,above] {$x=1$};
				%\draw (5,-0.4) node[rotate=90,anchor=north] {$\in$};
				%\draw (5.2,-0.75) node[] {\small $011$};
				\draw (4,-2.4) node[rotate=90,anchor=north] {$\in$};
				\draw (4.2,-2.8) node[] {$\substack{0\dotsc 01\\ \forall\ k\ge 2}$};
			\end{scope}
		\end{tikzpicture}
		\caption{$p=3$}
		\label{fig:3}
	\end{subfigure}
	\caption{The inductive step of computing $m(s)$: $s\in A$ means $s$ satisfies the conditions of $A$, and $A\xrightarrow{\sfx=i} B$ means if $s$ satisfies the conditions of $A$ then $(s,i)$ satisfies the conditions of $B$.}
	\label{fig:flowcharts}
\end{figure}

\begin{remark}
	In both figures, we have omitted sequences of the form $s=(0,\dotsc,0)$ since such $\Omega(s)$ were determined in Theorem~\ref{thm:GL1A}. To briefly explain Figure~\ref{fig:5}: we may view Proposition~\ref{prop:4.8} as the main driver of the inductive step, since once a sequence $s$ satisfies the conditions of the proposition, then so does $(s,\sfx)$ for any $\sfx\in[\overline{p}]$, and the assertions of the proposition then govern the shape of $\Omega(s)$ and $m(s)$. 
	
	The sequences which form the base cases are of form $s=(0,\dotsc,0,1)$: however, Proposition~\ref{prop:4.8} cannot immediately be applied to such $s$ since condition (ii) on Sylow branching coefficients (SBCs) would not be not satisfied, by Lemma~\ref{lem: 001}. Nor can we apply the proposition to sequences of the form $s=(0,\dotsc,0,1,0,\dotsc,0)$, since then $\Omega(s)\setminus\cB_{p^k}(m(s))$ \textit{does} contain thin partitions; this is why we need Lemma~\ref{lem:4.7}. 
	
	In a similar vein -- because of the different combinations of SBCs, forms of $s$ and $\Omega(s)$ -- we will see in the remainder of this section that there are many more inductive steps and base cases when $p=3$ than for $p\ge 5$, and Figure~\ref{fig:3} illustrates how they all fit together. In the following example, we give a concrete example of how to use Figure~\ref{fig:3}. \hfill$\lozenge$
\end{remark}

\begin{example}\label{eg:m}
	Let $p=3$ and $s=(110001010)$, so $k=9$, $F(s)=2$, $G(s)=6$, $H(s)=8$ and $s\in U_k(z)$ where $z=4=F(s)+2$.
	\begin{itemize}
		\item[$\circ$] By Lemma~\ref{lem:110x} with $k=3$ and $s=(110)$, we can describe $\Omega(1100)$.
		
		\item[$\circ$] Hence by Lemma~\ref{lem:1100x} we can describe $\Omega(11000)$. In particular, $k=5$ and $s=(11000)$ then satisfy the conditions of Proposition~\ref{prop:case3} with $m=135$. 
				
		\item[$\circ$] Applying this, we can describe $\Omega(110001)$ and find $m(110001)=3m-1=3\cdot 135-1=404$.
		
		In particular, $(110001)$ satisfies the conditions of Proposition~\ref{prop:case2}.
		
		\item[$\circ$] Hence we can describe $\Omega(1100010)$ and find $m(1100010)=3^2\cdot 135-3=1212$.
		
		In particular, $(1100010)$ satisfies the conditions of Proposition~\ref{prop:case2}.
		
		\item[$\circ$] Hence we can describe $\Omega(11000101)$ and find $m(11000101)=3^3\cdot135-3^2-1=3635$.
		
		In particular, $(11000101)$ satisfies the conditions of Proposition~\ref{prop:4.8}.
		
		\item[$\circ$] Hence we can describe $\Omega(110001010)$, and finally find that $m(110001010) = 3^4\cdot135-3^3-1=10905$.
	\end{itemize}
	In particular, this tells us that $\cB_{3^9}(10905)\subseteq\Omega(s)$, so clearly almost all partitions of $3^9$ belong to $\cB_{3^9}(10905)$ (cf.~Theorem~\ref{thm:c} and Remark~\ref{rem:2/9}). %Observe that $|\cP(3^9)|\approx 1.4\cdot 10^{151}$, while we may very crudely bound $|\cP(3^9)\setminus\cB_{3^9}(10905)|\le 2\cdot (3^9-10905)\cdot |\{ \lambda\vdash 3^9\mid \lambda_1=10906 \}| \approx 6.7\cdot 10^{103} \ll |\cP|$.
	\hfill$\lozenge$
\end{example}

With the above remark in mind, we now proceed to state and prove the remaining results which appear in Figure~\ref{fig:3}. %To aid readability, we collect together the combinatorial tools on Littlewood--Richardson coefficients and Sylow branching coefficients which will be used, and present them in the following section (Section~\ref{sec:LR}).

Proposition~\ref{prop:case2} below deals with the opposite situation to Proposition~\ref{prop:4.8} in terms of Sylow branching coefficients, when they are 1, rather than at least 2.

\begin{proposition}\label{prop:case2}
	Let $p$ be an odd prime and $k\in\N$. Suppose that $s\in[\overline{p}]^k$ satisfies the following:
	\begin{itemize}%\setlength\itemsep{0.5em}
		\item[(i)] $\tfrac{p^k}{2}+1<m(s)\le p^k-3$ and $\Omega(s)\setminus\cB_{p^k}(m(s))$ contains no thin partitions, and
		\item[(ii)] $Z^\nu_{\phi(s)}=1$ for $\nu\in\{\tworow{p^k}{m(s)}, \hook{p^k}{m(s)}\}$.
	\end{itemize}
	Then:
	\begin{itemize}
		\item $m(s,0)=p\cdot m(s)$, 
		\item $\Omega(s,0)\setminus\cB_{p^{k+1}}(p\cdot m(s))$ contains no thin partitions, and
		\item $Z^\mu_{\phi(s,0)}=1$ for $\mu\in\{\tworow{p^{k+1}}{p\cdot m(s)},\hook{p^{k+1}}{p\cdot m(s)}\}$. 
	\end{itemize}
	If $\sfx\in [p-1]$, then:
	\begin{itemize}
		\item $m(s,\sfx)=p\cdot m(s)-1$, 
		\item $\Omega(s,\sfx)\setminus\cB_{p^{k+1}}(p\cdot m(s))$ contains no thin partitions, and
		\item $Z^\mu_{\phi(s,\sfx)}\ge 2$ for $\mu\in\{\tworow{p^{k+1}}{p\cdot m(s)-1},\hook{p^{k+1}}{p\cdot m(s)-1}\}$.
	\end{itemize}
\end{proposition}

\begin{proof}
	Let $m=m(s)$. As in the proof of \cite[Proposition 4.8]{GL2}, we find for all $\sfx\in[\overline{p}]$ that $\cB_{p^{k+1}}(pm-1)\subseteq\Omega(s,\sfx)$, and $(pm,\alpha)\in\Omega(s,\sfx)$ for all $\alpha\in\cB_{p^{k+1}-pm}(p^{k+1}-pm-1)$.
	
	Let $\lambda=\tworow{p^{k+1}}{pm}$. By Theorem~\ref{thm:PW9.1}, $\langle\cX\big(\tworow{p^k}{m};(p)\big), \chi^\lambda\down_{S_{p^k}\wr S_p}\rangle = 1$. In fact, 
	\[ \chi^\lambda\down_{S_{p^k}\wr S_p} = \cX\big(\tworow{p^k}{m};(p)\big) + \Delta\]
	where $\Delta\in\Char(S_{p^k}\wr S_p)$ and $\langle \Delta\down_{P_{p^{k+1}}}, \phi(s,x)\rangle = 0$, since $\langle \Delta\down_{(P_{p^k})^{\times p}}, \phi(s)^{\times p}\rangle =0$. This follows from observing that $\chi^\lambda\down_{(S_{p^k})^{\times p}}=\big(\chi^{\tworow{p^{k}}{m}}\big)^{\times p}+\theta$ where $\theta\in\Char((S_{p^k})^{\times p})$ and for each irreducible constituent $\chi^{\mu_1}\times\cdots\times\chi^{\mu_p}$ of $\theta$ there exists $j\in [p]$ such that $(\mu_j)_1>m$. 
	Thus, for any $\sfx\in[\overline{p}]$, Lemma~\ref{lem: 9.6} implies that
	\[ Z^\lambda_{\phi(s,\sfx)} %\langle\chi^\lambda\down_{P_{p^{k+1}}}, \phi(s,x)\rangle 
	= \langle\cX\big(\tworow{p^k}{m};(p)\big)\down_{P_{p^{k+1}}}, \phi(s,\sfx)\rangle = \langle\cX\big(\tworow{p^k}{m}\down_{P_{p^k}};\phi_0\big), \phi(s,\sfx)\rangle = \delta_{0,\sfx}. \]	
	An analogous argument holds for $\lambda=\hook{p^{k+1}}{pm}$ since $\langle\cX(\hook{p^k}{m};\nu), \chi^\lambda\down_{S_{p^k}\wr S_p}\rangle = 1$ for some $\nu\in\{(p),(1^p)\}$ by Theorem~\ref{thm:GTT3.5}. %Hence $Z^{\hook{p^{k+1}}{pm}}_{\phi(s,x)} %$\langle \hook{p^{k+1}}{pm}\down_{P_{p^{k+1}}}, \phi(s,x)\rangle 
	%= \delta_{0,x}$. 
	
	Moreover, $\Omega(s,0)\setminus\cB_{p^{k+1}}(pm)$ and $\Omega(s,\sfx)\setminus\cB_{p^{k+1}}(pm-1)$ for $\sfx\ne 0$ contain no thin partitions by a similar argument to that in the proof of \cite[Proposition 4.8]{GL2}. To conclude, it remains to show for any $\sfx\neq 0$ that $Z^\lambda_{\phi(s,\sfx)}\ge 2$ for $\lambda\in\{\tworow{p^{k+1}}{pm-1},\hook{p^{k+1}}{pm-1}\}$. % we have $\langle \tworow{p^{k+1}}{pm-1}\down_{P_{p^{k+1}}}, \phi(s,x)\rangle \ge 2$ and $\langle \hook{p^{k+1}}{pm-1}\down_{P_{p^{k+1}}}, \phi(s,x)\rangle \ge 2$.
	
	First, consider $\lambda=\tworow{p^{k+1}}{pm-1}$. Let $\alpha=\tworow{p^k}{m}$ and $\beta=\tworow{p^k}{m-1}$, indeed partitions as $m>\tfrac{p^k}{2}+1$. By Theorem~\ref{theorem:LR}, $(\chi^\alpha)^{\times(p-1)}\times \chi^\beta \mid \chi^\lambda\down_{(S_{p^k})^{\times p}}$, and so (identifying $\alpha$ with $\chi^\alpha$)
	\[ (\alpha^{\times (p-1)}\times\beta)\ \up_{(S_{p^k})^{\times p}}^{S_{p^k}\wr P_p}\ \Big|\ \chi^\lambda\down_{S_{p^k}\wr P_p} \]
	by the characterisation of $\Irr(S_{p^k}\wr P_p)$ (see Section~\ref{sec:wreath}). Then
	\[ (\alpha^{\times (p-1)}\times\beta)\ \up_{(S_{p^k})^{\times p}}^{S_{p^k}\wr P_p} \down_{P_{p^k}\wr P_p} = (\alpha^{\times (p-1)}\times\beta)\down_{(P_{p^k})^{\times p}} \up^{P_{p^k}\wr P_p} \]
	by Theorem~\ref{thm:mackey}, which contains $\phi(s)^{\times p}\up^{P_{p^k}\wr P_p}=\sum_{\sfy=0}^{p-1}\phi(s,\sfy)$ as a direct summand since $\alpha,\beta\in\cB_{p^k}(m)\subseteq\Omega(s)$.
	Also, $\cX\big(\alpha;(p-1,1)\big)\mid\chi^\lambda\down_{S_{p^k}\wr S_p}$ by \cite[Theorem 1.5]{dBPW}, and observe that
	\[ \cX\big(\alpha;(p-1,1)\big)\down_{P_{p^k}\wr P_p} = \cX\big(\alpha\down_{P_{p^k}}; (p-1,1)\down_{P_p}\big) = \sum_{y=1}^{p-1} \cX(\alpha\down_{P_{p^k}}; \phi_{\sfy}), \]
	and that $\langle \cX(\alpha\down_{P_{p^k}}; \phi_\sfy), \phi(s,\sfx)\rangle = \delta_{\sfy,\sfx}$ by Lemma~\ref{lem: 9.6}.
	Thus since $\sfx\ne 0$ we have that
	\[ Z^\lambda_{\phi(s,\sfx)} % \langle \chi^\lambda\down_{P_{p^{k+1}}}, \phi(s,x)\rangle 
	\ge \langle (\alpha^{\times (p-1)}\times\beta)\ \up_{(S_{p^k})^{\times p}}^{S_{p^k}\wr P_p}\down_{P_{p^k}\wr P_p}, \phi(s,\sfx)\rangle + \sum_{\sfy=1}^{p-1} \langle \cX(\alpha\down_{P_{p^k}}; \phi_\sfy), \phi(s,\sfx)\rangle \ge 2. \]
	%since $x\ne 0$.
	For $\lambda=\hook{p^{k+1}}{pm-1}$, a similar argument holds with $\alpha=\hook{p^k}{m}$ and $\beta=\hook{p^k}{m-1}$ instead, noting that $\cX(\alpha;\omega)\mid \chi^\lambda\down_{S_{p^k}\wr S_p}$ for some $\omega\in\{(p-1,1)\}^\circ$ by Theorem~\ref{thm:GTT3.5}.	
\end{proof}

For the rest of the article, we focus our attention on $p=3$.
Fix $p=3$ and let $k\in\N$ and $\phi\in\Lin(P_{3^k})$. Suppose $\phi=\phi(s)$ for some $s\in\{0,1\}^k$.
We begin by recording the sets $\Omega(s)$ for small $k$, which can be obtained by direct computation. 

\begin{footnotesize}
	\begin{multicols}{2}
		
		\underline{$k=1$}:
		
		$\Omega(0) = \{(3),(1^3)\}$,
		
		$\Omega(1) = \{(2,1)\}$;
		
		\columnbreak
		
		\underline{$k=2$}:
		
		$\Omega(00) = \cP(9)\setminus\{(8,1),(5,4),(4,3,2)\}^\circ$,
		
		$\Omega(01) = \cB_9(8)\setminus\{(3^3)\}$,
		
		$\Omega(10) = \cB_9(6)\setminus\{(5,4),(5,1^4),(6,2,1) \}^\circ$,
		
		$\Omega(11) = \cB_9(6)\setminus\{(3^3),(6,3),(6,1^3)\}^\circ$;
		
	\end{multicols}	
	
	\vspace{-25pt}
	
	\underline{$k=3$}:
	
	$\Omega(000) = \cP(27) \setminus\{(26,1)\}^\circ$,
	
	$\Omega(001) = \cB_{27}(26)$,
	
	$\Omega(010) = \cB_{27}(24)\setminus\{(24,2,1)\}^\circ$,
	
	$\Omega(011) = \cB_{27}(24) \setminus \{(24,3),(24,1^3)\}^\circ$,
	
	$\Omega(100) = \cB_{27}(17)\setminus\{(17,10),(17,1^{10}),(14,13)\}^\circ \sqcup \{(18,\mu) \mid \mu\in\Omega(00)\}^\circ$,
	
	$\Omega(101) = \cB_{27}(18)\setminus\{(18,9),(18,3^3),(18,1^9)\}^\circ$,
	
	%$\Omega(110) = \cB_{27}(15)\setminus\{(14,13),(14,1^{13}) \}^\circ\sqcup\{(16,\omega) \mid \omega\in\cB_{11}(10) \}^\circ \sqcup \{(17,\nu) \mid \nu\in\cB_{10}(8) \}^\circ \sqcup \{(18,\mu) \mid \mu\in\Omega(10)\}^\circ$,
	$\Omega(110) = \cB_{27}(16)\setminus\{(16,11),(16,1^{11}),(14,13),(14,1^{13}) \}^\circ \sqcup \{(17,\nu) \mid \nu\in\cB_{10}(8) \}^\circ \sqcup \{(18,\mu) \mid \mu\in\Omega(10)\}^\circ$,
	
	$\Omega(111) = \cB_{27}(16)\setminus\{(16,11), (16,1^{11}), (15,12),(15,1^{12})\}^\circ \sqcup \{(17,\nu) \mid \nu\in\cB_{10}(8) \}^\circ \sqcup \{(18,\mu) \mid \mu\in\Omega(11)\}^\circ$.\\
	
	%$\Omega(1100) = \cB_{81}(45)\setminus \{(44,37),(44,1^{37}),(41,10)\}^\circ$
\end{footnotesize}

After describing the remaining results appearing in Figure~\ref{fig:3}, we conclude this section by proving Theorem~\ref{thm:b-primepower}, and Theorem~\ref{thm:a} in the case of $n$ a power of 3 (Theorem~\ref{thm:a-prime power}).

\begin{lemma}\label{lem:1000-base}
	We have that
	\begin{itemize}
		\item[(i)] $\Omega(1000)=\cB_{81}(54)\setminus\{(54,26,1),(54,2,1^{25}),(53,28),(53,1^{28})\}^\circ$,
		\item[(ii)] $Z^{(54,27)}_{\phi(1000)}=Z^{(54,1^{27})}_{\phi(1000)}=1$,
		\item[(iii)] $\Omega(1001)=\cB_{81}(54)\setminus\{(54,27),(54,1^{27})\}^\circ$, and
		\item[(iv)] $Z^{(53,28)}_{\phi(1001)}=Z^{(53,1^{28})}_{\phi(1001)}=1$.
	\end{itemize}
\end{lemma}

\begin{proof}
	Let $s=(1,0,0)\in\{0,1\}^3$. 
	
	Note $\Omega(s)\subset\cV:=\cB_{27}(18)\setminus\{(18,8,1),(18,2^{17}),(17,10),(17,1^{10})\}^\circ$.
	
	\noindent\emph{Step 1.} We first show $\cB_{81}(47)\subseteq\Omega(s,\sfx)$ for $\sfx\in\{0,1\}$.
	
	By direct calculation, $\cD(2,27,\Omega(s))=\cB_{54}(36)\setminus\cW^\circ$ where
	\[ \cW := \{(36,18),(36,17,1),(36,9^2),(36,2^9),(36,2,1^{16}), (36,1^{18}),(35,19),(35,1^{19}),(27,27) \}. \]
	It is then easy to check that $\cD(2,27,\Omega(s))=\cD(2,27,\cV)$, from which $\cD(3,27,\Omega(s))=\cD(3,27,\cV)$ then follows. 
	%Next we claim that $\cD(3,27,\Omega(s))=\cD(3,27,\cV)$. Clearly $\subseteq$ holds. Conversely, suppose $\lambda\in\cD(3,27,\cV)$, so that $c^\lambda_{\mu_1,\mu_2,\mu_3}>0$ for some $\mu_i\in\cV$ not all equal. By symmetry of Littlewood--Richardson coefficients we can assume $\mu_1\ne\mu_2$, so there exists $\alpha\in\cD(2,27,\cV)$ such that $c^\lambda_{\alpha,\mu_3}>0$. Since $\cD(2,27,\cV)=\cD(2,27,\Omega(s))$, then $c^\lambda_{\beta,\gamma,\mu_3}>0$ for some $\beta\ne\gamma\in\Omega(s)$. If $\mu_3\in\Omega(s)$ then immediately we see $\lambda\in\cD(3,27,\Omega(s))$. Otherwise $c^\lambda_{\beta,\gamma,\mu_3}=c^{\lambda}_{\mu_3,\beta,\gamma}>0$ implies there exists some $\Delta\in\cD(2,27,\cV)=\cD(2,27,\Omega(s))$ such that $c^\Delta_{\mu_3,\beta}>0$, whence $\lambda\in\cD(3,27,\Omega(s))$.
	Therefore $\cB_{81}(47)\subseteq\cD(3,27,\cB_{27}(16))\subseteq\cD(3,27,\cV)=\cD(3,27,\Omega(s))\subseteq\Omega(s,\sfx)$ by Proposition~\ref{prop: BinD} and Lemma~\ref{lem: DinOm}.
		
	\noindent\emph{Step 2.} Next, we show $\cB_{81}(52)\subseteq\Omega(s,\sfx)$ for $\sfx\in\{0,1\}$.
	
	Suppose $\lambda\vdash 81$ with $48\le\lambda_1\le 52$. There exists $\gamma\vdash 9$ such that $\mu:=(18,\gamma)$ satisfies $\mu\subset\lambda$ and $\mu\in\cV$. By considering an LR filling of $[(\lambda_2,\lambda_3,\dotsc)\setminus\gamma]$, replacing each number $i$ in the filling by $i+1$, and filling $[(\lambda_1)\setminus(18)]$ with 1s, we see that $c^\lambda_{\mu,\delta}>0$ for some $\delta\in\cB_{54}(34)$ with $\delta_1\ge 30$. Hence $\delta\in\cD(2,27,\Omega(s))$, and so $\lambda\in\cD(3,27,\cV)=\cD(3,27,\Omega(s))\subseteq\Omega(s,\sfx)$. Then $\cB_{81}(52)\subseteq\Omega(s,\sfx)$ follows since $\Omega(s,\sfx)=\Omega(s,\sfx)^\circ$.
	
	\noindent\emph{Step 3.} By direct calculation we know that if $\lambda\vdash 81$ with $\lambda_1=53$, then $\lambda\in\cD(3,27,\cV)$ if and only if $\lambda\notin\{(53,28),(53,1^{28})\}$. By Theorem~\ref{thm:M}, $M(1000)=M(1001)=54$, and we also know $\Omega(1001)\cap\{\lambda\vdash 81\mid\lambda_1=54\}^\circ=\{(54,\mu)\mid\mu\in\Omega(001)\}^\circ=\{(54,\mu)\mid \mu\in\cB_{27}(26)\}^\circ$, and $\Omega(1000)\cap\{\lambda\vdash 81\mid\lambda_1=54\}^\circ=\{(54,\mu)\mid\mu\in\Omega(000)\}^\circ=\{(54,\mu)\mid \mu\in\cP(27)\setminus\{(26,1)\}^\circ\}^\circ$ by Lemma~\ref{lem: 9.8}. 
	Thus it remains to compute the multiplicities $Z^{(54,27)}_{\phi(1000)}$, $Z^{(54,1^{27})}_{\phi(1000)}$, $Z^{(53,28)}_{\phi(1000)}$, $Z^{(53,1^{28})}_{\phi(1000)}$, $Z^{(53,28)}_{\phi(1001)}$ and $Z^{(53,1^{28})}_{\phi(1001)}$ to complete the proof of the lemma.
	
	\noindent\emph{Step 4.} We compute the multiplicities in question. Let $S=S_{81}$, $V=S_{27}\wr S_3$, $W=S_{27}\wr P_3$, $Y=S_{27}^{\times 3}$, $P=P_{81}=P_{27}\wr P_3$ and $B=B_{27}^{\times 3}$.
	
	Let $\varphi=\phi(100\sfx)$ with $\sfx\in\{0,1\}$, and first let $\lambda\in\{(54,27),(53,28)\}$. The only constituent $\psi$ of $\chi^\lambda\down_Y$ such that $\psi\in\Irr(W\mid\phi(s)^{\times 3})$ is $\psi=(18,9)^{\times 3}$, with multiplicity 1 if $\lambda=(54,27)$ and 2 if $\lambda=(53,28)$. Since $\varphi\down_{P_{27}^{\times 3}}=\phi(s)^{\times 3}$ and $\Irr(W\mid(18,9)^{\times 3})=\{\cX\big((18,9);\phi_i\big)\mid i\in[\overline{3}]\}$, then 
	\[ \langle\chi^\lambda\down_P,\varphi\rangle = \sum_{i=0}^2 \langle \chi^\lambda\down_W, \cX\big((18,9);\phi_i\big)\rangle \cdot \langle \cX\big((18,9);\phi_i\big)\down_P,\varphi\rangle. \]
	Since $\langle(18,9)\down^{S_{27}}_{P_{27}},\phi(100)\rangle=1$, we have $\langle \cX\big((18,9);\phi_i\big)\down_P,\varphi\rangle = \delta_{i,\sfx}$ by Lemma~\ref{lem: 9.6}. On the other hand, $\cX\big((18,9);(3)\big)\mid\chi^{(54,27)}\down^S_V$ by Theorem~\ref{thm:PW9.1} and $\cX\big((18,9);(2,1)\big)\mid\chi^{(53,28)}\down^S_V$ by \cite[Theorem 1.5]{dBPW}. Since $(3)\down^{S_3}_{P_3}=\phi_0$ and $(2,1)\down^{S_3}_{P_3}=\phi_1+\phi_2$, then $\langle \chi^\lambda\down_W, \cX\big((18,9);\phi_i\big)\rangle$ equals 1 if $\lambda=(54,27)$ and $i=0$ or if $\lambda=(53,28)$ and $i\ne 0$, and equals 0 otherwise.
	
	Thus $Z^{(54,27)}_{\phi(1000)}=1$, $Z^{(53,28)}_{\phi(1000)}=0$ and $Z^{(53,28)}_{\phi(1001)}=1$.
	The remaining multiplicities follow similarly, noting that $\langle(18,1^9)\down^{S_{27}}_{P_{27}},\phi(100)\rangle=1$, and $\cX\big((18,1^9);\alpha\big)\mid\chi^{(54,1^{27})}\down^S_V$ (where $\alpha\in\{(3),(1^3)\}$) and $\cX\big((18,1^9);(2,1)\big)\mid\chi^{(53,1^{28})}\down^S_V$ by Theorem~\ref{thm:GTT3.5}.
\end{proof}

\begin{lemma}\label{lem:1000}
	Let $k\ge 4$ and $a=3^{k-1}$. Then
	\begin{itemize}%\setlength\itemsep{0.5em}
		\item[(i)] $\Omega(10^{k-1}) = \cB_{3^k}(2a)\setminus \{ (2a,a-1,1), (2a,2,1^{a-2}), (2a-1,a+1), (2a-1,1^{a+1}) \}^\circ$, %1,\underbrace{0,\dotsc,0}_{k-1}
		\item[(ii)] $Z^{(2a,a)}_{\phi(10^{k-1})} = Z^{(2a,1^a)}_{\phi(10^{k-1})}=1$, %$\langle (2a,a)\down_{P_{3^k}}, \phi(10^{k-1})\rangle = \langle (2a,1^a)\down_{P_{3^k}}, \phi(10^{k-1})\rangle =1$;
		\item[(iii)] $\Omega(10^{k-2}1) = \cB_{3^k}(2a)\setminus \{ (2a,a), (2a,1^a) \}^\circ$, and %\cB_{3^k}(2a-1)\sqcup \{(2a,\mu) \mid \mu\in\Omega(\underbrace{0,\dotsc,0}_{k-2},1) \}^\circ$; and
		% 1,\underbrace{0,\dotsc,0}_{k-2},1
		\item[(iv)] $Z^{(2a-1,a+1)}_{\phi(10^{k-2}1)} = Z^{(2a-1,1^{a+1})}_{\phi(10^{k-2}1)}=1$. %$\langle (2a-1,a+1)\down_{P_{3^k}}, \phi(10^{k-2}1)\rangle = \langle (2a-1,1^{a+1})\down_{P_{3^k}}, \phi(10^{k-2}1)\rangle =1$.
	\end{itemize}
\end{lemma}

\begin{proof}
	We proceed by induction on $k$: the base case $k=4$ follows from Lemma~\ref{lem:1000-base}. Now assume $s=(1,0,\dotsc,0)\in\{0,1\}^k$ with $k\ge 4$ and that
%	\[ \Omega(s)=\cB_{3^k}(2a)\setminus\{(2a,a-1,a),(2a,2,1^{a-2}),\tworow{3^k}{2a-1},\hook{3^k}{2a-1}\}^\circ \ \text{and}\ Z^{(2a,a)}_{\phi(s)} = Z^{(2a,1^a)}_{\phi(s)} = 1 \]
	\begin{itemize}
		\item[$\circ$] $\Omega(s)=\cB_{3^k}(2a)\setminus\{(2a,a-1,a),(2a,2,1^{a-2}),\tworow{3^k}{2a-1},\hook{3^k}{2a-1}\}^\circ$ and
		\item[$\circ$] $Z^{(2a,a)}_{\phi(s)} = Z^{(2a,1^a)}_{\phi(s)} = 1$, %$\langle (2a,a)\down_{P_{3^k}},\phi(s)\rangle=\langle (2a,1^a)\down_{P_{3^k}},\phi(s)\rangle=1$,
	\end{itemize}
	where $a=3^{k-1}$. We claim that
	\begin{itemize}
		\item[(a)] $\Omega(s,0)=\cB_{3^{k+1}}(6a)\setminus\{(6a,3a-1,1), (6a,2,1^{3a-2}),\tworow{3^{k+1}}{6a-1},\hook{3^{k+1}}{6a-1}\}^\circ$,
		\item[(b)] $Z^{(6a,3a)}_{\phi(s,0)} = Z^{(6a,1^{3a})}_{\phi(s,0)} = 1$, %$\langle (6a,3a)\down_{P_{3^{k+1}}}, \phi(s,0)\rangle = \langle(6a,1^{3a})\down_{P_{3^{k+1}}},\phi(s,0)\rangle=1$,
		\item[(c)] $\Omega(s,1)=\cB_{3^{k+1}}(6a)\setminus\{(6a,3a),(6a,1^{3a})\}^\circ$, and
		\item[(d)] $Z^{(6a-1,3a+1)}_{\phi(s,1)} = Z^{(6a-1,1^{3a+1})}_{\phi(s,1)} = 1$, %$\langle (6a-1,3a+1)\down_{P_{3^{k+1}}},\phi(s,1)\rangle=\langle (6a-1,1^{3a+1})\down_{P_{3^{k+1}}},\phi(s,1)\rangle=1$,
	\end{itemize}
	from which (i) -- (iv) follow. We proceed in steps: (a), (c) and (d) follow after Step 4, and (b) after Step 5.
	
	\smallskip
	
	\noindent\emph{Step 1.} By Theorem~\ref{thm:M}, $M(s,\sfx)=3^{k+1}-3^k=6a$ for all $\sfx\in\{0,1\}$. Thus $\Omega(s,\sfx)\cap \{\lambda\vdash 3^{k+1} \mid \lambda_1=6a \}^\circ = \{(6a,\mu) \mid \mu\in\Omega(0^{k-1},x) \}^\circ$ by Lemma~\ref{lem: 9.8}. Moreover, $\Omega(0^k)=\cP(3^k)\setminus\{(3^k-1,1)\}^\circ$ by Theorem~\ref{thm:GL1A} and $\Omega(0^{k-1}1)=\cB_{3^k}(3^k-1)$ by Lemma~\ref{lem: 001} since $k\ge 4$.
	
	\noindent\emph{Step 2.} Since $k\ge 4$, we may apply Proposition~\ref{prop: BinD} with $q=3$, $m=3^k$ and $t=2a-2$ to deduce that $\cB_{3^{k+1}}(3\cdot(2a-2)-1)\subseteq \cD\big(3,3^k,\cB_{3^k}(2a-2)\big)$. Hence by Lemma~\ref{lem: DinOm}, for all $\sfx\in\{0,1\}$,
	\[ \cB_{3^{k+1}}(6a-7) \subseteq \cD(3,3^k,\Omega(s)) \subseteq\Omega(s,\sfx). \]
	
	\noindent\emph{Step 3.} Since $\Omega(s,\sfx)^\circ=\Omega(s,\sfx)$, it suffices to consider $\lambda\vdash 3^{k+1}$ such that $6a-6\le \lambda_1\le 6a-1$. Let $\lambda=(\lambda_1,\mu)$ for some $\mu\vdash 3^{k+1}-\lambda_1$. We wish to show $\lambda\in\Omega(s,\sfx)$ unless $\lambda$ is an exception as listed in (a) or (c).
	
	\noindent$\bullet$ If $\lambda_1=6a-6$: %observe $6a-6=(2a-2) + (2a-2) + (2a-2)$ and 
	let $b=a+2$. First suppose $\mu\in\cB_{3b}(3b-1)$. Since $\cB_{3b}(3b-1)\subseteq\cD(3,b,\cP(b))$ by Proposition~\ref{prop: BinD}, there exist $\nu_1,\nu_2,\nu_3\vdash a+2$ not all equal such that $c^\mu_{\nu_1,\nu_2,\nu_3}>0$. Hence $c^\lambda_{(2a-2,\nu_1),(2a-2,\nu_2),(2a-2,\nu_3)}>0$ by Lemma~\ref{lem: iteratedLR}. Since %Observe that
	$(2a-2,\nu_i)\in\Omega(s)$, then %so
	$\lambda\in\cD(3,3^k,\Omega(s))\subseteq\Omega(s,\sfx)$ for all $\sfx$. Otherwise $\mu\in\{(3b)\}^\circ$, i.e.~$\lambda\in\{\tworow{3^{k+1}}{6a-6}, \hook{3^{k+1}}{6a-6} \}$. If $\lambda=\tworow{3^{k+1}}{6a-6}$, then %observe that
	$c^\lambda_{\mu_1,\mu_2,\mu_3}>0$ where $\mu_1=\mu_2=\tworow{3^k}{2a}$ and $\mu_3=\tworow{3^k}{2a-6}$, so $\lambda\in\cD(3,3^k,\Omega(s))$ also. A similar argument holds for $\lambda=\hook{3^{k+1}}{6a-6}$.
	
	\noindent$\bullet$ If $\lambda_1=6a-5$: if $\mu\in\cB_{3a+5}(3a+4)$ then $c^\mu_{\nu_1,\nu_2,\nu_3}>0$ for some $\nu_1\in\cB_{a+1}(a)$ and $\nu_2,\nu_3\in\cP(a+2)$ by Proposition~\ref{prop: B star B}. Hence $c^\lambda_{(2a-1,\nu_1),(2a-2,\nu_2),(2a-2,\nu_3)}>0$ by Lemma~\ref{lem: iteratedLR}. By assumption, $(2a-1,\nu_1)$, $(2a-2,\nu_2)$, $(2a-2,\nu_3)\in\Omega(s)$, so $\lambda\in\cD(3,3^k,\Omega(s))\subseteq\Omega(s,\sfx)$ for all $\sfx\in\{0,1\}$.
	%since $6a-5=(2a-1)+(2a-2)+(2a-2)$ we similarly see that $\lambda\in\Omega(s,x)$ for all $x$ if $\mu\in\cB_{3a+5}(3a+4)$ (which equals $\cB_{a+1}(a)\star\cP(a+2)\star\cP(a+2)$ by Proposition~\ref{prop: B star B}). 
	Moreover, $\tworow{3^{k+1}}{6a-5}\in\Omega(s,\sfx)$ follows similarly by considering $\mu_1=\mu_2=\tworow{3^k}{2a}$ and $\mu_3=\tworow{3^k}{2a-5}$ (as in the case $\lambda_1=6a-6$), as does $\hook{3^{k+1}}{6a-5}\in\Omega(s,\sfx)$.
	
	\noindent$\bullet$ If $\lambda_1=6a-4$: note $\cP(3a+4)=(\cP(a)\setminus \{(a-1,1)\}^\circ)\star\cP(a+2)\star\cP(a+2)$ by Lemma~\ref{lem: 10.3} and Proposition~\ref{prop: B star B}, so similarly for all $\mu\in\cP(3a+4)$ we have $\lambda\in\Omega(s,\sfx)$ for all $\sfx$.
	
	\noindent$\bullet$ If $\lambda_1=6a-3$:	note $\cB_{3a+3}(3a+2) = (\cP(a)\setminus\{(a-1,1)\}^\circ)\star\cB_{a+1}(a)\star\cP(a+2)$ by Lemma~\ref{lem: 10.3} and Proposition~\ref{prop: B star B}, so for all $\mu\in\cB_{3a+3}(3a+2)$ we have $\lambda\in\Omega(s,\sfx)$ for all $\sfx$. Moreover, it is similarly easy to observe that $\{\tworow{3^{k+1}}{6a-3},\hook{3^{k+1}}{6a-3} \}\subseteq\Omega(s,\sfx)$ for all $\sfx$.
		
	\noindent$\bullet$ If $\lambda_1=6a-2$: note $\cP(3a+2)=(\cP(a)\setminus\{(a-1,1)\}^\circ)\star(\cP(a)\setminus\{(a-1,1)\}^\circ)\star\cP(a+2)$ by Lemma~\ref{lem: 10.3}, so for all $\mu\in\cP(3a+2)$ we have $\lambda\in\Omega(s,\sfx)$ for all $\sfx\in\{0,1\}$.
	
	\noindent$\bullet$ If $\lambda_1=6a-1$: note $\cB_{a+1}(a)\star(\cP(a) \setminus\{(a-1,1)\}^\circ)\star(\cP(a)\setminus \{(a-1,1)\}^\circ)=\cB_{3a+1}(3a)$ by Lemma~\ref{lem: 10.3}, so for all $\mu\in\cB_{3a+1}(3a)$ we have $\lambda\in\Omega(s,\sfx)$ for all $\sfx$.
	
	\noindent\emph{Step 4.} Now consider $\lambda=\tworow{3^{k+1}}{6a-1}$. By Theorem~\ref{theorem:LR}, the only irreducible constituent(s) $\mu_1\times\mu_2\times\mu_3$ of $\chi^\lambda\down_{(S_{3^k})^3}$ such that $\mu_i\in\Omega(s)$ for all $i$ is $(2a,a)^{\times 3}$ with multiplicity 2, since $(2a-1,a+1)\notin\Omega(s)$. Moreover, $\cX\big( (2a,a); (2,1)\big) \mid \chi^\lambda\down_{S_{3^k}\wr S_3}$ by \cite[Theorem 1.5]{dBPW}, and 
	\[ \cX\big( (2a,a); (2,1)\big)\down_{(S_{3^k})^3} = (2a,a)^{\times 3}\cdot \chi^{(2,1)}(1) = 2\cdot (2a,a)^{\times 3}. \]
	Since $\chi^{(2,1)}\down_{P_3}=\phi_1+\phi_2$, we have that
	\begin{align*}
		Z^\lambda_{\phi(s,\sfx)} %\langle \lambda\down_{P_{3^{k+1}}}, \phi(s,x) \rangle 
		&= \langle \cX\big( (2a,a); (2,1)\big)\down^{S_{3^k}\wr S_3}_{P_{3^{k+1}}}, \phi(s,\sfx) \rangle\\
		&= \langle \cX\big( (2a,a); \phi_1)\big)\down^{S_{3^k}\wr P_3}_{P_{3^{k+1}}}, \phi(s,\sfx) \rangle + \langle \cX\big( (2a,a); \phi_2\big)\down^{S_{3^k}\wr P_3}_{P_{3^{k+1}}}, \phi(s,\sfx) \rangle = \delta_{1,\sfx} + \delta_{2,\sfx}
	\end{align*}
	where the last equality follows from Lemma~\ref{lem: 9.6} since $Z^{(2a,a)}_{\phi(s)}=1$. %$\langle (2a,a)\down_{P_{3^k}}, \phi(s)\rangle=1$.
	If $\lambda=\hook{3^{k+1}}{6a-1}$, then a similar argument shows that $Z^\lambda_{\phi(s,\sfx)} %$\langle \lambda\down_{P_{3^{k+1}}}, \phi(s,x)\rangle
	=\delta_{1,\sfx}+\delta_{2,\sfx}$ since $\cX\big((2a,1^a); (2,1) \big) \mid \chi^\lambda\down_{S_{3^k}\wr S_3}$ by Theorem~\ref{thm:GTT3.5}. Thus we have proved (a), (c) and (d).
	
	\noindent\emph{Step 5.} Finally, let $\lambda=\tworow{3^{k+1}}{6a}$. The only irreducible constituent(s) $\mu_1\times\mu_2\times\mu_3$ of $\chi^\lambda\down_{(S_{3^k})^3}$ such that $\mu_i\in\Omega(s)$ for all $i$ is $(2a,a)^{\times 3}$ with multiplicity 1. Moreover, $\cX\big( (2a,a); (3)\big) \mid \chi^\lambda\down_{S_{3^k}\wr S_3}$ by Theorem~\ref{thm:PW9.1}. Since $\chi^{(3)}\down_{P_3}=\phi_0$, we have that $Z^{\lambda}_{\phi(s,\sfx)} %$\langle \lambda\down_{P_{3^{k+1}}}, \phi(s,x)\rangle
	=\delta_{0,\sfx}$, using Lemma~\ref{lem: 9.6} as above since $Z^{(2a,a)}_{\phi(s)}=1$.
	If $\lambda=\hook{3^{k+1}}{6a}$ then we also have $Z^\lambda_{\phi(s,\sfx)} %\langle \lambda\down_{P_{3^{k+1}}}, \phi(s,x)\rangle
	=\delta_{0,\sfx}$, since $\cX\big((2a,1^a);\nu \big) \mid \chi^\lambda\down_{S_{3^k}\wr S_3}$ for some $\nu\in\{(3)\}^\circ$ by Theorem~\ref{thm:GTT3.5}. Thus we have proven (b).
\end{proof}

\begin{lemma}\label{lem: LRa}
	Let $k\ge 3$. Suppose $s\in\{0,1\}^k$ satisfies $\cB_{3^k}(r)\subseteq\Omega(s)$ where $r=\tfrac{3^k+1}{2}$, and that $Z^\mu_{\phi(s)}\ge 2$ 
	%$$\langle \chi^\mu\down_{P_{3^k}}, \phi(s)\rangle \ge 2$$
	for all $\mu\in\cB_{3^k}(r)\setminus\{\tworow{3^k}{r},\hook{3^k}{r}\}^\circ$.
	Then 
	%$$\langle \chi^\lambda\down_{P_{3^{k+1}}}, \phi(s,x)\rangle \ge 2,$$
	\[ Z^\lambda_{\phi(s,\sfx)}\ge 2\quad \forall\ \sfx\in\{0,1\},\ \forall\ \lambda\in\cB_{3^{k+1}}(3r-1)\setminus\{\tworow{3^{k+1}}{3r-1},\hook{3^{k+1}}{3r-1} \}^\circ. \]
	%$Z^\lambda_{\phi(s,x)}\ge 2$ for all $x\in\{0,1\}$ and all $\lambda\in\cB_{3^{k+1}}(3r-1)\setminus\{\tworow{3^{k+1}}{3r-1},\hook{3^{k+1}}{3r-1} \}^\circ$.
\end{lemma}

\begin{proof}
	Let $\lambda\in\cB_{3^{k+1}}(3r-1)\setminus\{\tworow{3^{k+1}}{3r-1},\hook{3^{k+1}}{3r-1} \}^\circ.$ Since $r>\tfrac{3^k}{2}$, 
	$$\lambda\in\cB_{3^{k+1}}(3r) = \cB_{3^k}(r)\star\cB_{3^k}(r)\star\cB_{3^k}(r)$$
	by Proposition~\ref{prop: B star B}, so there exist $\mu_1,\mu_2,\mu_3\in\cB_{3^k}(r)$ such that $c^\lambda_{\mu_1,\mu_2,\mu_3}>0$. 
	Let $\cC=\cB_{3^k}(r)\setminus\{\tworow{3^k}{r},\hook{3^k}{r}\}^\circ$. First we prove the following: \textit{Claim:} if $\mu_i\in\cC$ for some $i$, then $Z^\lambda_{\phi(s,\sfx)}\ge 2$ for all $\sfx\in\{0,1\}$.
	
	To see that this claim holds: observe that if $\mu_i$ are not all equal, then by Theorem~\ref{thm:mackey} $\chi^\lambda\down_{P_{3^{k+1}}}$ contains $\prod_{i=1}^3 (\mu_i\down_{P_{3^k}})\up^{P_{3^{k+1}}}_{(P_{3^k})^{\times 3}}$ as a summand, and hence 
	\[ Z^\lambda_{\phi(s,\sfx)} \ge \prod_{i=1}^3 Z^{\mu_i}_{\phi(s)} \cdot \langle\phi(s)^{\times 3}\up^{P_{3^{k+1}}}, \phi(s,\sfx)\rangle \ge 2 \sum_{\sfy=0}^{p-1} \langle\phi(s,\sfy), \phi(s,\sfx)\rangle = 2,\]
	by Lemma~\ref{lem: easy observation}. If $\mu_i=\mu$ for all $i$, then $\cX(\mu;\phi_i) \mid \chi^\lambda\down_{S_{3^k}\wr P_3}$ for some $i\in\{0,1,2\}$, and so 
	\[ Z^\lambda_{\phi(s,x)} \ge \langle \cX(\mu;\phi_i)\down_{P_{3^{k+1}}}, \phi(s,\sfx)\rangle \ge 2 \]
	%$$\langle \chi^\lambda\down_{P_{3^{k+1}}}, \phi(s,x)\rangle \ge \langle \cX(\mu;\phi_i)\down_{P_{3^{k+1}}}, \phi(s,x)\rangle \ge 2$$
	by Lemma~\ref{lem: 9.5} since $Z^\mu_{\phi(s)}\ge 2$. %$\langle \chi^\mu\down_{P_{3^k}}, \phi(s)\rangle \ge 2$.
	Thus the claim holds.
	
	So now we may assume $\mu_i\in\{\tworow{3^k}{r},\hook{3^k}{r}\}^\circ$ for all $i\in[3]$. Recall that for any permutation $\sigma\in S_3$ we have $c^\lambda_{\mu_1,\mu_2,\mu_3}=c^\lambda_{\mu_{\sigma(1)},\mu_{\sigma(2)},\mu_{\sigma(3)}}$.
	
	Since $c^\lambda_{\mu_1,\mu_2,\mu_3}>0$, there exists $\alpha\vdash 2\cdot 3^k$ such that $\alpha\subset\lambda$, $c^\lambda_{\alpha,\mu_3}>0$ and $c^{\alpha}_{\mu_1,\mu_2}>0$. Fix an LR filling $\mathsf{F}$ of $[\alpha\setminus\mu_1]$ of weight $\mu_2$. By Lemma~\ref{lem: LR tworow filling} (a), if $\mu_2=t_{3^k}[r]$ then we may assume that all the 2s of $\mathsf{F}$ appear in the same row of $[\alpha\setminus\mu_1]$, or else there exists $\nu\in\cC\cap\cLR([\alpha\setminus\mu_1])$, whence $c^\lambda_{\mu_1,\nu,\mu_3}>0$ and we are done by the Claim.
	
	\noindent\emph{Case 1: $\mu_1=\hook{3^k}{r}$ and $\mu_2=\tworow{3^k}{r}$.} Since $\mu_1=\hook{3^k}{r}$, then either $[\alpha\setminus\mu_1]$ has a connected component being the single box in position $(r,1)$ or $(1,r)$, in which case $(r-1,r-1,1)\in\cC\cap\cLR([\alpha\setminus\mu_1])$, or $\alpha\in\{(2r,r,1^{r-2}),(2r-1,r+1,1^{r-2})\}$, in which case $c^\alpha_{\nu_1,\nu_2}>0$ where $\nu_1=(r,2,1^{r-3})\in\cC$ and $\nu_2=(r,r-2,1)\in\cC$. Thus $c^\lambda_{\nu_1,\nu_2,\nu_3}>0$ for some $\nu_i\in\cB_{3^k}(r)$ and moreover at least one $\nu_i$ belongs to $\cC$ for some $i$. So we are done by the Claim.
	
	\noindent\emph{Case 2: $\mu_1=\hook{3^k}{r}$ and $\mu_2=\tworow{3^k}{r}'$.} Since $c^\lambda_{\mu_1,\mu_2,\mu_3}=c^{\lambda'}_{\mu_1',\mu_2',\mu_3'}$ and $Z^\lambda_{\phi(s,\sfx)} = Z^{\lambda'}_{\phi(s,\sfx)}$, %$\langle \chi^\lambda\down_{P_{3^{k+1}}}, \phi(s,x)\rangle = \langle \chi^{\lambda'}\down_{P_{3^{k+1}}}, \phi(s,x)\rangle$, then
	we are done by Case 1.
	
	Hence we may now assume that either $\mu_1=\mu_2=\mu_3=\hook{3^k}{r}$, or that no $\mu_i$ equals $\hook{3^k}{r}$, since $\hook{3^k}{r}=\hook{3^k}{r}'$.
	
	\noindent\emph{Case 3: no $\mu_i$ equals $\hook{3^k}{r}$.} We may without loss of generality assume that $\mu_1=\mu_2=\tworow{3^k}{r}$. Since all the 2s of $\mathsf{F}$ in the same row of $[\alpha\setminus\mu_1]$, then one of the following holds: $\alpha\in\{(2r-1,2r-2,1),(2r-1,r,r-1),(r+1,(r-1)^3),(r,r,r-1,r-1)\}$, in which case $c^\alpha_{\nu_1,\nu_2}>0$ for $\nu_1=\mu_1=\hook{3^k}{r}$ and $\nu_2=(r-1,r-1,1)\in\cC$; or $\alpha=(2r,r-1,r-1)$, in which case we instead take $\nu_1=\nu_2=(r,r-2,1)\in\cC$; or $\alpha=(2r-2,r,r)$, in which case $\nu_1=\nu_2=(r-1,r-1,1)\in\cC$; or $\alpha\in\{(2r,2r-2),(2r-1,2r-1)\}$.
	
	For these two remaining possibilities for $\alpha$, we consider cases for $\mu\in\{\tworow{3^k}{r}\}^\circ$, deduce the possible values of $\lambda$, and exhibit $\nu_1,\nu_2,\nu_3\in\cB_{3^k}(r)$ such that $c^\lambda_{\nu_1,\nu_2,\nu_3}>0$ and at least one $\nu_i$ belongs to $\cC$, from where we would be done by the Claim.
	Since $c^\lambda_{\alpha,\mu_3}=c^\lambda_{\mu_3,\alpha}>0$, fix an LR filling $\mathsf{G}$ of $[\lambda\setminus\mu_3]$ of weight $\alpha$. If not all 2s of $\mathsf{G}$ lie in the same row, then by Lemma~\ref{lem: LR tworow filling} (a) we have that $c^\lambda_{\mu_3,\beta}>0$ where $\beta=(2r,2r-3,1)$ or $(2r-1,2r-2,1)$ respectively. But $c^\beta_{\nu_1,\nu_2}>0$ where $\nu_1=\tworow{3^k}{r}$, and $\nu_2=(r,r-2,1)\in\cC$ or $\nu_2=(r-1,r-1,1)\in\cC$ respectively.
	
	Thus we may assume all 2s in $\mathsf{G}$ lie in the same row. Hence if $\mu_3=\tworow{3^k}{r}$ and $\alpha=(2r-1,2r-1)$, then $\lambda=(3r-1,3r-2)$, a contradiction. If $\mu_3=\tworow{3^k}{r}'$ and $\alpha=(2r-1,2r-1)$, then $\lambda=(2r+1,2r+1,2^{r-3},1)$ and we take $\nu_1=(r-1,r-1,1)'\in\cC$, $\nu_2=(r-1,r-1,1)\in\cC$ and $\nu_3=\tworow{3^k}{r}$. 
	%The case when $\alpha=(2r,2r-2)$ is treated similarly.
	If $\mu_3=\tworow{3^k}{r}$ and $\alpha=(2r,2r-2)$, then $\lambda\in\{(3r,3r-3),(3r-1,3r-2)\}$, a contradiction, or $\lambda\in\{(3r-1,3r-3,1),(3r-2,3r-2,1)\}$, in which case take $\nu_1=\nu_2=\tworow{3^k}{r}$ and $\nu_3=(r-1,r-1,1)\in\cC$, or $\lambda=(3r-2,3r-3,2)$, in which case take $\nu_1=\tworow{3^k}{r}$ and $\nu_2=\nu_3=(r-1,r-1,1)\in\cC$.
	If $\mu_3=\tworow{3^k}{r}'$ and $\alpha=(2r,2r-2)$, then either $\lambda=(2r+2,2r,2^{r-3},1)$, in which case take $\nu_1=(r-1,r-1,1)'\in\cC$, $\nu_2=\tworow{3^k}{r}$ and $\nu_3=(r-1,r-1,1)\in\cC$, or $\lambda\in\{(2r+1,2r,2^{r-2}),(2r+1,2r,2^{r-3},1^2),(2r,2r,2^{r-2},1)\}$, in which case take $\nu_1=(r-1,r-1,1)'\in\cC$ and $\nu_2=\nu_3=(r-1,r-1,1)\in\cC$.
	
	\noindent\emph{Case 4; $\mu_i=\hook{3^k}{r}$ for all $i\in[3]$.} By Lemma~\ref{lem: LR hook filling} (a), we are done by the Claim unless there are at least $r-1$ 1s in the filling $\mathsf{F}$ that appear in the same row of $[\alpha\setminus\mu_1]$. Since $\mu_1=\hook{3^k}{r}$, then one of the following holds: 
	\begin{multline*}
		\alpha\in\{(2r,2^b,1^{2r-2-2b})\}_{b=1}^{r-1} \cup \{(2r-1,2^b,1^{2r-1-2b})\}_{b=1}^{r-1} \\ \cup \{ (r+1,r,2^b,1^{2r-3-2b})\}_{b=0}^{r-2} \cup \{ (r,r,2^b,1^{2r-2-2b})\}_{b=0}^{r-1},
	\end{multline*}
	in which case $c^\alpha_{\nu_1,\nu_2}>0$ with $\nu_1=(r,2,1^{r-3})\in\cC$ and $\nu_2=\hook{3^k}{r}$; or $\alpha\in\{(2r,1^{2r-2})\}^\circ$. 
	
	Thus it suffices to consider the case when $\alpha=(2r,1^{2r-2})$. Since $c^\lambda_{\alpha,\mu_3}=c^\lambda_{\mu_3,\alpha}>0$, fix an LR filling $\mathsf{G}$ of $[\lambda\setminus\mu_3]$ of weight $\alpha$. By Lemma~\ref{lem: LR hook filling} (a), we are done unless at least $2r-1$ 1s appear in the same row of $[\lambda\setminus\mu_3]$. Since $\mu_3=\hook{3^k}{r}$, then one of the following holds (noting that $\lambda_1<3r$): $\lambda=(3r-1,2^b,1^{3r-2-2b})$ for some $b\in[r]$ (note $b\ne 0$ as $\lambda\ne \hook{3^{k+1}}{3r-1}$), or $\lambda=(3r-1,3,2^b,1^{3r-5-2b})$ for some $0\le b\le r-2$. Then $c^\lambda_{\nu_1,\nu_2,\nu_3}>0$ with $\nu_1=(r,2,1^{r-3})\in\cC$ and $\nu_2=\nu_3=\hook{3^k}{r}$.
\end{proof}

\begin{lemma}\label{lem:111}
	Let $k\ge 3$ and $r=\tfrac{3^k+1}{2}$. Then
	\begin{itemize}%\setlength\itemsep{0.5em}
		\item[(i)] $m(1^k)=r$ and  $\Omega(1^k)\setminus\cB_{3^k}(r)$ contains no thin partitions,
		\item[(ii)] $Z^{\tworow{3^k}{r}}_{\phi(1^k)} = Z^{\hook{3^k}{r}}_{\phi(1^k)} = 1$, %$\langle \tworow{3^k}{r}\down_{P_{3^k}}, \phi(1^k)\rangle = \langle \hook{3^k}{r}\down_{P_{3^k}}, \phi(1^k)\rangle = 1$;
		\item[(iii)] $\cQ\subseteq\Omega(1^{k-1}0)$ and $\Omega(1^{k-1}0)\setminus\cQ$ contains no thin partitions, where 
		\[ \cQ:=\cB_{3^k}(r+1)\setminus \{ \tworow{3^k}{r}, \hook{3^k}{r} \}^\circ, \]
		\item[(iv)] $Z^{\tworow{3^k}{r+1}}_{\phi(1^{k-1}0)} = Z^{\hook{3^k}{r+1}}_{\phi(1^{k-1}0)} = 1$. %$\langle \tworow{3^k}{r+1}\down_{P_{3^k}}, \phi(1^{k-1}0)\rangle = \langle \hook{3^k}{r+1}\down_{P_{3^k}}, \phi(1^{k-1}0)\rangle = 1$.
	\end{itemize}
\end{lemma}

\begin{proof}
	We proceed by induction on $k$: the base case $k=3$ may be checked by direct computation. Also $Z^\mu_{\phi(111)}\ge 2$ 
	%We observe also that $\langle \chi^\mu\down_{P_{27}},\phi(111)\rangle \ge 2$ 
	for all $\mu\in\cB_{27}(14)\setminus\{(14,13),(14,1^{13})\}^\circ$.
	Now assume $s=(1,\dotsc,1)\in\{0,1\}^k$ with $k\ge 3$ and that
	\begin{itemize}
		\item[$\circ$] $m(s)=r$ and $\Omega(s)\setminus\cB_{3^k}(r)$ contains no thin partitions,
		\item[$\circ$] $Z^{\tworow{3^k}{r}}_{\phi(s)} = Z^{\hook{3^k}{r}}_{\phi(s)} = 1$, and %$\langle \tworow{3^k}{r}\down_{P_{3^k}},\phi(s)\rangle=\hook{3^k}{r}\down_{P_{3^k}},\phi(s)\rangle=1$, and
		\item[$\circ$] $Z^\mu_{\phi(s)}\ge 2$ %$\langle \chi^\mu\down_{P_{3^k}}, \phi(s)\rangle \ge 2$ 
		for all $\mu\in\cB_{3^k}(r)\setminus\{\tworow{3^k}{r},\hook{3^k}{r}\}^\circ$, 
	\end{itemize}
	where $r=\tfrac{3^k+1}{2}$. We claim that
	\begin{itemize}
		\item[(a)] $\cR\subseteq\Omega(s,0)$ and $\Omega(s,0)\setminus\cR$ contains no thin partitions, where
		\[ \cR := \cB_{3^{k+1}}(3r)\setminus\{\tworow{3^{k+1}}{3r-1}, \hook{3^{k+1}}{3r-1} \}^\circ, \]
		\item[(b)] $Z^{\tworow{3^{k+1}}{3r}}_{\phi(s,0)} = Z^{\hook{3^{k+1}}{3r}}_{\phi(s,0)} = 1$, %$\langle \tworow{3^{k+1}}{3r}\down_{P_{3^{k+1}}}, \phi(s,0)\rangle = \langle\hook{3^{k+1}}{3r}\down_{P_{3^{k+1}}},\phi(s,0)\rangle=1$,
		\item[(c)] $m(s,1)=3r-1$ and $\Omega(s,1)\setminus\cB_{3^{k+1}}(3r-1)$ contains no thin partitions,
		\item[(d)] $Z^{\tworow{3^{k+1}}{3r-1}}_{\phi(s,1)} = Z^{\hook{3^{k+1}}{3r-1}}_{\phi(s,1)} = 1$, and %$\langle \tworow{3^{k+1}}{3r-1}\down_{P_{3^{k+1}}},\phi(s,1)\rangle=\langle \hook{3^{k+1}}{3r-1}\down_{P_{3^{k+1}}},\phi(s,1)\rangle=1$, and
		\item[(e)] $Z^\lambda_{\phi(s,\sfx)} \ge 2$ %$\langle \chi^\lambda\down_{P_{3^{k+1}}}, \phi(s,x)\rangle \ge 2$ 
		for all $\lambda\in\cB_{3^{k+1}}(3r-1)\setminus\{\tworow{3^{k+1}}{3r-1},\hook{3^{k+1}}{3r-1} \}^\circ$ and all $\sfx\in\{0,1\}$,
	\end{itemize}
	from which (i) -- (iv) follow. We proceed in steps. %: (e) follows from Step 1, (d) from Step 4 and (a), (b), (c) from Step 5. 
	Note that $\Omega(s,\sfx)^\circ=\Omega(s,\sfx)$, so below it suffices to consider $\lambda\vdash 3^{k+1}$ such that $\lambda_1\ge l(\lambda)$.
	
	\noindent\emph{Step 1.} Note (e) holds by Lemma~\ref{lem: LRa}. In particular, this implies for all $\sfx\in\{0,1\}$ that
	\[ \cB_{3^{k+1}}(3r-1)\setminus\{\tworow{3^{k+1}}{3r-1},\hook{3^{k+1}}{3r-1} \}^\circ \subseteq\Omega(s,\sfx).  \]
	
	\noindent\emph{Step 2.} Let $\sfx\in\{0,1\}$ and suppose $\lambda\in\Omega(s,\sfx)\setminus\cB_{3^{k+1}}(3r)$ is thin. Since $\lambda\in\Omega(s,\sfx)$, then $\phi(s)^{\times 3}=\phi(s,\sfx)\down_{(P_{3^k})^{\times 3}} \mid \lambda\down_{(P_{3^k})^{\times 3}}$. Hence there exist $\mu_1,\mu_2,\mu_3\in\Omega(s)$ such that $\mu_1\times\mu_2\times\mu_3 \mid \lambda\down_{(S_{3^k})^{\times 3}}$. Since $\lambda$ is thin, then so are $\mu_i\subseteq\lambda$. Moreover, either $\lambda_1>3r$ or $l(\lambda)>3r$, so there exists $i\in[3]$ such that $(\mu_i)_1>r$ or $l(\mu_i)>r$. But then $\mu_i\notin\Omega(s)$, a contradiction. Hence $\Omega(s,\sfx)\setminus\cB_{3^{k+1}}(3r)$ contains no thin partitions.
	
	\noindent\emph{Step 3.} By the same argument as Step 4 in the proof of Lemma~\ref{lem:1000}, we find that $Z^\lambda_{\phi(s,\sfx)}=\delta_{1,\sfx}+\delta_{2,\sfx}$ 
	%$$\langle \lambda\down_{P_{3^{k+1}}},\phi(s,x)\rangle=\delta_{1x}+\delta_{2x}$$
	for $\lambda\in\{\tworow{3^{k+1}}{3r-1}, \hook{3^{k+1}}{3r-1} \}$.
	%	Let $\lambda=\tworow{3^{k+1}}{3r-1}$. The only irreducible constituent(s) $\mu_1\times\mu_2\times\mu_3$ of $\chi^\lambda\down_{(S_{3^k})^3}$ such that $\mu_i\in\Omega(s)$ for all $i$ is $\tworow{3^k}{r}^{\times 3}$ with multiplicity 2, by the Littlewood--Richardson rule. By \cite[Theorem 1.5]{dBPW}, $\cX(\tworow{3^k}{r};(2,1)) \mid \lambda\down_{S_{3^k}\wrS_3}$, and $$\cX(\tworow{3^k}{r};(2,1))\down_{(S_{3^k})^{\times 3}} = \tworow{3^k}{r}^{\times 3}\cdot \chi^{(2,1)}(1) = 2\cdot \tworow{3^k}{r}^{\times 3}.$$ 
	%	Since $\chi^{(2,1)}\down_{P_3}=\phi_1+\phi_2$, we have that
	%	\begin{align*}
	%		\langle \lambda\down_{P_{3^{k+1}}},\phi(s,x)\rangle &= \langle\cX(\tworow{3^k}{r};(2,1))\down^{S_{3^k}\wrS_3}_{P_{3^{k+1}}},\phi(s,x)\rangle\\
	%		&= \langle\cX(\tworow{3^k}{r};\phi_1)\down^{S_{3^k}\wr P_3}_{P_{3^{k+1}}},\phi(s,x)\rangle + \langle\cX(\tworow{3^k}{r};\phi_2)\down^{S_{3^k}\wr P_3}_{P_{3^{k+1}}},\phi(s,x)\rangle\\
	%		&= \delta_{1x} + \delta_{2x}
	%	\end{align*}
	%	where the last equality follows from Lemma~\ref{lem: 9.6} since $\langle \tworow{3^k}{r},\phi(s)\rangle=1$. 
	%	A similar argument shows that if $\lambda=\hook{3^{k+1}}{3r-1}$ then $\langle \lambda\down_{P_{3^{k+1}}},\phi(s,x)\rangle=\delta_{1x}+\delta_{2x}$, since $\cX(\hook{3^k}{r};(2,1)) \mid \chi^\lambda\down_{S_{3^k}\wrS_3}$ by \cite[Theorem 3.5]{GTT}. 
	In particular, this proves (d).
	
	\noindent\emph{Step 4.} Let $\lambda=(3r,\mu)$ for some $\mu\vdash 3^{k+1}-3r$. Suppose $\mu\in\cB_{3r-3}(3r-4)$. Since $k\ge 3$, we may apply Proposition~\ref{prop: BinD} with $m=t=r-1$ to deduce that 
	\[ \cB_{3r-3}(3r-4)\subseteq\cD\big(3,r-1,\cP(r-1)\big). \]
	Hence there exist $\nu_1,\nu_2,\nu_3\vdash r-1$ not all equal such that $c^\mu_{\nu_1,\nu_2,\nu_3}>0$, so $c^\lambda_{(r,\nu_1),(r,\nu_2),(r,\nu_3)}>0$ by Lemma~\ref{lem: iteratedLR}. Then $(r,\nu_i)\in\cB_{3^k}(r)\subseteq\Omega(s)$, so $\lambda\in\cD(3,3^k,\Omega(s))\subseteq\Omega(s,\sfx)$ for all $\sfx$ by Lemma~\ref{lem: DinOm}.
	
	\noindent\emph{Step 5.} By the same argument as Step 5 in the proof of Lemma~\ref{lem:1000}, we find that $Z^\lambda_{\phi(s,\sfx)} = \delta_{0,\sfx}$ 
	%$$\langle \lambda\down_{P_{3^{k+1}}},\phi(s,x)\rangle=\delta_{0x}$$ 
	for $\lambda\in\{\tworow{3^{k+1}}{3r}, \hook{3^{k+1}}{3r} \}$.
	%	Consider $\lambda=\tworow{3^{k+1}}{3r}$. %, i.e.~$\mu=(3r-3)$. 
	%	The only irreducible constituent(s) $\mu_1\times\mu_2\times\mu_3$ of $\chi^\lambda\down_{(S_{3^k})^3}$ such that $\mu_i\in\Omega(s)$ for all $i$ is $\tworow{3^k}{r}^{\times 3}$ with multiplicity 1. Moreover, $\cX(\tworow{3^k}{r};(3) ) \mid \chi^\lambda\down_{S_{3^k}\wrS_3}$ by \cite[Corollary 9.1]{PW}, and $\cX(\tworow{3^k}{r};(3) )\down_{(S_{3^k})^{\times 3}} = \tworow{3^k}{r}^{\times 3}$. Hence
	%	\begin{align*}
	%		\langle \lambda\down_{P_{3^{k+1}}},\phi(s,x)\rangle &= \langle\cX(\tworow{3^k}{r};(3))\down^{S_{3^k}\wrS_3}_{P_{3^{k+1}}},\phi(s,x)\rangle\\
	%		&= \langle\cX(\tworow{3^k}{r};\phi_0)\down^{S_{3^k}\wr P_3}_{P_{3^{k+1}}},\phi(s,x)\rangle = \delta_{0x}
	%	\end{align*}
	%	where the last equality follows from Lemma~\ref{lem: 9.6} since $\langle \tworow{3^k}{r},\phi(s)\rangle=1$.
	%	A similar argument shows that if $\lambda=\hook{3^{k+1}}{3r}$ then $\langle \lambda\down_{P_{3^{k+1}}},\phi(s,x)\rangle=\delta_{0x}$, since $\cX(\hook{3^k}{r};\nu) \mid \chi^\lambda\down_{S_{3^k}\wrS_3}$ for some $\nu\in\{(3)\}^\circ$ by \cite[Theorem 3.5]{GTT}. 
	Thus we have proven (a), (b) and (c).
\end{proof}

\begin{lemma}\label{lem: LRb}
	%Let the notation be as in Lemma~\ref{lem:110x}. 
	Let $k\ge 3$ and $s=(1^{k-1}0)$. Let $r=\tfrac{3^k+1}{2}$.
	Then for all $\sfx\in\{0,1\}$,
	\[ \cB_{3^{k+1}}(3r-1)\setminus\{\tworow{3^{k+1}}{3r-1}, \hook{3^{k+1}}{3r-1}\}^\circ\subseteq\Omega(s,\sfx). \]
\end{lemma}

\begin{proof}
	\noindent\emph{Step 1.} First we show that $\cB_{2\cdot 3^k}(2r+2) = \cB\star\cB$, %by Lemma~\ref{lem: LRa} (b), Lemma~\ref{lem: LRb} (b) and Proposition~\ref{prop: B star B}, 
	where $\cB := \cB_{3^k}(r+1)\setminus\{\tworow{3^k}{r},\hook{3^k}{r}\}^\circ$. Clearly $\cB\star\cB\subseteq\cB_{2\cdot 3^k}(2r+2)$.
	Let $\lambda\in\cB_{2\cdot 3^k}(2r+2)$. Then $c^\lambda_{\mu,\nu}>0$ for some $\mu,\nu\in\cB_{3^k}(r+1)$ by Proposition~\ref{prop: B star B}. 
	
	Suppose $\nu=\tworow{3^k}{r}$ (the case if $\nu=\hook{3^k}{r}$ is similar, and the case if $\nu=\tworow{3^k}{r}'$ is obtained by conjugation). By Lemma~\ref{lem: LR tworow filling} (b) either $c^\lambda_{\mu,\omega}>0$ for some $\omega\in\cB$, or $[\lambda\setminus\mu]\in\{ [\tworow{3^k}{r}], [\tworow{3^k}{r}]^c \}$. 
	
	First suppose $[\lambda\setminus\mu]\cong [\tworow{3^k}{r}]$. Let $\mathsf{c}$ be the top left box of $[\lambda\setminus\mu]$; clearly it is addable for $\mu$. Let $\mathsf{b}$ be any removable box of $\mu$. Let $\alpha\vdash 3^k$ be defined by $[\alpha]=([\mu]\setminus\mathsf{b})\cup\mathsf{c}$. 
	Then $\cLR([\lambda\setminus\alpha])\cap\cB\ne\emptyset$ by Lemma~\ref{lem: LR tworow filling} (b), and $\alpha\in\cB$ unless $\mu=(r-1,r-1,1)$ (and $\mathsf{b}$ was in position (3,1)). In this case, $c^\lambda_{\delta,\varepsilon}>0$ where $\delta=(r,r-2,1)\in\cB$ and $\varepsilon=(r+2,r-1)\in\cB$.
	
	A similar argument applies if $[\lambda\setminus\mu]\cong [\tworow{3^k}{r}]^c$, where we instead take $\mathsf{c}$ to be the leftmost box in the top row of $[\lambda\setminus\mu]$ and $\mathsf{b}$ to be any removable box of $\mu$ not immediately to the left of $\mathsf{c}$ (such $\mathsf{b}$ always exists unless $\lambda=(3^k,3^k)$, in which case $c^\lambda_{\tworow{3^k}{r+1},\tworow{3^k}{r+1}}>0$).
	
	Finally, suppose $c^\lambda_{\mu,\omega}>0$ where $\omega\in\cB$. Since $c^\lambda_{\mu,\omega}=c^\lambda_{\omega,\mu}$, the above argument shows that we can replace $\mu$ to obtain $\rho\in\cB$ such that $c^\lambda_{\omega,\rho}>0$ (or else directly exhibit partitions verifying $\lambda\in\cB\star\cB$).
	
	\noindent\emph{Step 2.} By a similar argument to Step 1, we find that $\cB_{3^{k+1}}(3r+3) = \cB_{2\cdot 3^k}(2r+2)\star\cB$, observing that $\cB_{3^{k+1}}(3r+3)=\cB_{2\cdot 3^k}(2r+2)\star\cB_{3^k}(r+1)$ by Proposition~\ref{prop: B star B}.
	
	Let $\lambda\in \cB_{3^{k+1}}(3r-1)\setminus\{\tworow{3^{k+1}}{3r-1}, \hook{3^{k+1}}{3r-1}\}^\circ$. Thus $\lambda\in\cB^{\star 3}$, so there exist $\alpha,\beta,\gamma\in\cB$ such that $c^\lambda_{\alpha,\beta,\gamma}>0$. If $\alpha,\beta,\gamma$ are not all equal, then by Theorem~\ref{thm:mackey} $\chi^\lambda\down_{P_{3^{k+1}}}$ contains $(\alpha\down\cdot\beta\down\cdot\gamma\down)\up^{P_{3^{k+1}}}_{(P_{3^k})^{\times 3}}$ as a summand where $\down$ denotes $\down_{P_{3^k}}$, and hence 
	\[ Z^\lambda_{\phi(s,\sfx)} \ge Z^\alpha_{\phi(s)}\cdot Z^\beta_{\phi(s)} \cdot Z^\gamma_{\phi(s)} \cdot \langle \phi(s)^{\times 3}\up^{P_{3^{k+1}}}, \phi(s,\sfx)\rangle \ge 1\cdot \sum_{\sfy=0}^{p-1} \langle\phi(s,\sfy), \phi(s,\sfx)\rangle = 1 \]
	for all $\sfx\in\{0,1\}$, by Lemma~\ref{lem: easy observation}.
	Otherwise $\alpha=\beta=\gamma$, so $\cX(\alpha;\phi_i) \mid \chi^\lambda\down_{S_{3^k}\wr P_3}$ for some $i\in\{0,1,2\}$. \underline{\textit{If}} $Z^\alpha_{\phi(s)} \ge 2$, then
	$Z^\lambda_{\phi(s,\sfx)} \ge \langle \cX(\alpha;\phi_i)\down_{P_{3^{k+1}}},\phi(s,\sfx)\rangle \ge 2$
	%$$\langle \chi^\lambda\down_{P_{3^{k+1}}}, \phi(s,x)\rangle \ge \langle \cX(\alpha;\phi_i)\down_{P_{3^{k+1}}},\phi(s,x)\rangle \ge 2$$
	for all $\sfx$ by Lemma~\ref{lem: 9.5}.
	
	\noindent\emph{Step 3.} We show that $Z^\alpha_{\phi(s)} \ge 2$ %$\langle\alpha\down_{P_{3^k}},\phi(s)\rangle \ge 2$ 
	for all 
	$\alpha\in\cB_{3^k}(r+1)\setminus\{\tworow{3^k}{r+1},\hook{3^k}{r+1},\tworow{3^k}{r},\hook{3^k}{r} \}^\circ$.
	By (e) in the proof of Lemma~\ref{lem:111}, $Z^\alpha_{\phi(s)}\ge 2$ % $\langle\alpha\down_{P_{3^k}},\phi(s)\rangle \ge 2$ 
	for all $\alpha\in\cB_{3^{k}}(r)\setminus\{\tworow{3^k}{r},\hook{3^k}{r}\}^\circ$.
	Since $\chi^\rho\down_{P_n}=\chi^{\rho'}\down_{P_n}$ for all $\rho\vdash n$, without loss of generality it remains to consider $\alpha\vdash 3^k$ such that $\alpha_1=r+1=\tfrac{3^k+3}{2}$; let $s:=\tfrac{3^{k-1}+1}{2}$ so $\alpha_1=3s$. 
	
	In Step 4 of the proof of Lemma~\ref{lem:111}, we showed that if $\alpha=(3s,\varepsilon)$ where $\varepsilon\in\cB_{3s-3}(3s-4)$ (in other words $\alpha\notin\{\tworow{3^k}{r+1},\hook{3^k}{r+1} \}$) then there exist $\nu_1,\nu_2,\nu_3\vdash s-1$ not all equal such that $c^\varepsilon_{\nu_1,\nu_2,\nu_3}>0$. Furthermore,
	$$c^\alpha_{(s,\nu_1),(s,\nu_2),(s,\nu_3)}=c^\varepsilon_{\nu_1,\nu_2,\nu_3}>0$$
	and $(s,\nu_i)\in\cB_{3^{k-1}}(s)$.	
	If $Z^{(s,\nu_i)}_{\phi(s^-)}\ge 2$, %$\langle (s,\nu_i)\down_{P_{3^{k-1}}}, \phi(s^-)\rangle \ge 2$, 
	then by the same argument as the Claim %the first part of $(\star)$ 
	in the proof of Lemma~\ref{lem: LRa}, $Z^\alpha_{\phi(s)}\ge 2$. %$\langle\alpha\down_{P_{3^k}},\phi(s)\rangle \ge 2$. 
	By (e) in the proof of Lemma~\ref{lem:111}, it remains to consider the case where $(s,\nu_i)\in\{\tworow{3^{k-1}}{s},\hook{3^{k-1}}{s}\}$ for all $i$. Since they are not all equal, we may assume $(s,\nu_1)=\hook{3^{k-1}}{s}$ and $(s,\nu_2)=\tworow{3^{k-1}}{s}$, but then the same replacement argument as in Case 1 of the proof of Lemma~\ref{lem: LRa} shows that we can replace at least one of $(s,\nu_1),(s,\nu_2)$ by a partition in $\cB_{3^{k-1}}(s)\setminus\{\tworow{3^{k-1}}{s},\hook{3^{k-1}}{s}\}^\circ$. Hence $Z^\alpha_{\phi(s)}\ge 2$ %$\langle\alpha\down_{P_{3^k}},\phi(s)\rangle \ge 2$ 
	for all $\alpha\in\cB_{3^k}(r+1)\setminus\{\tworow{3^k}{r+1},\hook{3^k}{r+1},\tworow{3^k}{r},\hook{3^k}{r} \}^\circ$.
	
	\noindent\emph{Step 4.} Combining the conclusion of Step 2 with Step 3, it remains to consider
	$$\lambda\in\cB_{3^{k+1}}(3r-1)\setminus\{\tworow{3^{k+1}}{3r-1}, \hook{3^{k+1}}{3r-1}\}^\circ$$
	such that $c^\lambda_{\alpha,\alpha,\alpha}>0$ where $\alpha\in\{\tworow{3^k}{r+1},\hook{3^k}{r+1}\}^\circ= \{(r+1,r-2),(r+1,1^{r-2})\}^\circ$. Our aim is to replace the $\alpha$ component(s) to give $c^\lambda_{A,B,C}>0$ for some $A,B,C\in\cB$ which are not equal and apply Step 2.
	
	First suppose $\alpha=(r+1,r-2)$. (For clarity, we use also $\beta$ and $\gamma$ and remind the reader that $\alpha=\beta=\gamma$.) 
	Let $\mu\vdash 2\cdot 3^k$ be such that $\mu\subset\lambda$, $c^\mu_{\alpha,\beta}>0$ and $c^\lambda_{\mu,\gamma}>0$. Fix an LR filling $\mathsf{F}$ of $[\lambda\setminus\mu]$ of weight $\gamma$. If $[\lambda\setminus\mu]$ has boxes in at least three rows, then $[\lambda\setminus\mu]\not\cong[\alpha]$ or $[\alpha]^c$ by Lemma~\ref{lem:BK}. So we may replace a 1 or 2 by a 3 (or a 1 by a 2 then another 2 by a 3) to produce a filling $\mathsf{F}'$ of weight $\omega\in\cB$ such that $\omega\ne\alpha=\beta$. %$(r,r-2,1)\ne\alpha=\beta$.
	
	So now suppose $[\lambda\setminus\mu]$ occupies only two rows. By the same argument, we may assume $[\mu\setminus\alpha]$ also only occupies two rows. In particular, $\alpha=\beta=(r+1,r-2)$ implies 
	$$\mu\in\{\tworow{2\cdot 3^k}{i} \mid 2r-1\le i\le 2r+2 \} \cup \{ (2r+2,r-2,r-2) \}.$$
	If $\mu$ is one of the thin partitions just listed, then since $\lambda$ is not thin we may replace $\gamma=(r+1,r-2)$ by $\delta=(r,r-2,1)\ne\alpha=\beta$ with $c^\lambda_{\alpha,\beta,\delta}\ge c^\lambda_{\mu,\delta}>0$. If $\mu=(2r+2,r-2,r-2)$ then since $\gamma=(r+1,r-2)$, we have that 
	%$$\lambda\in\{ (2r+2,r-2+r+1,r-2+r-2), (2r+2,r-2+r+1,r-2,r-2) \} \cup \{ (3r-1-i, ,r-2) \}_{i=0}^{} $$
	%\[ \lambda\in\{ (2r+2,2r-1,2r-4), (2r+2,2r-1,r-2,r-2) \} \cup \{ (3r-1-i, 3^k+i+1,r-2) \}_{i\ge 0} \]
	\[ \lambda\in\{ (2r+2,2r-1,2r-4), (2r+2,2r-1,r-2,r-2) \} \cup \{ (3r-1-i, 2r+i,r-2) \}_{i\ge 0} \]
	(where the last term only runs over values of $i$ giving genuine partitions), %(recalling that $\lambda_1\le 3r-1$)
	which can all be seen to satisfy $c^\lambda_{A,B,C}>0$ for some $A,B,C\in\cB$ not all equal.
	
	Finally suppose $\alpha=(r+1,1^{r-2})$. Since $\lambda\in\cB_{3^{k+1}}(3r-1)$ and $\lambda\ne\hook{3^{k+1}}{3r-1}$, then there exists $\mu\vdash 2\cdot 3^k$ such that $c^\mu_{\alpha,\alpha}>0$, $c^\lambda_{\mu,\alpha}>0$ and $\mu$ is not a hook. But then $c^\mu_{\omega,\nu}>0$ where $\omega=(r+1,2,1^{r-4})\in\cB$ for some $\nu\in\cB$, so $c^\lambda_{\omega,\nu,\alpha}>0$ with $\omega\ne\alpha$ as desired.
\end{proof}

\begin{lemma}\label{lem:110x}
	Let $k\ge 3$ and $s=(1^{k-1}0)$. Let $r=\tfrac{3^k+1}{2}$. Then
	\begin{itemize}%\setlength\itemsep{0.5em}
		\item[(i)] $\cQ\subseteq\Omega(s,0)$ and $\Omega(s,0)\setminus\cQ$ contains no thin partitions, where
		\[ \cQ := \cB_{3^{k+1}}(3r+3)\setminus \{ \tworow{3^{k+1}}{3r+2}, \hook{3^{k+1}}{3r+2}, \tworow{3^{k+1}}{3r-1} \}^\circ, \]
		\item[(ii)] $Z^{\tworow{3^{k+1}}{3r+3}}_{\phi(s,0)} = Z^{\hook{3^{k+1}}{3r+3}}_{\phi(s,0)} = 1$, %$\langle \tworow{3^{k+1}}{3r+3}\down_{P_{3^{k+1}}}, \phi(s,0) \rangle = \langle \hook{3^{k+1}}{3r+3}\down_{P_{3^{k+1}}}, \phi(s,0)\rangle = 1$;
		\item[(iii)] $m(s,1)=3r+2$ and $\Omega(s,1)\setminus\cB_{3^{k+1}}(3r+2)$ contains no thin partitions, and
		\item[(iv)] $Z^{\tworow{3^{k+1}}{3r+2}}_{\phi(s,1)} = Z^{\hook{3^{k+1}}{3r+2}}_{\phi(s,1)} = 1$. %$\langle \tworow{3^{k+1}}{3r+2}\down_{P_{3^{k+1}}}, \phi(s,1) \rangle = \langle \hook{3^{k+1}}{3r+2}\down_{P_{3^{k+1}}}, \phi(s,1)\rangle = 1$.
	\end{itemize}
\end{lemma}

\begin{proof}
	\noindent\emph{Step 1.} We have that $\cB_{3^{k+1}}(3r-1)\setminus\{\tworow{3^{k+1}}{3r-1}, \hook{3^{k+1}}{3r-1}\}^\circ\subseteq\Omega(s,\sfx)$ for all $\sfx\in\{0,1\}$ from Lemma~\ref{lem: LRb}.
	
	\noindent\emph{Step 2.} For $\lambda=\tworow{3^{k+1}}{3r-1}$, the only irreducible constituent(s) $\mu_1\times\mu_2\times\mu_3$ of $\chi^\lambda\down_{(S_{3^k})^{\times 3}}$ such that $\mu_i\in\Omega(s)$ for all $i$ is $\tworow{3^k}{r+1}^{\times 3}$ with multiplicity 2. Moreover, for $\eta\vdash 3$ we have
	\[ \langle\cX\big(\tworow{3^k}{r+1};\eta), \chi^\lambda\down_{S_{3^k}\wr S_3}\rangle = \langle\cX\big((3);\eta), \chi^{(5,4)}\down_{S_3\wr S_3}\rangle = \begin{cases}
		1 & \text{if}\ \eta=(2,1),\\
		0 & \text{if}\ \eta\in\{(3),(1^3)\}
	\end{cases}\]
	by \cite[Theorem 1.2]{dBPW}. Hence, by a similar argument to Step 4 in the proof of Lemma~\ref{lem:1000}, $Z^\lambda_{\phi(s,\sfx)}=\delta_{1,\sfx}+\delta_{2,\sfx}$  %$\langle\lambda\down_{P_{3^{k+1}}},\phi(s,x)\rangle = \delta_{1x}+\delta_{2x}$ 
	by Lemma~\ref{lem: 9.6} since $Z^{\tworow{3^k}{r}}_{\phi(s)} = 1$. %$\langle\tworow{3^k}{r}\down_{P_{3^k}},\phi(s)\rangle=1$.
	
	On the other hand, $\hook{3^{k+1}}{3r-1}\down_{(S_{3^k})^{\times 3}}$ contains $\hook{3^k}{r+1}^{\times 2}\times\hook{3^k}{r-1}\in\Omega(s)^{\times 3}$ as a constituent. So $\phi(s)^{\times 3}\up^{P_{3^{k+1}}} = \sum_{i=0}^2 \phi(s,i)\up^{P_{3^{k+1}}}$ is a constituent of $\hook{3^{k+1}}{3r-1}\down_{P_{3^{k+1}}}$ by Theorem~\ref{thm:mackey}. That is, $\hook{3^{k+1}}{3r-1}\in\Omega(s,\sfx)$ for all $\sfx$.
	
	\noindent\emph{Step 3.} By a similar argument to Step 3 in the proof of Lemma~\ref{lem:1000}, we find that $\lambda\in\Omega(s,\sfx)$ for all $\sfx$, for all $\lambda\vdash 3^{k+1}$ such that $\lambda_1\in\{3r,3r+1\}$ and $\lambda\notin\{\tworow{3^{k+1}}{3r}, \tworow{3^{k+1}}{3r+1} \}$. 
	
	For $\lambda=\tworow{3^{k+1}}{3r}$, the only irreducible constituent(s) $\mu_1\times\mu_2\times\mu_3$ of $\chi^\lambda\down_{(S_{3^k})^{\times 3}}$ such that $\mu_i\in\Omega(s)$ for all $i$ is $\tworow{3^k}{r+1}^{\times 3}$ with multiplicity 4. Moreover, for all $\eta\vdash 3$ we have
	\[ \langle\cX\big(\tworow{3^k}{r+1};\eta), \chi^\lambda\down_{S_{3^k}\wr S_3}\rangle = \langle\cX\big((3);\eta), \chi^{(6,3)}\down_{S_3\wr S_3}\rangle = 1 \]
	by \cite[Theorem 1.2]{dBPW}. Hence $Z^\lambda_{\phi(s,\sfx)} %\langle\lambda\down_{P_{3^{k+1}}},\phi(s,x)\rangle 
	= 2\delta_{0,\sfx}+\delta_{1,\sfx}+\delta_{2,\sfx}$ by Lemma~\ref{lem: 9.6} since $Z^{\tworow{3^k}{r}}_{\phi(s)} = 1$.  %$\langle\tworow{3^k}{r}\down_{P_{3^k}},\phi(s)\rangle=1$.
	For $\lambda=\tworow{3^{k+1}}{3r+1}$, the only such irreducible constituent(s) is $\tworow{3^k}{r+1}^{\times 3}$ with multiplicity 3. Moreover, for $\eta\vdash 3$ we have
	\[ \langle\cX\big(\tworow{3^k}{r+1};\eta), \chi^\lambda\down_{S_{3^k}\wr S_3}\rangle = \langle\cX\big((3);\eta), \chi^{(7,2)}\down_{S_3\wr S_3}\rangle = \begin{cases}
		1 & \text{if}\ \eta\in\{(2,1),(3)\},\\
		0 & \text{if}\ \eta=(1^3)
	\end{cases}\]
	by \cite[Theorem 1.2]{dBPW}.
	Hence $Z^\lambda_{\phi(s,\sfx)} %$\langle\lambda\down_{P_{3^{k+1}}},\phi(s,x)\rangle 
	= \delta_{0,\sfx}+\delta_{1,\sfx}+\delta_{2,\sfx}$.
	
	\noindent\emph{Step 4.} By similar arguments to Step 3 in the proof of Lemma~\ref{lem:1000}, we find that $\lambda\in\Omega(s,\sfx)$ for all $\sfx$, for all $\lambda\vdash 3^{k+1}$ such that $\lambda_1\in\{3r+2,3r+3\}$ and $\lambda$ is not thin. Moreover, $Z^{\tworow{3^{k+1}}{3r+2}}_{\phi(s,\sfx)} = Z^{\hook{3^{k+1}}{3r+2}}_{\phi(s,\sfx)} = \delta_{1,\sfx}+\delta_{2,\sfx}$ and $Z^{\tworow{3^{k+1}}{3r+3}}_{\phi(s,\sfx)} = Z^{\hook{3^{k+1}}{3r+3}}_{\phi(s,\sfx)} = \delta_{0,\sfx}$.
%	$$ \langle\tworow{3^{k+1}}{3r+2},\phi(s,x)\rangle = \langle\hook{3^{k+1}}{3r+2},\phi(s,x)\rangle = \delta_{1x}+\delta_{2x}$$
%	and
%	$$ \langle\tworow{3^{k+1}}{3r+3},\phi(s,x)\rangle = \langle\hook{3^{k+1}}{3r+3},\phi(s,x)\rangle = \delta_{0x}.$$
	
	\noindent\emph{Step 5.} Finally, by the same argument as Step 2 in the proof of Lemma~\ref{lem:111}, we find that $\Omega(s,\sfx)\setminus\cB_{3^{k+1}}(3r+3)$ contains no thin partitions for all $\sfx$, completing the proof.
\end{proof}

\begin{lemma}\label{lem: LRc}
	%Let the notation be as in Lemma~\ref{lem:1100x}. 
	Let $k\ge 3$ and $s=(1^{k-1}0)$. Let $r=\tfrac{3^k+1}{2}$. 
	Further, let 
	\[ \cR = \cB_{3^{k+1}}(3r+3)\setminus \{ \tworow{3^{k+1}}{3r+2}, \hook{3^{k+1}}{3r+2}, \tworow{3^{k+1}}{3r-1} \}^\circ.\]
	Then
	\[ \cD(3,3^{k+1}, \cR) = \cD,\quad\text{where}\ \ \cD:= \cD\big(3,3^{k+1},\cR\cup\{\tworow{3^{k+1}}{3r-1}\}^\circ\big). \]
\end{lemma}

\begin{proof}
	\noindent\emph{Step 1.} Let $\lambda\in\cD$, so $c^\lambda_{\alpha,\beta,\gamma}>0$ for some $\alpha,\beta,\gamma\in \cR\cup\{\tworow{3^{k+1}}{3r-1}\}^\circ$, not all equal. We are done if $\alpha,\beta,\gamma\notin \{\tworow{3^{k+1}}{3r-1}\}^\circ$, so without loss of generality assume $\gamma=\tworow{3^{k+1}}{3r-1}$ (since $c^\lambda_{\alpha,\beta,\gamma}$ is symmetric under permutation of $\alpha,\beta,\gamma$, and since $\cD$ is closed under conjugation). Let $\mu\vdash 2\cdot 3^{k+1}$ such that $\mu\subset\lambda$, $c^\lambda_{\mu,\gamma}>0$ and $c^\mu_{\alpha,\beta}>0$.
	
	By Lemma~\ref{lem: LR tworow filling} (b), either (i) $[\lambda\setminus\mu]\in\{[\tworow{3^{k+1}}{3r-1}], [\tworow{3^{k+1}}{3r-1}]^c \}$, or (ii) there is an LR filling of $[\lambda\setminus\mu]$ of weight $\nu\in\{ (3r,3r-3), (3r-1,3r-3,1),(3r-2,3r-2,1)\}$ which can be used to replace $\gamma$, that is, $c^\lambda_{\alpha,\beta,\nu}>0$.
	
	\noindent\emph{Step 2.} If (ii) holds, we must exhibit a different (iterated) LR filling of $\lambda$ by partitions in $\cR$ if $\alpha=\beta=\nu$. By the same argument as Step 4 in the proof of Lemma~\ref{lem: LRb}, if $\nu=(3r,3r-3)$, then $\mu\in\{(6r-i,6r-6+i)\}_{i=0}^3\cup\{(6r,3r-3,3r-3) \}$, whence $c^\lambda_{A,B,C}>0$ for some $A,B,C\in\cR$ not all equal. A similar argument holds if $\nu\in\{(3r-1,3r-3,1),(3r-2,3r-2,1)\}$: we list all $\mu\vdash 2\cdot 3^k=12r-6$ such that $c^\mu_{\alpha,\beta}>0$ (and assume $c^\lambda_{\mu,\nu}>0$). If particular, if $[\mu\setminus\alpha]$ occupies at least four rows then any LR filling of $[\mu\setminus\alpha]$ of weight $\beta$ may be modified to a weight $\omega\in\cR$ with $l(\omega)\ge 4$, so $\omega\ne\beta=\alpha$. This mean we need only consider $l(\mu)\le 4$, giving the lists of possible $\mu$ as
	\begin{align*}
		% [\mu\setminus\alpha] must occupy at least 3 rows since \alpha has 3 parts, so must occupy exactly 3 rows
		% can check that cannot occupy row 5 or below, so must be 3 rows from rows 1,2,3,4
		\mu &\in \{ (6r-2-i,3r-3+b,3r-1-b+i)\mid \lceil\tfrac{i}{2}\rceil+1\le b\le 3r+\min(1-i,i-3) \}_{i=0}^4\\
		% if the 3 rows occupied are rows 1,2,3:
		% (6r-2,3r-3+b,-), 1\le b\le 3r-3
		% (6r-3,3r-3+b,-), 2\le b\le 3r-2
		% (6r-4,3r-3+b,-), 2\le b\le 3r-1
		% (6r-5,3r-3+b,-), 3\le b\le 3r-2
		% (6r-6,3r-3+b,-), 3\le b\le 3r-3)
		&\quad\cup \{(6r-2-i,6r-6+i,1^2) \}_{i=0}^2
		% rows 1,2,4
		\cup  \{(3r-1+a,3r-3,3r-3,3r+1-a)\}_{a=4}^{3r-1}\\
		% rows 1,3,4
		&\quad\cup \{(3r-1,3r-1,3r-1-i,3r-3+i)\}_{i=0}^1
		% rows 2,3,4
		% in fact A can always be chosen to be \alpha
	\end{align*}
	when $\nu=(3r-1,3r-3,1)$, and
	\begin{multline*}
		% [\mu\setminus\alpha] must occupy at least 3 rows since \alpha has 3 parts, so must occupy exactly 3 rows
		% if row 5 not occupied then 3 rows from 1,2,3,4 are occupied
		\mu\in \{ (6r-4,3r-2+b,3r-b) \}_{b=1}^{3r-2}
		% if rows 1,2,3 are occupied
		\cup \{ (6r-4,6r-4,1,1) \}\\
		% if rows 1,2,4 are occupied
		\cup \{ (3r-2+a,3r-2,3r-2,3r-a) \}_{a=2}^{3r-2}
		% can't have rows 2,3,4 occupied and row 1 not.
		% final case if row 5 is occupied, then rows 3,4,5 must be occupied, and \mu is exactly:
		\cup \{ \big((3r-2)^4,2\big) \}
		% A can usually chosen to be \alpha, except if \mu=(6r-4,6r-4,2) or (6r-4,6r-4,1,1) then need to choose A\ne\alpha in order to have B\in\cR and B\ne A
	\end{multline*}
	when $\nu=(3r-2,3r-2,1)$. These $\mu$ can all be seen to satisfy $c^\mu_{A,B}>0$ for some $A\ne B\in\cR$, whence $\lambda\in\cD(3,3^{k+1},\cR)$ as $c^\lambda_{A,B,\nu}>0$.
	
	\noindent\emph{Step 3.} Now suppose (i) holds. If $[\lambda\setminus\mu]\cong [\tworow{3^{k+1}}{3r-1}]$, we may assume $\mu\in\cB_{2\cdot 3^{k+1}}(6r+2)$, since $\cB_{2\cdot 3^{k+1}}=\cB_{3^{k+1}}(r+1)^{\star 2}$. Let $\mathsf{c}$ be the top left box of $[\lambda\setminus\mu]$ and $\mathsf{b}$ be any removable box of $\mu$. Set $\hat{\mu}\vdash 2\cdot 3^{k+1}$ to be such that $[\hat{\mu}]=([\mu]\setminus\mathsf{b})\cup\mathsf{c}$. By a similar argument to Step 1 in the proof of Lemma~\ref{lem: LRb}, either $\hat{\mu}\in\cB_{2\cdot 3^{k+1}}(6r+2)$ and $\cLR([\lambda\setminus\hat{\mu}])\cap \cR\ne\emptyset$ by Lemma~\ref{lem: LR tworow filling} (b) (or $\lambda\in\cR^{\star 3}$ directly), from which we deduce $\lambda\in\cD(3,3^{k+1},\cR)$ by the argument in Step 2.	
	The case if $[\lambda\setminus\mu]\cong [\tworow{3^{k+1}}{3r-1}]^c$ also follows similarly (using Step 1 in the proof of Lemma~\ref{lem: LRb}). %the only partitions that could cause a problem is if \lambda=((n+3^k)/2,(n+3^k/2)) and it has to be filled with ((3^k+3)/2,(3^k-3)/2) and ((n+3)/2,(n-3)/2) but we need something in \cB_n(N) and N \not\ge (n+3)/2. However, for quasi-trivial types, we always apply similar arguments to this with N \ge (n+3)/2, so this two-row lambda is not a problem. 
\end{proof}

\begin{lemma}\label{lem:1100x}
	Let $k\ge 3$ and $s=(1^{k-1}0)$. Let $r=\tfrac{3^k+1}{2}$. Then
	\begin{itemize}\setlength\itemsep{0.5em}
		\item[(i)] $\cQ\subseteq\Omega(s,0,0)$ and $\Omega(s,0,0)\setminus\cQ$ contains no thin partitions, where
		\[ \cQ:=\cB_{3^{k+2}}(9r+9)\setminus \{ \tworow{3^{k+2}}{9r+8}, \hook{3^{k+2}}{9r+8} \}^\circ, \]
		\item[(ii)] $Z^{\tworow{3^{k+2}}{9r+9}}_{\phi(s,0,0)} = Z^{\hook{3^{k+2}}{9r+9}}_{\phi(s,0,0)} = 1$, %$\langle \tworow{3^{k+2}}{9r+9}\down_{P_{3^{k+2}}}, \phi(s,0,0) \rangle = \langle \hook{3^{k+2}}{9r+9}\down_{P_{3^{k+2}}}, \phi(s,0,0)\rangle = 1$;
		\item[(iii)] $m(s,0,1)=9r+8$ and $\Omega(s,0,1)\setminus\cB_{3^{k+2}}(9r+8)$ contains no thin partitions, and
		\item[(iv)] $Z^{\tworow{3^{k+2}}{9r+8}}_{\phi(s,0,1)} = Z^{\hook{3^{k+2}}{9r+8}}_{\phi(s,0,1)} = 1$. %$\langle \tworow{3^{k+2}}{9r+8}\down_{P_{3^{k+2}}}, \phi(s,0,1) \rangle = \langle \hook{3^{k+2}}{9r+8}\down_{P_{3^{k+2}}}, \phi(s,0,1)\rangle = 1$.
	\end{itemize}
	%	or equivalently, just shifting $k$ and $s$: makes it more obvious that Lemma~\ref{lem:110x} feeds into this; also change $r$
	%	Let $k\ge 4$ and $s=(1^{k-2}00)$. Let $r=\tfrac{3^{k-1}+1}{2}$. Then
	%	\begin{itemize}\setlength\itemsep{0.5em}
	%		\item[(i)] $\Omega(s,0) = (\cB_{3^{k+1}}(9r+9)\setminus \{ \tworow{3^{k+1}}{9r+8}, \hook{3^{k+1}}{9r+8} \}^\circ) \sqcup\mathcal{NT}$;
	%		\item[(ii)] $\langle \tworow{3^{k+1}}{9r+9}\down_{P_{3^{k+1}}}, \phi(s,0) \rangle = \langle \hook{3^{k+1}}{9r+9}\down_{P_{3^{k+1}}}, \phi(s,0)\rangle = 1$;
	%		\item[(iii)] $\Omega(s,1) = \cB_{3^{k+1}}(9r+8)\sqcup\mathcal{NT}$; and
	%		\item[(iv)] $\langle \tworow{3^{k+1}}{9r+8}\down_{P_{3^{k+1}}}, \phi(s,1) \rangle = \langle \hook{3^{k+1}}{9r+8}\down_{P_{3^{k+1}}}, \phi(s,1)\rangle = 1$.
	%	\end{itemize}
\end{lemma}

\begin{proof}
	\noindent\emph{Step 1.} By the same argument as Step 2 in the proof of Lemma~\ref{lem:111}, $\Omega(s,0,x)\setminus\cB_{3^{k+2}}(9r+9)$ contains no thin partitions.
	
	\noindent\emph{Step 2.} Let $\cR=\cB_{3^{k+1}}(3r+3)\setminus \{ \tworow{3^{k+1}}{3r+2}, \hook{3^{k+1}}{3r+2}, \tworow{3^{k+1}}{3r-1} \}^\circ$. 
	By Proposition~\ref{prop: BinD}, Lemma~\ref{lem: LRc} and Lemma~\ref{lem: DinOm}, we have for all $\sfx\in\{0,1\}$ that
	\begin{multline*}
		\cB_{3^{k+1}}\big(3\cdot(3r+1)-1\big)\subseteq\cD\big( 3,3^{k+1},\cB_{3^{k+1}}(3r+1)\big) \subseteq\cD(3,3^{k+1},\cR\cup\{\tworow{3^{k+1}}{3r-1} \}^\circ)\\
		=\cD(3,3^{k+1},\cR) \subseteq\cD(3,3^{k+1},\Omega(s,0)) \subseteq\Omega(s,0,\sfx).
	\end{multline*}
	
	\noindent\emph{Step 3.} By the same arguments as in Steps 3--5 in the proof of Lemma~\ref{lem:1000}, $\lambda\in\Omega(s,0,\sfx)$ for all $\sfx$, for all $\lambda\vdash 3^{k+2}$ such that $9r+3\le\lambda_1\le 9r+9$ %holds for $9r+1\le\lambda_1\le 9r+9$ in fact
	and 
	\[ \lambda\notin\{\tworow{3^{k+2}}{9r+8}, \hook{3^{k+2}}{9r+8}, \tworow{3^{k+2}}{9r+9}, \hook{3^{k+2}}{9r+9} \}. \]
	Moreover, 
	\[ Z^\lambda_{\phi(s,0,\sfx)} = \begin{cases}
		\delta_{0,\sfx} & \text{if}\ \lambda\in\{ \tworow{3^{k+2}}{9r+9}, \hook{3^{k+2}}{9r+9}\},\\
		\delta_{1,\sfx}+\delta_{2,\sfx} & \text{if}\ \lambda\in\{ \tworow{3^{k+2}}{9r+8}, \hook{3^{k+2}}{9r+8}\}.
	\end{cases} \]
	Since $\Omega(s,\sfx)^\circ=\Omega(s,\sfx)$, this concludes the proof.
\end{proof}

\begin{proposition}\label{prop:case3}
	Let $k\ge 4$ and $\tfrac{3^k+9}{2}\le m \le 3^k-3$. Suppose $s\in\{0,1\}^k$ satisfies $\cQ\subseteq \Omega(s)$ and $\Omega(s)\setminus\cQ$ contains no thin partitions where 
	\[ \cQ:=\cB_{3^k}(m)\setminus\{ \tworow{3^{k}}{m-1}, \hook{3^{k}}{m-1} \}^\circ,\]
	and $Z^{\tworow{3^k}{m}}_{\phi(s)} = Z^{\hook{3^k}{m}}_{\phi(s)} = 1$. Then
	\begin{itemize}
		\item[(i)] $\cR\subseteq\Omega(s,0)$ and $\Omega(s,0)\setminus\cR$ contains no thin partitions, where
		\[ \cR:=\cB_{3^{k+1}}(3m)\setminus\{ \tworow{3^{k+1}}{3m-1}, \hook{3^{k+1}}{3m-1} \}^\circ,\]
		\item[(ii)] $Z^{\tworow{3^{k+1}}{3m}}_{\phi(s,0)} = Z^{\hook{3^{k+1}}{3m}}_{\phi(s,0)} = 1$,
		\item[(iii)] $m(s,1)=3m-1$ and $\Omega(s,1)\setminus \cB_{3^{k+1}}(3m-1)$ contains no thin partitions, and
		\item[(iv)] $Z^{\tworow{3^{k+1}}{3m-1}}_{\phi(s,1)} = Z^{\hook{3^{k+1}}{3m-1}}_{\phi(s,1)} = 1$.
	\end{itemize}
\end{proposition}

\begin{proof}
	Since $\Omega(s,\sfx)^\circ=\Omega(s,\sfx)$, it suffices to consider $\lambda\vdash 3^{k+1}$ such that $\lambda_1\ge l(\lambda)$.
	
	\noindent\emph{Step 1.} By the same argument as Step 2 in the proof of Lemma~\ref{lem:111}, $\Omega(s,\sfx)\setminus\cB_{3^{k+1}}(3m)$ contains no thin partitions for all $\sfx\in\{0,1\}$.
	
	\noindent\emph{Step 2.} Since $m\ge\tfrac{3^k+9}{2}$, then $m-2>\tfrac{3^k}{2}+1$. So by Proposition~\ref{prop: BinD} and Lemma~\ref{lem: DinOm},
	\[ \cB_{3^{k+1}}(3m-7) \subseteq\cD\big(3,3^k,\cB_{3^k}(m-2)\big) \subseteq\cD(3,3^k,\Omega(s))\subseteq\Omega(s,\sfx) \]
	for all $\sfx\in\{0,1\}$.
	 
	\noindent\emph{Step 3.} By the same argument as Step 3 in the proof of Lemma~\ref{lem:1000}, if $\lambda\vdash 3^{k+1}$ and $3m-6\le \lambda_1\le 3m-2$, or $\lambda=(3m-1,\mu)$ where $\mu\in\cB_{3^{k+1}-3m+1}(3^{k+1}-3m)$, then $\lambda\in\cD(3,3^k,\Omega(s))\subseteq\Omega(s,\sfx)$ for all $\sfx$.
	
	\noindent\emph{Step 4.} By the same argument as Step 4 in the proof of Lemma~\ref{lem:1000}, $Z^\lambda_{\phi(s,\sfx)}=\delta_{1,\sfx}+\delta_{2,\sfx}$ %we find that
	%$$\langle \lambda\down_{P_{3^{k+1}}},\phi(s,x)\rangle=\delta_{1x}+\delta_{2x}$$
	for $\lambda\in\{\tworow{3^{k+1}}{3m-1}, \hook{3^{k+1}}{3m-1} \}$.
	
	\noindent\emph{Step 5.} Let $\lambda=(3m,\mu)$ for some $\mu\vdash 3^{k+1}-3m$. Since $\cB_{3^{k+1}-3m}(3^{k+1}-3m-1)\subseteq\cD\big(3,3^k-m,\cP(3^k-m)\big)$ by Proposition~\ref{prop: BinD}, we have that $\lambda\in\Omega(s,\sfx)$ for all $\sfx$ if $\mu\notin\{(3^{k+1}-3m) \}^\circ$, by the same argument as Step 4 in the proof of Lemma~\ref{lem:111}.
	
	\noindent\emph{Step 6.} By the same argument as Step 5 in the proof of Lemma~\ref{lem:1000}, $Z^\lambda_{\phi(s,\sfx)}=\delta_{0,\sfx}$ %we find that 
	%$$\langle \lambda\down_{P_{3^{k+1}}},\phi(s,x)\rangle=\delta_{0x}$$ 
	for $\lambda\in\{ \tworow{3^{k+1}}{3m}, \hook{3^{k+1}}{3m}\}$.
\end{proof}

We can now deduce Theorem~\ref{thm:b-primepower}.

\begin{proof}[Proof of Theorem~\ref{thm:b-primepower}]
	The claims are clear if $k\le 3$ so we may assume $k\ge 4$. The value of $M(s)$ follows from Theorem~\ref{thm:M}. It remains to verify the values of $m(s)$. If $s\in U_k(0)$ then $\Omega(s)$ and hence $m(s)$ was determined in Theorem~\ref{thm:GL1A}. Otherwise $m(s)$ may be calculated using the results in Section~\ref{sec:prime-power} according to Figure~\ref{fig:3} and the forms of the leading subsequences of $s$.
	
	To illustrate one example, suppose $s=(\sfs_1,\sfs_2,\dotsc,\sfs_k) = (1,0,0,\sfs_4,\dotsc,\sfs_k)$. Then $F(s)=f(s)=1$, and $G(s)=g(s)$ if $s$ contains a second 1. 
	
	\noindent$\bullet$ If $s\in U_k(1)$ (so $z=F(s)$), or $s\in U_k(2)$ and $s_k=1$ (so $z=F(s)+1$ and $G(s)=k$), then the result follows from Lemma~\ref{lem:1000}.
		
	\noindent$\bullet$ If $s\in U_k(2)$ but $s_k=0$ (so $z=F(s)+1$), then $m(t)=2\cdot 3^{G(s)-1}-1$ where $t=(\sfs_1,\dotsc,\sfs_{G(s)})$ by Lemma~\ref{lem:1000}. Moreover, $t$ satisfies the conditions of Proposition~\ref{prop:case2}, so $m(t,0)=3m(t)$. Further, $(t,0)$ again satisfies the conditions of Proposition~\ref{prop:case2}. By repeated application of the proposition, we find that $m(s)=3^{k-G(s)}\cdot m(t)=2\cdot 3^{k-1}-3^{k-G(s)}= \tfrac{3^k+3^{k-F(s)}}{2} - 3^{k-G(s)}$.
		
	\noindent$\bullet$ Finally, if $s\in U_k(z)$ with $z\ge 3$ (so $z\ge F(s)+2$), then $m(t)=2\cdot 3^{H(s)-1}-3^{H(s)-G(s)}$ where $t=(\sfs_1,\dotsc,\sfs_{H(s)-1})$, and $t$ satisfies the conditions of Proposition~\ref{prop:case2}. Thus $m(t,1)=3m(t)-1$, and $(t,1)$ satisfies the conditions of Proposition~\ref{prop:4.8}. By repeated application %s ($k-H(s)$ times) of Proposition~\ref{prop:4.8}, 
	we deduce $m(s)=3^{k-H(s)}\cdot m(t,1)=2\cdot 3^{k-1}-3^{k-G(s)}-3^{k-H(s)}=\tfrac{3^k+3^{k-F(s)}}{2} - 3^{k-G(s)} - 3^{k-H(s)}$.
\end{proof}
	
We collect in the following statement the results obtained so far on $\Omega(s)$ when $\phi(s)$ is quasi-trivial, from which we immediately deduce Theorem~\ref{thm:a} in the case that $n$ is a power of 3. Recall that $\Omega(0^k)$ is known from Theorem~\ref{thm:GL1A}: it is $\cP(3^k)\setminus\{(3^k-1,1)\}^\circ$ unless $k=2$, in which case $\Omega(0,1)=\cP(9)\setminus\{(8,1),(5,4),(4,3,2)\}^\circ$.

\begin{theorem}\label{thm:a-prime power}
	Let $k\in\N$ and $\phi(s)\in\Lin(P_{3^k})\setminus\{\triv_{P_{3^k}} \}$ be quasi-trivial, that is, $s\in U_k(1)$. Let $a=3^{k-1}$ and $f=f(s)$. Then
	\begin{small}
		\[ \Omega(s) = \begin{cases}
			\cB_{3^k}(3^k-1) & \text{if}\ f(s)=k\ne 2,\\
			\cB_9(8)\setminus\{(3^3)\} & \text{if}\ s = (0,1),\\
			\cB_{3^k}(2a)\setminus \{ (2a,a-1,1), (2a,2,1^{a-2}), \tworow{3^k}{2a-1}, \hook{3^k}{2a-1} \}^\circ & \text{if}\ f(s)=1\ \text{and}\ k\notin\{1,3\},\\
			\cB_{27}(17)\setminus\{(17,10),(17,1^{10}),(14,13)\}^\circ \sqcup \{(18,\mu) \mid \mu\in\Omega(0,0)\}^\circ & \text{if}\ s = (1,0,0),\\
			\cB_{3^k}(3^k-3^{k-f}-1)\sqcup \{ (3^k-3^{k-f},\mu) \mid \mu\in \Omega(0^{k-f}) \}^\circ  & \text{otherwise}.
			%\cB_{3^k}(3^k-3^{k-f}-1)\sqcup \{ (3^k-3^{k-f},\mu) \mid \mu\in \Omega(s_{f+1},\dotsc,s_k) \}^\circ  & \text{otherwise}. % Lemma~\ref{lem:4.7}
		\end{cases} \]
	\end{small}
	%where %$a=3^{k-1}$, $f=f(s)$ and
	%$\Omega(s_{f(s)+1},\dotsc,s_k)=\Omega(0,\dotsc,0)$ is given by \cite[Theorem A]{GL1}.
\end{theorem}

\begin{proof}
	The statement follows from direct computation for small $k$, Lemma~\ref{lem:1000} if $f(s)=1$ and $k\ge 4$, Lemma~\ref{lem: 001} if $f(s)=k\ge 3$, and Lemma~\ref{lem:4.7} otherwise.
\end{proof}

\bigskip

%%%%%%%%%%%%%%%%%%%%%%%%%%%%%%%%%%%%%%%%%%%%%%%%%%%%%%%%%%%%%%%%%%%%%%%%%
\section{Proofs of the main results}\label{sec:all-n}

Let $k\in\N_{\ge 3}$ and $s\in\{0,1\}^k$.
As a helpful summary for the reader, we collect all of the relevant properties of $\Omega(s)$ and $m(s)$ that we have proved in Section~\ref{sec:prime-power} (along with Theorem~\ref{thm:GL1A} for $\Omega(\triv_{P_n})$) in Table~\ref{table:summary} below, emphasising those rows corresponding to quasi-trivial $s$ in bold. Note we have abbreviated $\tworow{3^k}{t}$ and $\hook{3^k}{t}$ to simply $\tworow{}{t}$ and $\hook{}{t}$, for various $t$. % since the size of the partition ($3^k$) is understood.) 
These properties of $\Omega(s)$ and $m(s)$ served to motivate the definitions and naming of the types $\sigma(s)$ introduced in Definition~\ref{def: typephi3}. 

\begin{table}[h]
	\centering
	\begin{small}
		\[ \def\arraystretch{1.2}
		\begin{array}{|c|cl|c|}
			\hline
			\sigma(s) & \Omega(s) & m:=m(s) & N:=N(s)\\
			\hline
			\hline
			\bm{1} & \bm{\cP(3^k)\setminus\{(3^k-1,1)\}^\circ} & \bm{3^k-2} & \bm{3^k}\\
			10 & \supseteq\cB_{3^k}(m+2)\setminus\{\tworow{}{m+1},\hook{}{m+1} \}^\circ\ \text{a.c.n.o.t.p} & \tfrac{3^k+3^{k-F(s)}}{2}-2 & m(s)+2\\ 
			11 & \supseteq\cB_{3^k}(m+2)\setminus\{\tworow{}{m+1},\hook{}{m+1} \}^\circ\ \text{a.c.n.o.t.p} & \tfrac{3^k-1}{2} & m(s)+2\\
			\bm{2} & \bm{\cB_{3^k}(3^k-1)} & \bm{3^k-1} & \bm{3^k-1}\\
			21 & \supseteq\cB_{3^k}(m(s))\ \text{a.c.n.o.t.p} & \tfrac{3^k+1}{2} & m(s)\\
			\bm{3} & \bm{\cB_{3^k}(N-1)\sqcup\{(N,\mu) : \mu\in\Omega(\triv_{P_{3^k-N}}) \}^\circ} & \bm{3^k-3^{k-f(s)}-1} & \bm{\scriptstyle{3^k-3^{k-f(s)}\ (k-f(s)\ge 3)}}\\
			\bm{30} & \bm{\cB_{3^k}(3^k-10)\sqcup\{(3^k-9,\mu) : \mu\in\Omega(\triv_{P_9})\}^\circ} & \bm{3^k-10} & \bm{3^k-9}\\
			\bm{31} & \bm{\cB_{3^k}(3^k-4)\sqcup\{(3^k-3),(3^k-3,1^3) \}^\circ} & \bm{3^k-4} & \bm{3^k-3}\\
			\bm{5} & \bm{\cC} & \bm{2\cdot3^{k-1}-2} & \bm{2\cdot 3^{k-1}}\\
			\bm{6} & \bm{\Omega(1,0,0)} & & \bm{18}\\
			7 & \cB_{3^k}(m+5)\setminus\{\tworow{}{m+4},\hook{}{m+4},\tworow{}{m+1} \}^\circ & \tfrac{3^k-1}{2} & m(s)+5\\
			\hdashline
			22 & \supseteq\cB_{3^k}(m(s))\ \text{a.c.n.o.t.p}  & \text{(see Remark~\ref{rem: 22})} & m(s)\\
			\hline
		\end{array}\]
	\end{small}
	\caption{Summary of $\sigma(s)$, $\Omega(s)$, $m(s)$ and $N(s)$. Here `and contains no other thin partitions' is abbreviated to \emph{a.c.n.o.t.p}. Moreover, $\cC:=\cB_{3^k}(2a)\setminus\{(2a,a-1,1),(2a,2,1^{a-2}),(2a-1,a+1),(2a-1,1^{a+1}) \}^\circ$ where $a=3^{k-1}$ and $k\ge 4$, and $\Omega(1,0,0)$ was explicitly calculated in Section~\ref{sec:prime-power}.}\label{table:summary}
\end{table}

\begin{remark}\label{rem: 22}
	For $\sigma(s)=22$ in Table~\ref{table:summary}, $m(s)$ is one of the following values by Theorem~\ref{thm:b-primepower}: %Theorem B$'$ for $n=3^k$:
	\[ 3^k-3^{k-f(s)}-3^{k-g(s)}\ ,\qquad\tfrac{3^k+3^{k-F(s)}}{2}-3^{k-G(s)}\ ,\qquad \tfrac{3^k+3^{k-F(s)}}{2}-3^{k-G(s)}-3^{k-H(s)}.\]
	We remark that $\tfrac{3^k+3}{2} \le m(s)=N(s) \le 3^k-2$ in all of these cases.
	\hfill$\lozenge$
\end{remark}

\medskip

%========================================================================
\subsection{Quasi-trivial characters}\label{sec:a-full}
The main aim of this section is to prove Theorem~\ref{thm:a}. 

\begin{notation}\label{not:qt}
	Throughout Section~\ref{sec:a-full}, we fix the following notation (recall Definition~\ref{def: all p=3}): let $n\in\N$ be divisible by 27 and $l\in\{0,1,\dotsc,26\}$. Let $\phi\times\psi\in\Lin(P_{n+l})$ be quasi-trivial, where $\phi\in\Lin(P_n)$ and $\psi\in\Lin(P_l)$. Suppose $\phi$ corresponds to $\{s_1,\dotsc,s_R\}$, so $R\ge 1$ and $\sigma(s_i)\in\{1,2,3,30,31,5,6\}$ for all $i\in[R]$. Let $\sigma(\phi)=(y_1,y_2,y_3,y_{30},y_{31},y_5,y_6)$ where $y_i:=\#\{j\in[R] \mid \sigma(s_j)=i \}$.
\end{notation}

We proceed by substituting the various forms of $\Omega(s_i)$ %which we have found in Section~\ref{sec:prime-power} 
into Lemma~\ref{lem: omega star omega}, in order to first compute $\Omega(\phi)$, then $\Omega(\phi\times\psi)$.
The structure of the proof is as follows (recalling that the case of $R=1$ and $l=0$ was completely determined in Section~\ref{sec:prime-power}):

\begin{itemize}
	\item We first find $\Omega(\phi)$ if $y_i=0$ for all $i\notin\{1,2\}$ (Lemma~\ref{lem: 12}).
	\item We then find $\Omega(\phi)$ if $y_i=0$ for all $i\notin\{3,30,31\}$ (Lemma~\ref{lem: 33031}).
	\item Combining, we find $\Omega(\phi)$ whenever $y_i=0$ for all $i\notin\{1,2,3,30,31\}$ (Lemma~\ref{lem: 1233031}).%, i.e.~$y_5=y_6=0$.
	\item Then we find $\Omega(\phi)$ if $y_i=0$ for all $i\notin\{5,6\}$ (Lemma~\ref{lem: 56}).
	\item Combining, we find $\Omega(\phi)$ for all $\phi\in\Lin(P_n)$ (Proposition~\ref{prop:qt}).
	\item Finally, we find $\Omega(\phi\times\psi)=\Omega(\phi)\star\Omega(\psi)$ %(Theorems~\ref{theorem: pre A'} and~\ref{theorem: pre A' R=1})
	(Theorem~\ref{thm:pre-A'}), and hence deduce Theorem~\ref{thm:a}.
\end{itemize}

\begin{lemma}\label{lem: 12}
	Suppose $\sigma(\phi)=(y_1,y_2,0,0,0,0,0)$. Then
	\[ \Omega(\phi) = \begin{cases}
		\cP(n)\setminus\{(n-1,1)\}^\circ & \text{if}\ R=1\ \text{and}\ y_1=1,\\
		\cB_n(N(\phi)) & \text{otherwise}.
	\end{cases} \]	
\end{lemma}
\begin{proof}
	This follows from Theorem~\ref{thm:GL1A}, Proposition~\ref{prop: B star B}, Lemma~\ref{lem: 10.3} and Lemma~\ref{lem:4.7}.
	%the same argument as in the proof of \cite[Theorem 5.4]{GL2} ($\tau(s_i)=1,2,3$ when $p\ge 5$ \cite[Definition 2.6]{GL2} are analogous to $\sigma(s_i)=1,3,2$ when $p=3$ respectively).
\end{proof}

\begin{lemma}\label{lem: 33031}
	Suppose $\sigma(\phi)=(0,0,y_3,y_{30},y_{31},0,0)$. Then
	\[ \Omega(\phi) = \begin{cases}
		\cB_n(N(\phi)-1)\sqcup\{(N(\phi),\mu) \mid \mu\in\Omega(\triv_{P_{n-N(\phi)}}) \}^\circ & \text{if}\ R=1,\ \text{or}\ R=2\ \text{and}\ y_{31}=2,\\
		\cB_n(N(\phi)) & \text{otherwise}.
	\end{cases} \]
\end{lemma}

\begin{proof}
	The case of $R=1$ follows from Lemma~\ref{lem:4.7}, so now we may assume $R\ge 2$ and proceed by induction on $R$.
	First suppose $R=2$. If $y_3=2$, then the statement follows from the same argument as the $(\tau(s_1),\tau(s_2))=(2,2)$ case in the proof of \cite[Theorem 5.4]{GL2}. If $y_3=y_{30}=1$, then a similar argument holds, observing that $n-N(\phi)=3^{k_i-f(s_i)}+9$ for some $i$, and $\Omega(\triv_{P_{3^{k_i-f(s_i)}}})\star\Omega(\triv_{P_9})=\Omega(\triv_{P_{n-N(\phi)}})$ by Lemma~\ref{lem: omega star omega} since $k_i-f(s_i)\ge 3$ (i.e.~$3^{k_i-f(s_i)} + 9$ is the 3-adic expansion of $n-N(\phi)$). Finally, $\Omega(\triv_{P_{n-N(\phi)}}) = \cP(n-N(\phi))$ by Theorem~\ref{thm:GL1A}. The cases of $y_3=y_{31}=1$, $y_{30}=2$ and $y_{30}=y_{31}=1$ follow similarly. However, notice when $y_{31}=2$ then $\Omega(\triv_{P_{n-N(\phi)}})=\Omega(\triv_{P_6})\ne\cP(6)$.
	
	Now suppose $R\ge 3$. If $R=3$ and $y_3=31$, then the statement follows from the same argument as above, by observing that $\Omega(\triv_{P_6})\star \Omega(\triv_{P_{3^t}})=\cP(3^t+6)$ for all $t\in\N$ by direct calculation if $t=1$, while if $t\ge 2$ then noticing that $3^t+6$ is a 3-adic expansion and hence $\Omega(\triv_{P_6})\star \Omega(\triv_{P_{3^t}})=\Omega(\triv_{P_{{3^t}+6}})=\cP(3^t+6)$.
	
	Otherwise, we may assume that $\phi=\phi(s)\times\tilde{\phi}$ where $s\in\{0,1\}^k$ for some $k$ and $\sigma(s)\in\{3,30,31\}$, and $\sigma(\tilde{\phi})=(0,0,\tilde{y}_3,\tilde{y}_{30},\tilde{y}_{31},0,0)\ne(0,0,0,0,2,0,0)$. By the inductive hypothesis, we have $\Omega(\tilde{\phi}) = \cB_{n-3^k}(N(\tilde{\phi}))$. The statement of the present lemma then follows by applying the same argument as in the $(\tau(s_1),\tau(s_2))=(2,3)$ case in the proof of \cite[Theorem 5.4]{GL2}.
\end{proof}

Recall the definition of the set $E$ from Definition~\ref{def: all p=3} (iii).
\begin{lemma}\label{lem: 1233031}
	Suppose $\sigma(\phi)=(y_1,y_2,y_3,y_{30},y_{31},0,0)$. Then
	\[ \Omega(\phi) = \begin{cases}
		\cP(n)\setminus\{(n-1,1)\}^\circ & \text{if}\ R=1\ \text{and}\ y_1=1,\\
		\cB_n(N(\phi)-1)\sqcup\{(N(\phi),\mu) \mid \mu\in\Omega(\triv_{P_{n-N(\phi)}}) \}^\circ & \text{if}\ \sigma(\phi)\in E,\\
		\cB_n(N(\phi)) & \text{otherwise}.
	\end{cases} \]
\end{lemma}

\begin{proof}
	As usual, we may assume $R\ge 2$. Let $\phi=\pi\times\rho$ where $\rho\in\Lin(P_{n'})$, $\pi\in\Lin(P_{n-n'})$,  $\sigma(\pi)=(y_1,y_2,0,0,0,0,0)$ and $\sigma(\rho)=(0,0,y_3,y_{30},y_{31},0,0)$. By Lemmas~\ref{lem: 12} and~\ref{lem: 33031}, we are done if either $\phi=\pi$ or $\phi=\rho$. Observe that $\Omega(\phi)=\Omega(\pi)\star\Omega(\rho)$ by Lemma~\ref{lem: omega star omega}. %, and $\sigma(\phi)=\sigma(\pi)+\sigma(\rho)$ entry-wise.
	
	By Proposition~\ref{prop: B star B} and Lemma~\ref{lem: 10.3}, the statement follows if $\Omega(\rho)=\cB_{n'}(N(\rho))$, %where $\rho\in\Lin(P_{n'})$, 
	since $N(\rho)>\tfrac{n'}{2}$ (see Definition~\ref{def: typephi3} and Table~\ref{table:summary}). So now we may assume 
	\[ \Omega(\rho)=\cB_{n'}(N(\rho)-1)\sqcup\{(N(\rho),\nu) \mid \nu\in\Omega(\triv_{P_{n'-N(\rho)}}) \}^\circ, \]
	in other words $\sigma(\rho)\in\{(0,0,1,0,0,0,0),(0,0,0,1,0,0,0),(0,0,0,0,1,0,0),(0,0,0,0,2,0,0)\}$.
	In particular, $n'-N(\rho)\in\{3^s \mid s\in\N \}\cup\{6\}$ by Lemma~\ref{lem: 33031}.
	
	Looking at the statements of Lemmas~\ref{lem: 12} and~\ref{lem: 33031}, the key is to describe the following, for natural numbers $t>\tfrac{m}{2}$ and $N-1>\tfrac{q}{2}$:
	\begin{equation}\label{eqn: not ball}
		X:= \cB_m(t) \star \big( \cB_{q}(N-1)\sqcup \{(N,\nu) \mid \nu\in\Omega(\triv_{P_{q-N}}) \}^\circ \big).
	\end{equation}
	By Proposition~\ref{prop: B star B}, we have that $\cB_{m+q}(t+N-1)\subseteq X\subseteq\cB_{m+q}(t+N)$. Observe $X = X^\circ$, so it suffices to consider $\lambda=(t+N,\mu)$ where $\mu\vdash m-t+q-N$. %$\lambda=(N,\mu)$ where $\mu\vdash q-N$. 
	Indeed, $\lambda\in X$ if and only if there exist $\alpha\vdash m-t$ and $\beta\in\Omega(\triv_{P_{q-N}})$ such that $c^\mu_{\alpha,\beta}>0$, whence $c^\lambda_{(t,\alpha),(N,\beta)}>0$ by Lemma~\ref{lem: iteratedLR} (notice that $(t,\alpha)\in\cB_m(t)$ since $\alpha_1\le t$ and $l((t,\alpha))<t$ because $t>\tfrac{m}{2}$).
	
	\begin{enumerate}
		\item[(1)] If $m=t$, then this simplifies to $\lambda\in X$ if and only if $\lambda=(t+N,\mu)$ such that $\mu\in\Omega(\triv_{P_{q-N}})$.
		\item[(2)] If $m-t=1$, then $\lambda\in X$ if and only if $\mu\in\cP(1)\star\Omega(\triv_{P_{q-N}})$.
		\item[(3)] If $m-t\ge 2$, then $\lambda\in X$ if and only if $\mu\in\cP(m-t)\star\Omega(\triv_{P_{q-N}})$ also.\footnote{In fact, (1) %in the proof of Lemma~\ref{lem: 1233031}
		dealing with the case $m-t=0$ is similar to \cite[Lemma 5.3]{GL2}, while (2) and (3) dealing with $m-t\ge 1$ is similar to the $(\tau(s_1),\tau(s_2))=(2,3)$ case in the proof of \cite[Theorem 5.4]{GL2}. 
		There was no difference between $m-t=1$ and $m-t\ge 2$ when $p\ge 5$ because $\cP(1)\star\Omega(\triv_n)=\cP(n+1)$ for \underline{all} $n\in\N$ when $p\ge 5$ (either $\Omega(\triv_{P_n})=\cP(n)$ if $n\ne p^k$ and use Proposition~\ref{prop: B star B}, or $n=p^k$ and $\Omega(\triv_{P_{p^k}})=\cP(p^k)\setminus\{(p^k-1,1)\}^\circ$, but then $\cP(1)\star\Omega(\triv_{P_{p^k}})=\cP(p^k+1)$ directly). However, this is not true when $p=3$, which is why (2) and (3) must be differentiated here.}
	\end{enumerate}
	First suppose $\Omega(\pi)=\cB_{n-n'}(N(\pi))$. We apply (1)--(3) with $m=n-n'$, $t=N(\pi)$, $q=n'$ and $N=N(\rho)$ (observing that $t>m/2$ and $N-1>q/2$ are satisfied by Definition~\ref{def: typephi3}).
	\begin{itemize}
		\item[$\circ$] If $N(\pi)=n-n'$: this occurs exactly when $\pi=\triv_{P_{n-n'}}$, i.e.~$\sigma(\pi)=(y_1,0,0,0,0,0,0)$. That is, $\sigma(\phi)=\sigma(\pi)+\sigma(\rho)$ is exactly one of the first four elements of $E$. Moreover, by (1) we have that
		$\Omega(\phi)=\cB_n(N(\phi)-1)\sqcup\{(N(\phi),\mu) \mid \mu\in\Omega(\triv_{P_{n-N(\phi)}}) \}^\circ$.
		
		\item[$\circ$] If $n-n'-N(\pi)=1$: this occurs exactly when $\sigma(\pi)=(y_1,1,0,0,0,0,0)$. Moreover, $\Omega(\phi)=\cB_n(N(\phi))$ by (2) \emph{unless} $\mu\in\cP(1)\star\Omega(\triv_{P_{q-N}})\ne\cP(1+q-N)$. Since $q-N=n'-N(\rho)\in\{3^s\mid s\in\N\}\cup\{6\}$, we have by Theorem~\ref{thm:GL1A} that $\mu\in\cP(1)\star\Omega(\triv_{P_{q-N}})\ne\cP(1+q-N)$ when $n'-N(\rho)\in\{3,9\}$, i.e.~when $\sigma(\rho)\in \{(0,0,0,1,0,0,0),(0,0,0,0,1,0,0) \}$.
		That is, $\sigma(\phi)=\sigma(\pi)+\sigma(\rho)$ is exactly the fifth or sixth element of $E$.
		
		\item[$\circ$] Otherwise, $n-n'-N(\pi)\ge 2$: notice $\sigma(\pi)\notin E$ and so $\sigma(\phi)\notin E$. In this case $\lambda\in X$ for all $\mu$ %$\mu\vdash q-N$ 
		by (3) and Lemma~\ref{lem: P star 00}, so $\Omega(\phi)=\cB_n(N(\phi))$.
	\end{itemize}
	(Note that since $y_5=y_6=0$, it is not possible for $\sigma(\phi)$ to be the seventh, eighth or ninth element of $E$.)
	Finally, the case of $\Omega(\pi)=\cP(3^k)\setminus\{(3^k-1,1)\}^\circ$ for some $k\ge 3$ follows similarly to (1) above: $\sigma(\phi)$ is exactly one of the first four elements of $E$ (with $R=2$).
\end{proof}

\begin{lemma}\label{lem: 56}
	Suppose $R\ge 2$ and $\sigma(\phi)=(0,0,0,0,0,y_5,y_6)$. Then $\Omega(\phi)=\cB_n(N(\phi))$.
\end{lemma}

\begin{proof}
	We induct on $R$, beginning with $R=2$. Let $\phi=\phi(s)\times\phi(t)$ with $s\in\{0,1\}^k$ and $t\in\{0,1\}^h$.
	If $\sigma(s)=\sigma(t)=6$, then $\Omega(100)\star\Omega(100)=\cB_{54}(36)$ by direct computation. %swap out (14,13) using Lemma~\ref{lem: LR tworow filling} or deal with skew shapes = [14,13], [14,13]^c directly; similarly with (14,13)'. That removed, Omega(100) then essentially contains B_{27}(16), so Prop~\ref{prop: B star B} means B_{54}(32) \subseteq \Omega(100)\star\Omega(100). Python now gives \lambda_1\in[33,36] => lambda in Om(100)\star Om(100).
	
	Next suppose $\sigma(s)=\sigma(t)=5$. Recall we determined $\Omega(s)$ and $\Omega(t)$ in Lemma~\ref{lem:1000} (i) (also listed in Table~\ref{table:summary}). By Proposition~\ref{prop: B star B} and Lemma~\ref{lem: omega star omega}, 
	\[ \cB_{n}(2a-2+2b-2)\subseteq\Omega(\phi)=\Omega(s)\star\Omega(t)\subseteq\cB_n(2a+2b), \]
	where $a=2\cdot 3^{k-1}$ and $b=2\cdot 3^{h-1}$. Moreover, $N(s)=2a$ and $N(t)=2b$.
	Since $\Omega(\phi)^\circ=\Omega(\phi)$, it suffices to consider $\lambda=(\lambda_1,\mu)$ with $2a+2b-3\le\lambda_1\le 2a+2b$ and $\mu\vdash 3a+3b-\lambda_1$.
	
	If $\lambda_1=2a+2b-i$ with $i\in\{2,3\}$, then $(\cP(a)\setminus\{(a-1,1)\}^\circ) \star \cP(b+i)=\cP(a+b+i)$ by Lemma~\ref{lem: 10.3} then implies $c^\mu_{\alpha,\beta}>0$ for some $\alpha\in \cP(a)\setminus\{(a-1,1)\}^\circ$ and $\beta\vdash b+i$. Hence $c^\lambda_{(2a,\alpha),(2b-i,\beta)}>0$ by Lemma~\ref{lem: iteratedLR}, where $(2a,\alpha)\in\Omega(s)$ and $(2b-i,\beta)\in\Omega(t)$.
	
	If $\lambda_1=2a+2b-1$, a similar argument shows that $\lambda\in\Omega(\phi)$ whenever $\mu\in\cB_{a+b+1}(a+b)=(\cP(a)\setminus\{(a-1,1)\}^\circ)\star \cB_{b+1}(b)$. But if $\mu=(a+b+1)$, we see that $c^\lambda_{(2a,a),(2b,b)}>0$ (and if $\mu=(1^{a+b+1})$ then $c^\lambda_{(2a,1^a),(2b,1^b)}>0$), so $\lambda\in\Omega(\phi)$ also. 
	
	If $\lambda_1=2a+2b$, then $\lambda\in\Omega(\phi)$ if and only if $\mu\in(\cP(a)\setminus\{(a-1,1)\}^\circ) \star(\cP(b)\setminus\{(b-1,1)\}^\circ)$. Clearly this contains $(\cP(a)\setminus\{(a-1,1)\}^\circ) \star\cB_b(b-2)=\cB_{a+b}(a+b-2)$ where the equality follows from Lemma~\ref{lem: 10.3}. For $\mu\in\cP(a+b)\setminus\cB_{a+b}(a+b-2)$, we see that either $c^\lambda_{(2a,a),(2b,b)}>0$ or $c^\lambda_{(2a,1^a),(2b,1^b)}>0$, and so $\lambda\in\Omega(\phi)$.
	
	Finally suppose $\sigma(s)=5$ and $\sigma(t)=6$. Note that if $\lambda\in\Omega(\phi)$, then $c^\lambda_{\mu,\nu}>0$ for some $\mu\in\Omega(s)$ and $\nu\in\Omega(100)$. By %the same argument as in Step 1 of the proof of Lemma~\ref{lem: LRb} (i.e.~using
	Lemma~\ref{lem: LR tworow filling} (b), %), 
	%for the $[\lambda\setminus\mu]\cong[\tworow{}{x}]^c$ case of the argument, see comment at the end of proof of Lemma~\ref{lem: LRc}
	if $\nu\in\{(14,13)\}^\circ$ then we may replace $\nu$ by some 
	\begin{equation}\label{eqn:Q}
		\omega\in\cQ:=\cB_{27}(17)\setminus\{(17,10),(17,1^{10})\}^\circ\sqcup\{(18,\gamma) \mid \gamma\in\Omega(\triv_{P_9}) \}^\circ,
	\end{equation}
	or $[\lambda\setminus\mu]\in\{[(14,13)],[(14,13)]^c\}$. In the latter case since $\mu\in\Omega(s)$ was determined in Lemma~\ref{lem:1000} we can directly see that $\lambda\in\Omega(s)\star\cQ$ via $c^\lambda_{\alpha,\beta}>0$ for some $\alpha\in\Omega(s)$ and $\beta\in\cQ$ obtained by interchanging boxes of $[\mu]$ and $[\lambda\setminus\mu]$. That $\Omega(\phi)=\cB_n(N(\phi))$ then follows from a similar argument to the above.
	
	If $R\ge 3$, we may assume $\phi=\pi\times\phi(s)$ where $s\in\{0,1\}^k$, $\sigma(s)\in\{5,6\}$, $\pi\in\Lin(P_{n-3^k})$ and $\Omega(\pi)=\cB_{n-3^k}(N(\pi))$ by the inductive hypothesis. %Similar arguments to the above then show that 
	That $\Omega(\phi)=\cB_n(N(\phi))$ follows similarly.
\end{proof}

\begin{proposition}\label{prop:qt}
	Suppose $R\ge 2$ and $\sigma(\phi)=(y_1,y_2,y_3,y_{30},y_{31},y_5,y_6)$. Then
	\[ \Omega(\phi)=\begin{cases}
		\cB_n(N(\phi)-1)\sqcup\{(N(\phi),\mu) \mid \mu\in\Omega(\triv_{P_{n-N(\phi)}}) \}^\circ & \text{if}\ \sigma(\phi)\in E,\\
		\cB_n(N(\phi)) & \text{otherwise}.
	\end{cases} \]
\end{proposition}

\begin{proof}
	Let $\phi=\pi\times\rho$ where $\pi\in\Lin(P_{n'})$, $\rho\in\Lin(P_{n-n'})$, $\sigma(\pi)=(y_1,y_2,y_3,y_{30},y_{31},0,0)$ and $\sigma(\rho)=(0,0,0,0,0,y_5,y_6)$. If $y_5+y_6\ge 2$, then by Lemmas~\ref{lem: 1233031} and~\ref{lem: 56}, and the same argument as in the proof of Lemma~\ref{lem: 1233031}, we find $\Omega(\phi)=\Omega(\pi)\star\cB_{n-n'}(N(\rho))=\cB_n(N(\phi))$. Clearly the statement also holds if $y_5=y_6=0$ by Lemma~\ref{lem: 1233031}. % (in particular, this covers the case if $\sigma(\phi)$ is one of the first six elements of $E$). 
	So it remains to consider $y_5+y_6=1$. 
	
	In particular, there are three forms of $\Omega(\pi)$ (given by Lemma~\ref{lem: 1233031}) and two forms of $\Omega(\rho)$ (letting $\rho=\phi(s)$, either $\sigma(s)=5$ or $\sigma(s)=6$) to consider. Namely, $\Omega(\pi)$ is
	\begin{small}
		\[ \text{(i)}\ \cP(3^k)\setminus\{(3^k-1,1)\}^\circ\ (k\ge 3),\text{ (ii) }\cB_{n'}(N(\pi)),\text{ or (iii)}\ \cB_{n'}(N(\pi)-1)\sqcup \{(N(\pi),\mu) \mid \mu\in\Omega(\triv_{P_{n'-N(\pi)}}) \}^\circ, \]
	\end{small}
	while (I) $\sigma(s)=5$ or (II) $\sigma(s)=6$. Moreover, if (I) $\sigma(s)=5$, then we have
	\[ \Omega(s)=\cB_{3^h}(2a)\setminus\{(2a,a-1,1),(2a,2,1^{a-2}),(2a-1,a+1),(2a-1,1^{a+1}) \}^\circ \]
	for some $h\ge 4$ and $a:=3^{h-1}$ from Lemma~\ref{lem:1000} (or Table~\ref{table:summary}), while if (II) $\sigma(s)=6$, then
	\[ \Omega(s)=\Omega(100)=\cB_{27}(17)\setminus\{(17,10),(17,1^{10}),(14,13)\}^\circ \sqcup\{ (18,\mu)\mid \mu\in\Omega(\triv_{P_9})\}^\circ. \]
	If (II) holds, we may apply the same replacement argument as in the proof of Lemma~\ref{lem: 56} to see that $\Omega(\phi)=\Omega(\pi)\star\Omega(s)=\Omega(\pi)\star\cQ$ where $\cQ:=\cB_{27}(17)\setminus\{(17,10),(17,1^{10})\}^\circ\sqcup\{(18,\gamma) \mid \gamma\in\Omega(\triv_{P_9}) \}^\circ$ (see \eqref{eqn:Q}).
	To complete the proof, we now consider cases (i)--(iii) in turn.
	
	\noindent\emph{Case (i): $\sigma(\pi)=(1,0,0,0,0,0,0)$ and $\Omega(\pi)=\cP(3^k)\setminus\{(3^k-1,1)\}^\circ$ for some $k\ge 3$.}
	\begin{itemize}
		\item[$\circ$] If $\sigma(s)=5$ (i.e.~to calculate $\Omega(\phi)$ we want to compute ``(i) $\star$ (I)"): in this case $\sigma(\phi)=(1,0,0,0,0,1,0)\in E$.
		By Proposition~\ref{prop: B star B} and Lemma~\ref{lem: omega star omega}, $$\cB_{3^k+3^h}(3^k-2+2a-2)\subseteq\Omega(\phi)=\Omega(\pi)\star\Omega(s)\subseteq\cB_{3^k+3^h}(3^k+2a).$$
		Since $\Omega(\phi)^\circ=\Omega(\phi)$, it remains to consider $\lambda=(\lambda_1,\mu)$ where $3^k+2a-3\le\lambda_1\le 3^k+2a$ and $\mu\vdash 3^k+3^h-\lambda_1$. 
		If $\lambda_1=3^k+2a-i$ where $i\in\{2,3\}$, or if $i=1$ and $\mu\in\cB_{a+1}(a)$, then $c^\lambda_{(3^k),\alpha}>0$ where $\alpha=(2a-i,\mu)\in\Omega(s)$, so $\lambda\in\Omega(\phi)$. 
		If $i=1$ and $\mu\in\{(a+1)\}^\circ$, then observe either $c^\lambda_{(3^k),(2a,a)}>0$ or $c^\lambda_{(3^k),(2a,1^a)}$, so $\lambda\in\Omega(\phi)$ also. 
		However, if $i=0$ then $\lambda\in\Omega(\phi)$ if and only if $\mu\in\cP(a)\setminus\{(a-1,1)\}^\circ=\Omega(\triv_{P_{3^{h-1}}})$.
		Hence $\Omega(\phi)=\cB_n(N(\phi)-1)\sqcup\{(N(\phi),\mu) \mid \mu\in\Omega(\triv_{P_{n-N(\phi)}}) \}^\circ$.
		
		\item[$\circ$] If $\sigma(s)=6$ (i.e.~``(i) $\star$ (II)"): in this case $\sigma(\phi)=(1,0,0,0,0,0,1)\in E$.		
		This follows from the same argument as in the above case, %(i) $\star$ (I) 
		setting $a=9$ and observing that $(18,9),(18,1^9)\in\cQ$.
	\end{itemize}

	\noindent\emph{Case (ii): $\sigma(\pi)\notin E$ and $\Omega(\pi)=\cB_{n'}(N(\pi))$.}
	\begin{itemize}
		\item[$\circ$] If $\sigma(s)=5$ (i.e.~``(ii) $\star$ (I)"): by a similar argument to the proof of Lemma~\ref{lem: 56}, we find $\cB_n(N(\pi)+2a-1)\subseteq\Omega(\phi)\subseteq\cB_n(N(\pi)+2a)$. Let $\lambda=(N(\pi)+2a,\mu)\vdash n$. If $N(\pi)<n'$, then $\lambda\in\Omega(\phi)$ if and only if $\mu\in\cP(n'-N(\pi))\star\cP(a)\setminus\{(a-1,1)\}^\circ$, but this equals $\cP(n-N(\pi)+a)$ by Lemma~\ref{lem: 10.3}, % since $a=3^{h-1}\ge 27$, 
		and so $\Omega(\phi)=\cB_n(N(\phi))$. Otherwise $N(\pi)=n'$, i.e.~$\sigma(\pi)=(y_1,0,0,0,0,0,0)$, in which case $\lambda\in\Omega(\phi)$ if and only if $\mu\in\Omega(\triv_{P_a})$. This is exactly $\sigma(\phi)=(R-1,0,0,0,0,1,0)\in E$, and so $\Omega(\phi)=\cB_n(N(\phi)-1)\sqcup\{(N(\phi),\mu) \mid \mu\in\Omega(\triv_{P_{n-N(\phi)}}) \}^\circ$ since $n-N(\phi)=3^h-N(s)=a$.
		
		\item[$\circ$] If $\sigma(s)=6$ (i.e.~``(ii) $\star$ (II)"): by a similar argument, %to the case of (ii) $\star$ (II),
		$\cB_n(N(\pi)+17)\subseteq\Omega(\phi)\subseteq\cB_n(N(\pi)+18)$, and now $\lambda=(N(\pi)+18,\mu)\vdash n$ belongs to $\Omega(\phi)$ if and only if $\mu\in\cP(n'-N(\pi))\star\Omega(00)$. Therefore $\Omega(\phi)=\cB_n(N(\phi))$ if $n'-N(\pi)\ge 2$, by Lemma~\ref{lem: P star 00}. If $n'-N(\pi)=1$, i.e.~$\sigma(\pi)=(y_1,1,0,0,0,0,0)$, then $\Omega(\phi)=\cB_n(N(\phi)-1)\sqcup\{(N(\phi),\mu) \mid \mu\in\Omega(\triv_{P_{n-N(\phi)}}) \}^\circ$. This is exactly $\sigma(\phi)=(R-2,1,0,0,0,0,1)\in E$. If $N(\pi)=n'$ then $\Omega(\phi)$ also takes this form, and $\sigma(\phi)=(R-1,0,0,0,0,0,1)\in E$.
	\end{itemize}
	
	\noindent\emph{Case (iii): $\sigma(\pi)=(y_1,y_2,y_3,y_{30},y_{31},0,0)\in E$ and $\Omega(\pi)=\cB_{n'}(N(\pi)-1)\sqcup \{(N(\pi),\mu) \mid \mu\in\Omega(\triv_{P_{n'-N(\pi)}}) \}^\circ$. Moreover, note that $n'-N(\pi)\in\{3,4,6,9,10\}\cup\{3^s \mid s\ge 3 \}$.}
	\begin{itemize}
		\item[$\circ$] If $\sigma(s)=5$ (i.e.~``(iii) $\star$ (I)"): by Proposition~\ref{prop: B star B}, $\cB_n(N(\pi)-1+2a-2)\subseteq\Omega(\phi)\subseteq\cB_n(N(\pi)+2a)$. Since $\Omega(\phi)^\circ=\Omega(\phi)$, it remains to consider $\lambda=(\lambda_1,\mu)\vdash n$ such that $N(\pi)+2a-2\le\lambda_1\le N(\pi)+2a$.
		
		If $\lambda_1=N(\pi)+2a-2$, first suppose $\mu\in\cB_{n-\lambda_1}(n-1-\lambda_1)$. Then $\mu\in\cP(n'-N(\pi)+1)\star\cB_{a+1}(a)$ by Proposition~\ref{prop: B star B}, so $c^\lambda_{(N(\pi)-1,\alpha),(2a-1,\beta)}=c^\mu_{\alpha,\beta}>0$ for some $\alpha\vdash n'-N(\pi)+1$ and $\beta\in\cB_{a+1}(a)$. Thus $\lambda\in\Omega(\phi)$ since $(N(\pi)-1,\alpha)\in\cB_{n'}(N(\pi)-1)\subseteq\Omega(\pi)$ and $(2a-1,\beta)\in\Omega(s)$. Otherwise $\mu=(|\mu|)$ or $\mu=(1^{|\mu|})$, in which case notice that $c^\lambda_{(N(\pi),n'-N(\pi)),(2a-2,a+2)}>0$ or $c^\lambda_{(N(\pi),1^{n'-N(\pi)}), (2a-2,1^{a+2})}>0$, and so $\lambda\in\Omega(\phi)$ also.
		
		If $\lambda_1=(N(\pi)-1)+2a$, then $\lambda\in\Omega(\phi)$ if and only if $\mu\in\cP(n'-N(\pi)+1)\star(\cP(a)\setminus\{(a-1,1)\}^\circ)$. But $n'-N(\pi)\in\{3,4,6,9,10\}\cup\{3^s \mid s\ge 3 \}$, so 
		\[ \cP(n'-N(\pi)+1)\star(\cP(a)\setminus\{(a-1,1)\}^\circ) \supseteq \Omega(\triv_{P_{n'-N(\pi)+1}})\star\Omega(\triv_{P_{3^{h-1}}}) = \Omega(\triv_{P_{|\mu|}})=\cP(|\mu|) \]
		by Theorem~\ref{thm:GL1A}, since $h\ge 4$ and so $(n'-N(\pi)+1) + 3^{h-1}$ is a 3-adic expansion of $|\mu|$. Hence $\lambda\in\Omega(\phi)$ whenever $\lambda_1=N(\pi)+2a-1$.
		
		Finally if $\lambda_1=N(\pi)+2a$, then $\lambda\in\Omega(\phi)$ if and only if $\mu\in\Omega(\triv_{P_{n'-N(\pi)}}) \star(\cP(a)\setminus\{(a-1,1)\}^\circ)$, but this equals $\Omega(\triv_{P_{|\mu|}})=\cP(|\mu|)$ similarly since $(n'-N(\pi))+3^{h-1}$ is a 3-adic expansion of $|\mu|$.
		Thus $\Omega(\phi)=\cB_n(N(\phi))$.
		
		\item[$\circ$] If $\sigma(s)=6$ (i.e.~``(iii) $\star$ (II)"): this follows from the same argument as in the preceding case, %of (iii) $\star$ (I) 
		setting $a=9$ and observing that $(16,11),(16,1^{11})\in\cQ$, and that $(n'-N(\pi)+1) + 3^{h-1}$ and $(n'-N(\pi))+3^{h-1}$ are still 3-adic expansions when $h=3$ for the given values of $n'-N(\pi)$. Thus $\Omega(\phi)=\cB_n(N(\phi))$.
	\end{itemize}
\end{proof}

\begin{theorem}\label{thm:pre-A'}
	Suppose $l\ge 1$. If $R=1$ then further suppose $\phi=\phi(s)$ with $\sigma(s)\ne 6$. Then
	\begin{small}
		\[ \Omega(\phi\times\psi) = \begin{cases}
			\cB_{n+l}(N-1)\sqcup\{(N,\mu) \mid \mu\in\Omega(\triv_{P_{n+l-N}}) \}^\circ & \text{if}\ \sigma(\phi)\in E\ \text{and}\ \psi=\triv_{P_l},\\
			& \quad\sigma(\phi)\in F\ \text{and}\ \psi\in\Psi_2,\\
			& \quad\sigma(\phi)=(R,0,0,0,0,0,0)\ \text{and}\ \psi\in\Psi_1,\\
			& \quad\sigma(\phi)=(R-1,0,0,0,1,0,0)\ \text{and}\\
			& \qquad\qquad \psi=\phi(1,0)\cdot\triv_{P_{l-9}},\ l\ge 9,\ \text{or}\\
			& \quad\sigma(\phi)=(R-1,1,0,0,0,0,0)\ \text{and}\\
			& \qquad\qquad \psi=\phi(1,0)\cdot\triv_{P_{l-9}},\ l\ge 9;\\
			\cB_{n+l}(N) & \text{otherwise}.
		\end{cases} \]
	\end{small}
\end{theorem}

\begin{proof}
	Let $N:=N(\phi\times\psi)=N(\phi)+N(\psi)$, and $\Omega:=\Omega(\phi\times\psi)=\Omega(\phi)\star\Omega(\psi)$. Our goal is to compute $\Omega$ where the form of $\Omega(\phi)$ is given by Proposition~\ref{prop:qt} if $R\ge 2$ or Table~\ref{table:summary} if $R=1$, and $\Omega(\psi)$ is given in Table~\ref{table:m<27}. %(see rows in bold where $s_i\ne(1,1)$ for all $i$ since $\psi$ is quasi-trivial).
	First suppose $R\ge 2$.
	
	\noindent\emph{Case 1: $\Omega(\phi)=\cB_n(N(\phi))$, i.e.~$\sigma(\phi)\notin E$.}
	\begin{itemize}
		\item[$\circ$] If $\Omega(\psi)=\cB_l(m(\psi))$: then $\Omega=\cB_{n+l}(N)$ by Proposition~\ref{prop: B star B} since $N(\phi)>n/2$ and $N(\psi)=m(\psi)>l/2$.
		
		\item[$\circ$] If $\Omega(\psi)=\cB_l(l-|\nu|)\setminus\{(l,\nu)\}^\circ$ for some $\nu\in\{(2,1),(2,2),(3,2,1) \}$: then $\Omega$ is of the form in \eqref{eqn: not ball}, so we may apply arguments (1)--(3) from the proof of Lemma~\ref{lem: 1233031} to see that if $n-N(\phi)\ge 2$, then $\Omega=\cB_{n+l}(N)$.
		
		If $n-N(\phi)=1$, i.e.~$\sigma(\phi)=(R-1,1,0,0,0,0,0)$, then since $|\nu|\in\{3,4,6\}$, we observe that $\cP(1)\star\Omega(\triv_{P_{|\nu|}})=\cP(|\nu|)$ when $|\nu|=4$ or 6 but not when $|\nu|=3$. That is, $\Omega=\cB_{n+l}(N-1)\sqcup\{(N,\mu) \mid \mu\in\Omega(\triv_{P_{n+l-N}}) \}^\circ$ if $\psi=\phi(1,0)\cdot\triv_{P_{l-9}}$ for some $l\ge 11$, and $\Omega=\cB_{n+l}(N)$ otherwise.
		
		If $N(\phi)=n$, i.e.~$\sigma(\phi)=(R,0,0,0,0,0,0)$, then (1) gives us $\Omega=\cB_{n+l}(N-1)\sqcup\{(N,\mu) \mid \mu\in\Omega(\triv_{P_{n+l-N}}) \}^\circ$.
		
		\item[$\circ$] What remains in Case 1 is to consider $\Omega(\psi)$ of one of the two above forms but excluding some small partitions $\varepsilon$. Since $\varepsilon_1, l(\varepsilon)<m(\psi)$ or $\varepsilon_1, l(\varepsilon)<l-|\nu|$ respectively, we deduce from an analysis similar to that in Lemma~\ref{lem: 10.3} (since $m(\psi)<l/2$) that $\Omega$ has the same respective form. This same observation holds in each of the remaining cases, so we will not repeat this below.
	\end{itemize}
	
	\noindent\emph{Case 2: $\Omega(\phi)=\cB_n(N(\phi)-1)\sqcup\{(N(\phi),\mu) \mid \mu\in\Omega(\triv_{P_{n-N(\phi)}}) \}^\circ$, i.e.~$\sigma(\phi)\in E$.}
	\begin{itemize}
		\item[$\circ$] If $\Omega(\psi)=\cB_l(m(\psi))$, then $\Omega$ is of the form in \eqref{eqn: not ball}, so we may apply arguments (1)--(3) from the proof of Lemma~\ref{lem: 1233031}: if $l-m(\psi)\ge 2$, then $\Omega=\cB_{n+l}(N)$. 
		
		If $l-m(\psi)=1$, i.e.~$\psi\in\Psi_2$ with $l\notin\{6,9,12\}$, observe that $n-N(\phi)\in\{3,4,6,9,10\}\cup\{3^s\mid s\ge 3\}$. Moreover, $\Omega(\triv_{P_{n-N(\phi)}})\star\cP(1)=\cP(n-N(\phi)+1)$ unless $n-N(\phi)\in\{3,9\}$, i.e.~$\sigma(\phi)\in F$. Hence $\Omega=\cB_{n+l}(N)$ if $\sigma(\phi)\notin F$, but $\Omega=\cB_{n+l}(N-1)\sqcup\{(N,\mu) \mid \mu\in\Omega(\triv_{P_{n+l-N}}) \}^\circ$ if $\sigma(\phi)\in F$.
		
		If $m(\psi)=l$, i.e.~$\psi=\triv_{P_l}$ with $l\notin\{3,4,6,9,10\}$, then (1) gives us $\Omega=\cB_{n+l}(N-1)\sqcup\{(N,\mu) \mid \mu\in\Omega(\triv_{P_{n+l-N}}) \}^\circ$.
		
		\item[$\circ$] If $\Omega(\psi)=\cB_l(l-|\nu|)\setminus\{(l,\nu)\}^\circ$ for $\nu\in\{(2,1),(2,2),(3,2,1) \}$, then similarly to Case 1 we find that $\cB_{n+l}(N-1)\subseteq\Omega$. Noting that $N(\psi)=l-|\nu|$ and $N=N(\phi)+N(\psi)$, if $\lambda=(N,\mu)\vdash n+l$ then $\lambda\in\Omega$ if and only if $\mu\in\Omega(\triv_{n-N(\phi)})\star\Omega(\triv_{|\nu|})$. By Theorem~\ref{thm:GL1A}, this is $\cP(n-N(\phi)+|\nu|)$ unless $|\nu|=n-N(\phi)=3$ (note $\Omega(\triv_{P_3})\star\Omega(\triv_{P_6})=\cP(9)\ne\Omega(\triv_{P_9})$; similarly $\Omega(\triv_{P_{6}})^{\star 2}=\cP(12)\ne\Omega(\triv_{P_{12}})$), in other words, unless $\psi=\phi(1,0)\cdot\triv_{P_{l-9}}$ with $l\ge 11$, and $\sigma(\phi)=(R-1,0,0,0,1,0,0)$, and in this case we find $\Omega=\cB_{n+l}(N-1)\sqcup\{(N,\mu) \mid \mu\in\Omega(\triv_{P_{n+l-N}}) \}^\circ$.

		%\item[$\circ$] A similar analysis shows that $\Omega$ is as claimed if $\Omega(\psi)$ has one of the two above forms but excluding some small partitions. %What remains in Case 2 is to consider $\Omega(\psi)$ of one of the two above forms but excluding some small partitions $\varepsilon$. Since $\varepsilon_1, l(\varepsilon)<m(\psi)$ or $\varepsilon_1, l(\varepsilon)<l-|\nu|$ respectively, we deduce from a similar analysis that $\Omega$ has the same respective form.
	\end{itemize}

	Now suppose $R=1$. If $\sigma(s)=2$ or $\sigma(s)\in\{3,30,31\}$ then the same argument as in Case 1 or 2 respectively holds, so it remains to consider $\sigma(s)\in\{1,5\}$.
	
	\noindent\emph{Case 3: $\sigma(s)=1$, $\Omega(s)=\cP(3^k)\setminus\{3^k-1,1\}^\circ$.}
	\begin{itemize}
		\item[$\circ$] Observe that $\Omega(\phi)\star\cB_l(m(\psi))=\cB_{n+l}(N)$ by Lemma~\ref{lem: 10.3}, while $\Omega(\phi)\star\Omega(\triv_{P_l})=\Omega(\triv_{3^k+l})$ since $3^k+l$ is a 3-adic expansion. By Theorem~\ref{thm:GL1A}, this equals $\cP(n)$.
		
		\item[$\circ$] On the other hand, observe $\Omega(\phi)\star\cB_l(l-|\nu|)\setminus\{(l,\nu)\}^\circ \supseteq\cB_{n+l}(N-1)$ when $\nu\in\{(2,1),(2,2),(3,2,1)\}$. Moreover, if $\lambda=(N,\mu)\vdash n$ then $\lambda\in\Omega$ if and only if $\mu\ne\nu$, i.e.~$\mu\in\Omega(\triv_{P_{|\nu|}})$. Since $n+l-N=l-N(\psi)=|\nu|$, this means that when $\sigma(s)=1$ then $\Omega=\cB_{n+l}(N-1)\sqcup\{(N,\mu) \mid \mu\in\Omega(\triv_{P_{n+l-N}}) \}^\circ$ exactly when $\psi\in\Psi_1$, and $\Omega=\cB_{n+l}(N)$ otherwise. 
		%(When $\Omega(\psi)$ has one of the aforementioned forms but excludes small exceptions, a similar analysis holds.)
	\end{itemize}
	
	\noindent\emph{Case 4: $\sigma(s)=5$, $\Omega(s)=\cB_{3^k}(2a)\setminus\{(2a,a-1,1),(2a,2,1^{a-2}),(2a-1,a+1),(2a-1,1^{a+1}) \}^\circ$ where $a=3^{k-1}$ and $k\ge 4$.}
	\begin{itemize}
		\item[$\circ$] First, if $\Omega(\psi)=\Omega(\triv_{P_l})$, then $\Omega=\cB_{n+l}(N-1)\sqcup\{(N,\mu) \mid \mu\in\Omega(\triv_{P_{n+l-N}}) \}^\circ$ by similar arguments to cases ``(i) $\star$ (I)" and ``(ii) $\star$ (I)" in the proof of Proposition~\ref{prop:qt}; this is exactly $\sigma(\phi)=(0,0,0,0,0,1,0)\in E$ and $\psi=\triv_{P_l}$.
		
		\item[$\circ$] Next, suppose $\Omega(\psi)=\cB_l(m(\psi))$. As in case ``(ii) $\star$ (I)" in the proof of Proposition~\ref{prop:qt}: if $m(\psi)<l$ then $\Omega=\cB_{n+l}(N)$, but if $m(\psi)=l$ (i.e.~$\psi=\triv_{P_l}$, $l\notin\{3,4,6,9,10\}$) then %this falls under the case of 
		by the preceding paragraph $\Omega=\cB_{n+l}(N-1)\sqcup\{(N,\mu) \mid \mu\in\Omega(\triv_{P_{n+l-N}}) \}^\circ$.
		
		\item[$\circ$] Finally, suppose $\Omega(\psi)=\cB_l(l-|\nu|)\setminus\{(l,\nu)\}^\circ$. 
		By Proposition~\ref{prop: B star B}, $\cB_{n+l}\big((2a-2)+(l-|\nu|)-1\big)\subseteq\Omega\subseteq\cB_{n+l}(2a+l-|\nu|)$. It suffices to consider $\lambda=(\lambda_1,\mu)\vdash n+l$ with $\lambda_1=2a+l-|\nu|-i$ for $i\in\{0,1,2\}$.
		
		If $i=2$, there exists $\beta\in\Omega(\triv_{P_{|\nu|}})$ such that $\beta\subset\mu$, then choose any $\alpha\in\cLR([\mu\setminus\beta])$. Then $c^\mu_{\alpha,\beta}=c^\mu_{\beta,\alpha}>0$, whence $c^\lambda_{(2a-2,\alpha),(t,\beta)}>0$ by Lemma~\ref{lem: iteratedLR}. Moreover, $(2a-2,\alpha)\in\Omega(s)$ and $(t,\beta)\in\Omega(\psi)$, so $\lambda\in\Omega$.
		
		If $i=1$, then $\lambda_1=(2a)+(l-|\nu|-1)$ so by a similar  %LR-coefficient 
		argument, $\lambda\in\Omega$ if and only if $\mu\in(\cP(a)\setminus\{(a-1,1)\}^\circ)\star\cP(|\nu|+1)$. But this is always $\cP(a+|\nu|+1)$ by Theorem~\ref{thm:GL1A}, so $\lambda\in\Omega$.
		
		If $i=0$, then $\lambda\in\Omega$ if and only if $\mu\in (\cP(a)\setminus\{(a-1,1)\}^\circ)\star\Omega(\triv_{P_{|\nu|}})$. But this equals $\Omega(\triv_{P_{3^{k-1}}+|\nu|})=\cP(3^{k-1}+|\nu|)$ by Theorem~\ref{thm:GL1A} and Lemma~\ref{lem: omega star omega} since $3^{k-1}+|\nu|$ is a 3-adic expansion. Hence $\lambda\in\Omega$.
		%(When $\Omega(\psi)$ has one of the aforementioned forms but excludes small exceptions, a similar analysis holds.)
	\end{itemize}
\end{proof}

\begin{proof}[Proof of Theorem~\ref{thm:a}]
	If $l=0$, then this follows from the results in Section~\ref{sec:prime-power} if $R=1$ and Proposition~\ref{prop:qt} if $R\ge 2$. If $l\ge 1$, then this follows from Theorem~\ref{thm:pre-A'}, unless $R=1$ and $\sigma(\phi)=6$ in which case $n<54$ and $\Omega(1,0,0)\star\Omega(\psi)$ may be computed directly.
\end{proof}

\medskip

%------------------------------------------------------------------------
\subsection{Proof of Theorem~\ref{thm:b-general}}\label{sec:b-full}

The main aim of this section is to prove Theorem~\ref{thm:b-general}, and then deduce that $\cB_n(\frac{2n}{9})\subseteq\Omega_n$ in Theorem~\ref{thm:c}. We keep the notation $\phi$ and $\psi$ from the previous section for denoting quasi-trivial linear characters, and introduce notation %$\theta$ and $\vartheta$ 
to denote non-quasi-trivial components as described below.

\begin{notation}\label{not:nqt}
	Throughout Section~\ref{sec:b-full}, we fix the following notation: let $n,l$ and $\phi,\psi$ be as in Notation~\ref{not:qt}. Additionally, let $u\in\N$ be divisible by 27 and $v\in\{9,10,\dotsc,26\}$ be such that $n+u$ and $l+v$ are each 3-adic expansions (i.e.~if $n=\sum_{i\ge 0}3^i\cdot a_i$ and $u=\sum_{i\ge 0}3^i\cdot b_i$ with $a_i,b_i\in\{0,1,2\}$, then $a_i+b_i\le 2$ for all $i$; similarly for $l+v$).
	
	Let $\theta\in\Lin(P_u)$ and $\vartheta\in\Lin(P_v)$. Further suppose that $\theta$ corresponds to $\{t_1,\cdots,t_T\}$ with $T\ge 1$ (in the sense of Notation~\ref{not:correspond}) and $\sigma(t_i)\in\{10,11,21,22,7\}$ for all $i\in[T]$. Let $\sigma(\theta)=(y_{10},y_{11},y_{21},y_{22},y_7)$ where $y_i:=\#\{j\in[R] \mid \sigma(t_j)=i \}$. In addition, suppose $\vartheta=\phi(1,1)\times\varphi$ for some $\varphi\in\Lin(P_{v-9})$ containing no quasi-trivial components. That is, $\theta\times\vartheta$ contains no quasi-trivial components.
\end{notation}

%We recall that the 3-Sylow-types $\sigma(t_i)$ and various useful properties of $\Omega(t_i)$ and $m(t_i)$ were summarised in Table~\ref{table:summary}. By assumption, 

%By Lemma~\ref{lem: omega star omega}, $\Omega(\theta\times\vartheta)=\Omega(\theta)\star\Omega(\vartheta)$. We first find all possible forms of this set, i.e.~we describe $\Omega(\Theta)$ for linear characters $\Theta$ with no quasi-trivial components.

%More generally, 
Any linear character $\Theta$ of a Sylow subgroup of $S_w$ may be split up into its quasi-trivial and non-quasi-trivial components. We call these $\phi\times\psi$ and $\theta\times\vartheta$ respectively, split further as a product of linear characters of $P_{w-w'}$ and $P_{w'}$, where $w'\in\{0,1,\dotsc,26\}$ is given by $w'\equiv w$ (mod 27). We prove Theorem~\ref{thm:b-general} by investigating all possible forms of $\Omega(\phi\times\psi\times\theta\times\vartheta)$, in order to compute $m(\Theta)$ for all linear $\Theta$. More precisely, let $\Theta\in\Lin(P_w)$ where $w=n+u+l+v$ as in Notation~\ref{not:nqt}, or a subsum thereof, such as $w=n+u$, $w=u+l+v$, $w=l$, etc. We describe all possible forms of $\Omega(\Theta)$ in order to deduce the value of $m(\Theta)$. 

The structure of this section is as follows:
\begin{itemize}
	\item We first describe $\Omega(\theta)$ for all $\theta$ as described in Notation~\ref{not:nqt} (Lemma~\ref{lem: type 11}).
	\item Next, we look at $\Omega(\phi\times\theta)$ (Lemma~\ref{lem: RT=1}). In other words, this combines all quasi-trivial and non-quasi-trivial $\sigma$ types when $27\mid w$ (i.e.~this is the case $w=n+u$).
	\item We then describe $\Omega(\theta\times\vartheta)$ (Lemma~\ref{lem: all nqt}). This covers linear characters $\Theta$ where every component of $\Theta$ is non-quasi-trivial (i.e.~$w=u+v$).
	\item We then describe $\Omega(\Theta)$ for the remaining subsums%, which follow similarly
	: $w=n+v$, $w=u+l$, $w=n+u+l$, $w=n+u+v$, $w=n+l+v$, $w=u+l+v$ and $w=n+l+u+v$ (Lemma~\ref{lem: rest}).
	\item Thus, we are finally able to prove Theorem~\ref{thm:b-general}, and deduce that Theorem~\ref{thm:c} holds.
\end{itemize}

To describe $\Omega(\theta)$, we first look at when $\sigma(\theta)=(y_{10},y_{11},y_{21},y_{22},y_7)$ has $y_i=0$ for some $i$.
\begin{lemma}\label{lem:2122}
	Suppose $\sigma(\theta)=(0,0,y_{21},y_{22},0)$. Then $\cB_u(N(\theta))\subseteq\Omega(\theta)$, and $\Omega(\theta)\setminus\cB_u(N(\theta))$ contains no thin partitions.
\end{lemma}
\begin{proof}
	The first statement follows from Proposition~\ref{prop: B star B}, and the second from a similar argument to Step 2 in the proof of Lemma~\ref{lem:111} (Littlewood--Richardson rule and pigeonhole principle).
\end{proof}

\begin{lemma}\label{lem:102122}
	Suppose $T\ge 2$ and $\sigma(\theta)=(y_{10},0,y_{21},y_{22},y_7)$. Then $\cB_u(N(\theta))\subseteq\Omega(\theta)$, and $\Omega(\theta)\setminus\cB_u(N(\theta))$ contains no thin partitions.
\end{lemma}
\begin{proof}
	First suppose $y_7=0$. We are done by Lemma~\ref{lem:2122} if $y_{10}=0$ also, so we may assume $\sigma(t_1)=10$ and let $\Omega(t_1)=\cQ\sqcup\mathcal{N}$ where $\cQ=\cB_{3^k}(m+2)\setminus\{\tworow{3^k}{m+1},\hook{3^k}{m+1}\}^\circ$ and $\mathcal{N}$ contains no thin partitions, $k\ge 3$ and $m:=m(t_1)=\tfrac{3^k+3^{k-F(t_1)}}{2}-2\ge\tfrac{3^k+23}{2}$. 
	%We first investigate $\Omega(t_1)\star\Omega(\theta')$ for $\sigma(\theta')=(0,0,y_{21},y_{22},0)$ and $\sigma(\theta')=(1,0,0,0,0)$ in Steps 1 and 2 below. Then we are able to deduce $\Omega(\theta)$ for $\sigma(\theta)=(y_{10},0,y_{21},y_{22},0)$ in Step 3.
	\begin{itemize}
		\item[$\circ$] \emph{Step 1.} We consider $\Omega:=\cQ\star\cB_n(N)$ where $N>n/2$.
		
		By Proposition~\ref{prop: B star B}, $\cB_{n+3^k}(m+N)\subseteq\Omega\subseteq\cB_{n+3^k}(m+2+N)$. Let $\lambda=(\lambda_1,\mu)\vdash n+3^k$. 
		
		If $\lambda_1=m+2+N$, then $c^\mu_{\alpha,\beta}>0$ for some $\alpha\in\cP(3^k-(m+2))$ and $\beta\in\cP(n-N)$, as $\cP(3^k-(m+2))\star\cP(n-N)=\cP(|\mu|)$ by Proposition~\ref{prop: B star B}. Hence $c^\lambda_{(m+2,\alpha),(N,\beta)}=c^\mu_{\alpha,\beta}>0$ by Lemma~\ref{lem: iteratedLR}, and so $\lambda\in\Omega$ since $(m+2,\alpha)\in\cQ$ and $(N,\beta)\in\cB_n(N)$.
		
		If $\lambda_1=m+1+N$, then by the same argument if $\mu\in\cB_{3^k-(m+1)}(3^k-(m+1)-1)\star\cP(n-N)=\cB_{|\mu|}(|\mu|-1)$ then $\lambda\in\Omega$. Otherwise, $\lambda\in\{\tworow{n+3^k}{m+1+N},\hook{n+3^k}{m+1+N} \}$. In this case, notice that $c^\lambda_{\tworow{3^k}{m+2},\tworow{n}{N}}>0$ or $c^\lambda_{\hook{3^k}{m+2},\hook{n}{N}}>0$ respectively, and so $\lambda\in\Omega$.
		Thus $\Omega=\cB_{n+3^k}(m+2+N)$.
		
		\item[$\circ$] \emph{Step 2.} Now let $\Omega:=\cQ\star\cQ'$, where $\cQ':=\cB_{3^{k'}}(m'+2)\setminus\{\tworow{3^{k'}}{m'+1},\hook{3^{k'}}{m'+1}\}^\circ$ for some $k'\in\N_{\ge 3}$ and $m'\ge\tfrac{3^{k'}+23}{2}$.
		Clearly $\cB_{3^k+3^{k'}}(m+m')\subseteq\Omega\subseteq\cB_{3^k+3^{k'}}(m+m'+4)$. By an analogous argument to Step 1, we find that $\lambda\in\Omega$ for all $\lambda\vdash 3^k+3^{k'}$ with $m+m'+1\le\lambda_1\le m+m'+4$. Hence $\Omega=\cB_{3^k+3^{k'}}(m+m'+4)$.
		
		\item[$\circ$] \emph{Step 3.} Using Steps 1 and 2, we deduce from Lemma~\ref{lem:2122} that if $T\ge 2$, $y_{10}\ge 1$ and $y_7=0$, then $\Omega(\theta)\supseteq\cB_{u}(N(\theta))$ and contains no other thin partitions.
	\end{itemize}

	Now suppose $y_7>0$. Assume $\sigma(t_1)=7$ and recall 
	\[ \Omega(t_1)\supseteq\cB_{3^k}(m+5)\setminus\{\tworow{3^k}{m+4},\hook{3^k}{m+4},\tworow{3^k}{m+1}\}^\circ\]
	and $\Omega(t_1)$ contains no other thin partitions. Moreover, $k\ge 4$ and set $m:=m(t_1)=\tfrac{3^k-1}{2}$.
	By a similar argument to the proof of Lemma~\ref{lem: 56} (in particular using Lemma~\ref{lem: LR tworow filling} (b) to replace $\{\tworow{3^k}{m+1}\}^\circ$), we have that 
	$$\Omega(t_1)\star\Omega(\theta')=(\Omega(t_1)\sqcup\{\tworow{3^k}{m+1}\}^\circ)\star\Omega(\theta')$$ 
	for any linear character $\theta'$ with $\sigma(\theta')=(y_{10},0,y_{21},y_{22},0)$ such that $y_{10}+y_{21}+y_{22}>0$. Then, by using Steps 1--3 above but where we replace $\cQ$ with $\cB_{3^k}(m+5)\setminus\{\tworow{3^k}{m+4},\hook{3^k}{m+4}\}^\circ$, we find that $\Omega(\theta)\supseteq\cB_u(N(\theta))$ and contains no other thin partitions.
\end{proof}

\begin{lemma}\label{lem: type 11}
	Suppose $T\ge 2$ and $\sigma(\theta)=(y_{10},y_{11},y_{21},y_{22},y_7)$. Then $$\Omega(\theta)\supseteq\begin{cases}
		\cB_u(N(\theta))\setminus\{(\tfrac{u}{2},\tfrac{u}{2}) \}^\circ & \text{if}\ T=2\ \text{and}\ y_{11}=y_{21}=1,\\
		\cB_u(N(\theta)) & \text{otherwise},
	\end{cases}$$
	and $\Omega(\theta)$ contains no other thin partitions.
\end{lemma}

\begin{proof}
	We first investigate $\Omega(t_1)\star\Omega(\theta')$ for $\sigma(t_1)=11$ and various forms of $\Omega(\theta')$. Recall $\Omega(t_1)$ contains $\cB_{3^k}(m+2)\setminus\{\tworow{3^k}{m+1},\hook{3^k}{m+1}\}^\circ$ and no other thin partitions, where $k\ge 3$ and $m=m(t_1)=\tfrac{3^k-1}{2}$. Let $\cQ:=\Omega(t_1)\sqcup\{\tworow{3^k}{m+1},\hook{3^k}{m+1}\}^\circ$.
	
	\noindent\emph{Step 1.} Let $N>n/2$. We claim that 
	$$\Omega(t_1)\star\cB_n(N) = \begin{cases}
		\big( \cQ\star\cB_n(N) \big) \setminus \{(\tfrac{n+3^k}{2},\tfrac{n+3^k}{2}) \}^\circ & \text{if}\ n\ \text{is odd and}\ N < \tfrac{n+3}{2},\\
		\cQ\star\cB_n(N) & \text{otherwise}.\\
	\end{cases}$$
	To show this claim, first let $\lambda\in \big(\cQ\star\cB_n(N)\big) \setminus \{(\tfrac{n+3^k}{2},\tfrac{n+3^k}{2}) \}^\circ$. Then there exists $\alpha\in\cQ$ and $\beta\in\cB_n(N)$ such that $c^\lambda_{\alpha\beta}=c^\lambda_{\beta\alpha}>0$. If $\alpha\notin\{\tworow{3^k}{m+1},\hook{3^k}{m+1}\}^\circ$, then $\lambda\in\Omega(t_1)\star\cB_n(N)$ as desired. Since $c^\lambda_{\alpha\beta}=c^{\lambda'}_{\alpha'\beta'}$, we may assume $\alpha\in\{\tworow{3^k}{m+1},\hook{3^k}{m+1}\}$.
	
	If $\alpha=\tworow{3^k}{m+1}$, by Lemma~\ref{lem: LR tworow filling} (b) we may either replace $\alpha$ by some $\gamma\in\Omega(t_1)$, or $[\lambda\setminus\beta]\in\{[(m+1,m)], [(m+1,m)]^c \}$. If $[\lambda\setminus\beta]\cong [(m+1,m)]$, then as in Step 1 in the proof of Lemma~\ref{lem: LRb}, we choose $\mathsf{b}$ to be any removable box of $\beta$ and $\mathsf{c}$ to be the top left box of $[\lambda\setminus\beta]$. Then we set $[\tilde{\beta}]=([\beta]\setminus\mathsf{b})\cup\mathsf{c}$ and observe that $\cLR([\lambda\setminus\tilde{\beta}])\cap\Omega(t_1)\ne\emptyset$. Moreover, $\mathsf{c}$ cannot be in position $(1,N+1)$ or $(N+1,1)$ as $N>n/2$, so $\tilde{\beta}\in\cB_n(N)$. Hence $\lambda\in\Omega(t_1)\star\cB_n(N)$.
	
	If $[\lambda\setminus\beta]\cong[(m+1,m)]^c$, we instead take $\mathsf{c}$ to be the leftmost box in the top row of $[\lambda\setminus\beta]$ and $\mathsf{b}$ to be any removable box of $\beta$ not immediately to the left of $\mathsf{c}$. Such $\mathsf{b}$ always exists since $\lambda\ne(\tfrac{n+3^k}{2},\tfrac{n+3^k}{2})$, so $[\lambda\setminus\beta]\cong[(m+1,m)]^c$ implies that $\lambda$ is not a two-row partition. Then we define $\tilde{\beta}$ similarly and deduce that $\lambda\in\Omega(t_1)\star\cB_n(N)$. 
	
	Next, if $\alpha=\hook{3^k}{m+1}$, then a similar argument follows where we use Lemma~\ref{lem: LR hook filling} instead of Lemma~\ref{lem: LR tworow filling}. In the case $[\lambda\setminus\beta]\cong[(m+1,1^m)]^c$, we may choose $\mathsf{c}$ to be either the foot or hand of the skew hook $[\lambda\setminus\beta]$, and $\mathsf{b}$ to be the north-west neighbour of the corner of $[\lambda\setminus\beta]$.
	
	Note that $\tworow{3^k}{m+2}$ is the only two-row partition in $\Omega(t_1)$. Thus if ($n$ is odd and) $\lambda=(\tfrac{n+3^k}{2},\tfrac{n+3^k}{2})$, then $c^\lambda_{\alpha\beta}>0$ with $\alpha\in\Omega(t_1)$ implies that $\alpha=\tworow{3^k}{m+2}$ necessarily. This in turn implies that $\beta=(\tfrac{n+3}{2},\tfrac{n-3}{2})$, which belongs to $\cB_n(N)$ if and only if $N\ge\tfrac{n+3}{2}$. The case of $\lambda=(\tfrac{n+3^k}{2},\tfrac{n+3^k}{2})'$ follows similarly.
	
	\smallskip
	
	\noindent\emph{Step 2.} Next, we consider $\Omega(t_1)\star\Omega(t_2)$ where $\sigma(t_2)=10$.
	
	Recall $\Omega(t_2)$ contains $\cB_{3^{k'}}(m'+2)\setminus\{\tworow{3^{k'}}{m'+1},\hook{3^{k'}}{m'+1} \}^\circ$ and no other thin partitions, for some $k'\ge 3$ and $m'=m(t_2)\ge\tfrac{3^{k'}+23}{2}$. First we show that $\Omega(t_1)\star\Omega(t_2)=\cQ\star\Omega(t_2)$. 
	We follow the strategy of Step 1. 
	If $\lambda\in \big(\cQ\star\Omega(t_2)\big)\setminus\{ (\tfrac{3^k+3^{k'}}{2},\tfrac{3^k+3^{k'}}{2}) \}^\circ$, we define $\alpha,\beta$ accordingly. 
	In the cases where $[\lambda\setminus\beta]\in\{[(m+1,m)], [(m+1,m)]^c \}$, we define $\mathsf{b}$, $\mathsf{c}$ and $\tilde{\beta}$ via $[\tilde{\beta}]=([\beta]\setminus\mathsf{b})\cup\mathsf{c}$ as before. However, we must now also ensure that $\tilde{\beta}\notin\{\tworow{3^{k'}}{m'+1},\hook{3^{k'}}{m'+1} \}^\circ$. This is clear by inspection from the choices of $\mathsf{b}$ and $\mathsf{c}$.
	If $\lambda=(\tfrac{3^k+3^{k'}}{2},\tfrac{3^k+3^{k'}}{2})$, observe that $\tworow{3^k}{m+2}\in\Omega(t_1)$ and $\tworow{3^{k'}}{\tfrac{3^{k'}+3}{2}}\in\Omega(t_2)$ as $m'+1>\tfrac{3^{k'}+3}{2}$.
	Thus $\Omega(t_1)\star\Omega(t_2)=\cQ\star\Omega(t_2)$.
	
	Iterating, $\Omega(t_1)\star\Omega(t_2)=\cQ\star\Omega(t_2)=\cQ\star\cQ'$ where 
	$\cQ':=\Omega(t_2)\sqcup \{\tworow{3^{k'}}{m'+1},\hook{3^{k'}}{m'+1} \}^\circ$.
	Finally, observe that $\cQ\star\cQ'$ contains $\cB_{3^k+3^{k'}}(m+2+m'+2)$ and no other thin partitions, as in the proof of Lemma~\ref{lem:2122} %by Proposition~\ref{prop: B star B} and the pigeonhole principle.
	
	\noindent\emph{Step 3.} Next, we consider $\Omega(t_1)\star\Omega(t_2)$ where $\sigma(t_2)=7$.
	
	Recall $\Omega(t_2)$ contains $\cB_{3^{k'}}(m'+5)\setminus\{\tworow{3^{k'}}{m'+4},\hook{3^{k'}}{m'+4},\tworow{3^{k'}}{m'+1} \}^\circ$ and no other thin partitions, where $k'\ge 3$ and $m'=\tfrac{3^{k'}-1}{2}$. By a similar argument to the proof of Lemma~\ref{lem: 56}, $\Omega(t_1)\star\Omega(t_2)=\Omega(t_1)\star \big(\Omega(t_2)\sqcup\{\tworow{3^{k'}}{m'+1} \}^\circ\big)$. Then applying the argument in Step 2, 
	\[ \Omega(t_1)\star\Omega(t_2)=\cQ\star\big(\Omega(t_2)\sqcup\{\tworow{3^{k'}}{m'+4},\hook{3^{k'}}{m'+4},\tworow{3^{k'}}{m'+1} \}^\circ\big), \]
	which contains $\cB_{3^k+3^{k'}}(m+2+m'+5)$ and no other thin partitions.
	
	\noindent\emph{Step 4.} Finally, we consider $\Omega(t_1)\star\Omega(t_2)$ where $\sigma(t_2)=11$.
	
	Then $\Omega(t_2)$ contains $\cB_{3^{k'}}(m'+2)\setminus\{\tworow{3^{k'}}{m'+1},\hook{3^{k'}}{m'+1} \}^\circ$ and no other thin partitions, for some $k'\ge 3$ and $m'=\tfrac{3^{k'}-1}{2}$. This exactly the same argument as Step 2. (Note if $\lambda=(\tfrac{3^k+3^{k'}}{2},\tfrac{3^k+3^{k'}}{2})$ then we observe that $\tworow{3^{k'}}{\tfrac{3^{k'}+3}{2}}=\tworow{3^{k'}}{m'+2}\in\Omega(t_2)$.)
	
	\noindent\emph{Step 5.} The assertions of the current lemma follows from Steps 1--4 if $T=2$ and $y_{11}\ge 1$, or from Lemma~\ref{lem:102122} and Step 1 if $T\ge 3$ or $y_{11}=0$.
\end{proof}

\begin{lemma}\label{lem: RT=1}
	Suppose $T\ge 1$ and $R\ge 1$. Then $\Omega(\phi\times\theta)\supseteq\cB_{n+u}(N(\phi)+N(\theta))$ and contains no other thin partitions.
\end{lemma}

\begin{proof}
	From Lemma~\ref{lem: type 11}, one of the following statements holds for $\Omega(\theta)$:
	\begin{itemize}
		\item[(i)] $T=1$ and $\sigma(t_1)\in\{10,11,7\}$ (and $\Omega(\theta)=\Omega(t_1)$ is described e.g.~in Table~\ref{table:summary});
		\item[(ii)] $T=2$, $y_{11}=y_{21}=1$, and $\Omega(\theta)$ contains $\cB_u(N(\theta))\setminus\{(\tfrac{u}{2},\tfrac{u}{2} )\}^\circ$ and no other thin partitions; or
		\item[(iii)] $\Omega(\theta)$ contains $\cB_u(N(\theta))$ and no other thin partitions (note in this case $u-N(\theta)\ge 2$).
	\end{itemize}
	From Theorem~\ref{thm:a-prime power} and Proposition~\ref{prop:qt}, one of the following statements holds for $\Omega(\phi)$:
	\begin{itemize}
		\item[(1)] $R=1$ and $\sigma(s_1)\in\{1,5,6\}$ (and $\Omega(\phi)=\Omega(s)$ is recorded e.g.~in Table~\ref{table:summary});
		\item[(2)] $\Omega(\phi)=\cB_n(N(\phi)-1)\sqcup\{(N(\phi),\mu) \mid \mu\in\Omega(\triv_{P_{n-N(\phi)}}) \}^\circ$; or
		\item[(3)] $\Omega(\phi)=\cB_n(N(\phi))$.
	\end{itemize}
	
	Let $\Omega:=\Omega(\phi)\star\Omega(\theta)$ and $N:=N(\phi)+N(\theta)$. We list why $\Omega$ contains $\cB_{n+u}(N)$ in each of the possible cases. That $\Omega\setminus\cB_{n+u}(N)$ contains no thin partitions then follows by the Littlewood--Richardson rule and the pigeonhole principle (as in Lemma~\ref{lem:2122})
	
	\begin{itemize}%\setlength\itemsep{0.5em}
		\item \emph{Case (iii) $\star$ (3):} by Proposition~\ref{prop: B star B}.
		
		\item \emph{Case (iii) $\star$ (2):} by the argument in the proof of Lemma~\ref{lem: 1233031}, since $u-N(\theta)\ge 2$. %whenever (iii) holds.
		
		\item \emph{Case (iii) $\star$ (1):} if $\sigma(s)=1$, then by Lemma~\ref{lem: 10.3}; if $\sigma(s)\in\{5,6\}$, then by the argument in the proof of Lemma~\ref{lem: 56}.
		
		\item \emph{Case (ii) $\star$ (3):} we show that $\Omega=\Omega(\phi)\star ( \Omega(\theta)\sqcup\{(\tfrac{u}{2},\tfrac{u}{2}) \}^\circ)$, from which it follows that $\Omega\supseteq\cB_{n+u}(N)$. Let $\lambda\in\Omega(\phi)\star ( \Omega(\theta)\sqcup\{(\tfrac{u}{2},\tfrac{u}{2}) \}^\circ)$, so $c^\lambda_{\beta\alpha}=c^\lambda_{\alpha\beta}>0$ for some $\alpha\in\cB_u(N(\theta))$ and $\beta\in\Omega(\phi)$.  If $\alpha\notin\{(\tfrac{u}{2},\tfrac{u}{2}) \}^\circ$ then $\lambda\in\Omega$ as desired. Without loss of generality we may assume $\alpha=(\tfrac{u}{2},\tfrac{u}{2})$. 
		%By Lemma~\ref{lem: LR tworow filling} (a), 
		Note $c^\lambda_{\gamma\beta}>0$ for $\gamma=(\alpha_1+1,\alpha_2-1)=(\tfrac{u}{2}+1,\tfrac{u}{2}-1)\in\Omega(\theta)$ unless $[\lambda\setminus\beta]\cong[\alpha]$ (else the leftmost 2 in an LR filling of $[\lambda\setminus\beta]$ of weight $\alpha$ may be replaced by a 1). 
		When this holds, let $\mathsf{c}$ be the top left box of $[\lambda\setminus\beta]$, and $\mathsf{b}$ be any removable box of $\beta$. Defining $\tilde{\beta}$ by $[\tilde{\beta}]=([\beta]\setminus\mathsf{b})\cup\mathsf{c}$, clearly $\cLR([\lambda\setminus\tilde{\beta}])\cap\Omega(\theta)\ne\emptyset$. Moreover, $\tilde{\beta}\in\cB_n(N(\phi))$ as $\mathsf{c}$ cannot be in position $(1,N(\phi)+1)$ or $(N(\phi)+1,1)$, since $N(\phi)>n/2$. Thus $\lambda\in\Omega(\phi)\star\Omega(\theta)=\Omega$.
		
		\item \emph{Case (ii) $\star$ (2):} we first show $\Omega=\Omega(\phi)\star ( \Omega(\theta)\sqcup\{(\tfrac{u}{2},\tfrac{u}{2}) \}^\circ)$, by using the same argument as in the previous case %(ii) $\star$ (3) 
		with one modification. 
		We choose $\mathsf{b}$ to be the box in position $(1,N(\phi))$ or $(N(\phi),1)$ if either exists in $\beta$ ($\beta$ cannot have both as $|\beta|=n<2N(\phi)-1$: because $\Omega(\phi)$ belongs to case (2), therefore $\sigma(\phi)\in E$ and so $N(\phi)\ge\tfrac{n+7}{2}$ in fact). Thus $\tilde{\beta}\in\cB_n(N(\phi)-1)\subseteq\Omega(\phi)$ by inspection, where $[\tilde{\beta}]=([\beta]\setminus\mathsf{b})\cup\mathsf{c}$.
		Then, $\Omega(\phi)\star ( \Omega(\theta)\sqcup\{(\tfrac{u}{2},\tfrac{u}{2}) \}^\circ)$ contains $\cB_{n+u}(N)$ by the same argument as in the case (iii) $\star$ (2).% (i.e.~use Lemma~\ref{lem: 1233031} argument, observing that $u-N(\theta)\ge 2$).
		
		\item \emph{Case (ii) $\star$ (1):} similarly we find that $\Omega=\Omega(\phi)\star ( \Omega(\theta)\sqcup\{(\tfrac{u}{2},\tfrac{u}{2}) \}^\circ)$, and then proceed as in case (iii) $\star$ (1).
		
		\item \emph{Case (i) $\star$ (3):} we replace $\Omega(t_1)$ by $\cQ$ as in the proof of Lemma~\ref{lem: type 11}, then proceed as in case (iii) $\star$ (3).
		
		\item \emph{Case (i) $\star$ (2) (or (1)):} we again apply similar replacements then proceed as in case (iii) $\star$ (2) (or (1), respectively).
	\end{itemize}
\end{proof}

\begin{lemma}\label{lem: all nqt}
	Recall $\theta\in\Lin(P_u)$ and $\vartheta\in\Lin(P_v)$ from Notation~\ref{not:nqt}. We have that $\Omega(\theta\times\vartheta)$ contains $\cB_{u+v}(N(\theta)+N(\vartheta))$ and no other thin partitions.
\end{lemma}

\begin{proof}
	We have the same possibilities (i)--(iii) for $\Omega(\theta)$ as in Lemma~\ref{lem: RT=1}. Moreover, from Table~\ref{table:m<27} we see that $\Omega(\vartheta)$ contains $\cB_v(m(\vartheta))$ (sometimes excepting $(3^3), (6^2), (9^2)$ and conjugates: call these the small exceptions) and no other thin partitions.
	
	When $\Omega(\vartheta)$ contains the whole of $\cB_v(m(\vartheta))$%(without the small exceptions)
	, then $\Omega(\theta\times\vartheta)\supseteq\cB_{u+v}(N(\theta)+N(\vartheta))$ by the same argument as in cases (i)--(iii) $\star$ (3) in the proof of Lemma~\ref{lem: RT=1}. When there are the small exception(s), then the result follows by a similar analysis to case (ii) $\star$ (3) in the proof of Lemma~\ref{lem: RT=1}.
\end{proof}

\begin{lemma}\label{lem: rest}
	Recall $\phi\in\Lin(P_n)$, $\theta\in\Lin(P_u)$, $\psi\in\Lin(P_l)$ and $\vartheta\in\Lin(P_v)$ from Notation~\ref{not:nqt}, where $27\mid n,u\in\N$, $l\in\{0,1,\dotsc,26\}$ and $v\in\{9,10,\dotsc,26\}$. Abbreviating `and no other thin partitions' to `a.n.o.t.p', then
	\begin{enumerate}[(a)]
		\item if $T=1, y_{21}=1$ and $\psi=\triv_{P_3}$, then $\Omega(\theta\times\psi)$ contains $\cB_{u+l}(N(\theta)+N(\psi))\setminus\{(\tfrac{u+l}{2},\tfrac{u+l}{2}) \}^\circ$ a.n.o.t.p. Otherwise, $\Omega(\theta\times\psi)$ contains $\cB_{u+l}(N(\theta)+N(\psi))$ a.n.o.t.p.
		
		\item $\Omega(\phi\times\vartheta)$ contains $\cB_{n+v}(N(\phi)+N(\vartheta))$ a.n.o.t.p.
		
		\item $\Omega(\phi\times\theta\times\vartheta)$ contains $\cB_{n+u+v}(N(\phi)+N(\theta)+N(\vartheta))$ a.n.o.t.p.
		
		\item $\Omega(\phi\times\theta\times\psi)$ contains $\cB_{n+u+l}(N(\phi)+N(\theta)+N(\psi))$ a.n.o.t.p.
		
		\item $\Omega(\theta\times\psi\times\vartheta)$ contains $\cB_{u+l+v}(N(\theta)+N(\psi)+N(\vartheta))$ a.n.o.t.p. 
		
		\item $\Omega(\phi\times\psi\times\vartheta)$ contains $\cB_{n+l+v}(N(\phi)+N(\psi)+N(\vartheta))$ a.n.o.t.p.
		
		\item $\Omega(\phi\times\psi\times\theta\times\vartheta)$ contains $\cB_{n+l+u+v}(N(\phi)+N(\psi)+N(\theta)+N(\vartheta))$ a.n.o.t.p.
	\end{enumerate}
\end{lemma}

\begin{proof}
	We have the same possibilities (i)--(iii) for $\Omega(\theta)$ and (1)--(3) for $\Omega(\phi)$ as in Lemma~\ref{lem: RT=1}. From Table~\ref{table:m<27}, we have that $\Omega(\psi)=\cB_l(m(\psi))$ or $\Omega(\psi)=\cB_l(l-|\nu|)\setminus\{(l-|\nu|,\nu) \}^\circ$ where $\nu\in\{(2,1),(2,2),(3,2,1)\}$, or $\Omega(\psi)$ satisfies one of these two but with some small partitions excluded. Also, $\Omega(\vartheta)$ contains $\cB_v(m(\theta))$ (sometimes with some small partitions excluded) and no other thin partitions. In all instances, that $\Omega(-)$ contains no other thin partitions than those described %in $\cB_{-}(N(-))$ 
	follows from the Littlewood--Richardson rule and the pigeonhole principle.
	
	\noindent\emph{(a)} If $\Omega(\psi)=\cB_l(m(\psi))$ then $\Omega(\theta\times\psi)$ contains $\cB_{u+l}(N(\theta)+N(\psi))$ by the same argument as in cases (i)--(iii) $\star$ (3) in the proof of Lemma~\ref{lem: RT=1}. If $\Omega(\psi)=\cB_l(l-|\nu|)\setminus\{(l-|\nu|,\nu) \}^\circ$, then we follow cases (i)--(iii) $\star$ (2) instead. A similar case analysis follows if $\Omega(\psi)$ has either of these forms but with small exceptions, by modifying skew shapes and considering adjusted LR fillings similar to the proof of Lemma~\ref{lem: RT=1}, since there are only very few possibilities that connected components of such skew shapes can take. The only difference is when $\psi=\triv_{P_3}$ and $\Omega(\psi)=\{(3),(1^3)\}$, and $\theta$ is such that $N(\theta)=\tfrac{u+1}{2}$, in which case $\Omega(\theta\times\psi)$ contains $\cB_{u+l}(N(\theta)+N(\psi))\setminus\{(\tfrac{u+l}{2},\tfrac{u+l}{2}) \}^\circ$. This occurs precisely when $\theta$ is such that $T=1$ and $y_{21}=1$.
	
	\noindent\emph{(b)} If $\Omega(\vartheta)$ contains $\cB_v(m(\vartheta))$, then $\Omega(\phi\times\vartheta)$ contains $\cB_{n+v}(N(\phi)+N(\vartheta))$ and no other thin partitions by the same argument as in cases (iii) $\star$ (1)--(3) in Lemma~\ref{lem: RT=1}. If $\Omega(\vartheta)$ contains $\cB_v(m(\vartheta))$ except for some small partitions, then we instead follow cases (ii) $\star$ (1)--(3).
	
	\noindent\emph{(c)} Observe $\Omega(\phi\times\theta)$ contains $\cB_{n+u}(N(\phi)+N(\theta))$ by Lemma~\ref{lem: RT=1}. Thus $\Omega(\phi\times\theta\times\vartheta)=\Omega(\phi\times\theta)\star\Omega(\vartheta)$ contains $\cB_{n+u+v}(N(\phi)+N(\theta)+N(\vartheta))$ by the same argument as in the proof of Lemma~\ref{lem: all nqt}.
	
	\noindent\emph{(d)--(f)} Let $\pi=\phi\times\theta$ (respectively $\theta\times\vartheta$ or $\phi\times\vartheta$). Then $\Omega(\pi\times\psi)$ contains $\cB_{w'}(N')$ where $N'=N(\pi)+N(\psi)$ and $w'$ is defined appropriately. This follows from the same argument as in (a), since $\Omega(\pi)$ contains $\cB_{w'-l}(N(\pi))$ (by Lemma~\ref{lem: RT=1}, Lemma~\ref{lem: all nqt} or (b) resp.) and since $N(\pi)>\tfrac{w'-l+1}{2}$.
		
	\noindent\emph{(g)} Since $\Omega(\phi\times\theta\times\psi)$ contains $\cB_{n+u+l}(N(\phi)+N(\theta)+N(\psi))$ by (d), then $\Omega(\phi\times\theta\times\psi)\star\Omega(\vartheta)$ contains $\cB_{n+l+u+v}(N(\phi)+N(\psi)+N(\theta)+N(\vartheta))$ by the same argument as (c).
\end{proof}

The results of Theorems~\ref{thm:a} and~\ref{thm:b-primepower} and Lemmas~\ref{lem: RT=1}, \ref{lem: all nqt} and~\ref{lem: rest} can be summarised in the following corollary, from which we can also immediately deduce Theorem~\ref{thm:c} below.

\begin{corollary}\label{cor: theorem B'}
	Let $w\in\N$ and $\Theta\in\Lin(P_w)$, where $w=n+u+l+v$ or a subsum thereof, as in Notation~\ref{not:nqt}. Then $m(\Theta)= N(\Theta)$, unless (at least) one of the following holds:
	\begin{itemize}
		\item $\Theta=\phi(s)$ where $\sigma(s)\in\{1,5,6\}$, in which case $m(s)$ %\ne N(s)$ 
		is given in Theorem~\ref{thm:b-primepower} (or~\ref{thm:a}). % Theorem B$'$ for powers of 3 (or Theorem A$'$).
		\item $\Theta=\phi\times\psi$ and one of the following holds, in which case $m(\Theta)=N(\Theta)-1$:
		\begin{itemize}
			\item[$\circ$] $\sigma(\phi)\in E$ and $\psi=\triv_{P_l}$;
			\item[$\circ$] $\sigma(\phi)\in F$ and $\psi\in\Psi_2$;
			\item[$\circ$] $\sigma(\phi)=(R,0,0,0,0,0,0)$ (i.e.~$\phi=\triv_{P_n}$) and $\psi\in\Psi_1$;
			\item[$\circ$] $\sigma(\phi)=(R-1,0,0,0,1,0,0)$ and $\psi=\phi(1,0)\cdot\triv_{P_{l-9}}$, $l\ge 9$; or
			\item[$\circ$] $\sigma(\phi)=(R-1,1,0,0,0,0,0)$ and $\psi=\phi(1,0)\cdot\triv_{P_{l-9}}$, $l\ge 9$.
		\end{itemize}
		\item $\Theta=\phi(t)$ for $\sigma(t)\in\{10,11,7\}$, in which case $m(t)$ %\ne N(t)$ 
		is given in Theorem~\ref{thm:b-primepower} %Theorem B$'$ for powers of 3 
		(or Table~\ref{table:summary}).
		\item $\Theta=\phi(t_1)\times\phi(t_2)$ where $\sigma(t_1)=11$ and $\sigma(t_2)=21$, in which case $m(\Theta)=\tfrac{w}{2}-1$.
		\item $\Theta=\phi(t)\times\triv_{P_3}$ where $\sigma(t)=21$, in which case $m(\Theta)=\tfrac{w}{2}-1$.
		\item $w<27$, in which case $m(\Theta)$ may be calculated directly; in all such instances, $m(\Theta)\ge\tfrac{2w}{9}$.
	\end{itemize}
\end{corollary}

\begin{proof}[Proof of Theorem~\ref{thm:b-general}]
	This follows from Corollary~\ref{cor: theorem B'}.
\end{proof}

%The next theorem also follows immediately from Corollary~\ref{cor: theorem B'}.
\begin{theorem}\label{thm:c}
	Let $w\in\N$ and $\Omega_w:=\bigcap_{\Theta\in\Lin(P_w)}\Omega(\Theta)$. Then $\cB_w(\tfrac{2w}{9})\subseteq\Omega_w$.
	In particular, $\underset{w\to\infty}{\lim} \frac{|\Omega_w|}{|\cP(w)|} = 1$.
\end{theorem}

%\begin{proof}
%	By Corollary~\ref{cor: theorem B'}, $\cB_w(w\cdot c_3)\subseteq\Omega(\Theta)$ for all $\Theta\in\Lin(P_w)$, where $c_3=\tfrac{2}{9}$. The limit then follows since by \cite[(1.4)]{EL}, if $f(n)$ is any function such
%	that $f(n)\to\infty$ as $n\to\infty$, then for all but $o(|\cP(n)|)$ partitions $\lambda$ of $n$, $\lambda_1$ and $l(\lambda)$ lie between $\sqrt{n}\cdot(\tfrac{\log n}{c}\pm f(n))$ where $c$ is a constant.
%\end{proof}

\begin{remark}\label{rem:2/9}
	We remark that the value of the constant $\tfrac{2}{9}$ in Theorem~\ref{thm:c} was limited by $m(s)=2$ for $s=(0,1)$ and $s=(1,1)$ when $w=9$. In general, $m(\Theta)$ is much closer to $\tfrac{w}{2}$, resembling $\cB_w(\tfrac{w}{2})\subseteq\Omega_w$ for all primes $p\ge 5$ \cite[Theorem C]{GL2}. Indeed, when $p=3$, then for all sufficiently small $\epsilon>0$ and sufficiently large $w$, our main results %Theorem~\ref{thm:a}
	imply that $\cB_w(\tfrac{w}{2+\epsilon})\subseteq\Omega(\Theta)$ for all $\Theta\in\Lin(P_w)$ (recall $N(\Theta)$ may be calculated using Tables~\ref{table:m<27}, \ref{table:summary} and Definition~\ref{def: all p=3}).\hfill$\lozenge$
\end{remark}

\subsection*{Acknowledgements}
The author is grateful to the anonymous referee for their many detailed and helpful comments and suggestions. The author would also like to thank Karin Erdmann for her support and hospitality and the University of Oxford for hosting a London Mathematical Society Early Career Fellowship during which part of this work was done. 
\bigskip

%%%%%%%%%%%%%%%%%%%%%%%%%%%%%%%%%%%%%%%%%%%%%%%%%%%%%%%%%%%%%%%%%%%%%%%%%

\end{document}